\documentclass[12pt,reqno]{amsart}

\newcommand\version{October 21, 2026}

\usepackage{amsmath,amsfonts,amsthm,amssymb,amsxtra}
\usepackage{hyperref}
\usepackage{bbm} 
\usepackage{bm} 
\usepackage{accents} 

\newtheorem{theorem}{Theorem}[section]
\newtheorem{proposition}[theorem]{Proposition}
\newtheorem{lemma}[theorem]{Lemma}
\newtheorem{corollary}[theorem]{Corollary}

\newtheorem*{problem*}{Problem}

\theoremstyle{definition}

\newtheorem{definition}[theorem]{Definition}

\newtheorem{assumption}[theorem]{Assumption}

\theoremstyle{remark}

\newtheorem{remark}[theorem]{Remark}

\numberwithin{equation}{section}

\newcommand{\C}{\mathbb{C}}

\newcommand{\D}{\mathcal{D}}

\renewcommand{\epsilon}{\varepsilon}

\newcommand{\F}{\mathcal{F}}

\newcommand{\ck}{\mathcal{K}}

\newcommand{\N}{\mathbb{N}}

\renewcommand{\phi}{\varphi}
\newcommand{\R}{\mathbb{R}}

\newcommand{\bs}{\mathbb{S}}

\DeclareMathOperator{\im}{Im}

\DeclareMathOperator{\ran}{ran}
\DeclareMathOperator{\re}{Re}
\DeclareMathOperator*{\slim}{s-lim}

\DeclareMathOperator{\supp}{supp}
\DeclareMathOperator{\sgn}{sgn}

\def\bs{\mathbb{S}}

\def\cb{\mathcal{B}}
\def\ce{\mathcal{E}}
\def\ch{\mathcal{H}}

\def\Rd{{\mathbb{R}^d}}

\newcommand{\me}[1]{\mathrm{e}^{#1}}
\newcommand{\one}{\mathbf{1}}
\newcommand{\f}{\frac}
\newcommand{\La}{L_\kappa}

\begin{document}

\title[Wave operators for Hardy operators --- \version]{On wave operators for Hardy operators \\ with and without perturbations}

\author[A. Ishida]{Atsuhide Ishida}
\address[Atsuhide Ishida]{Katsushika Division, Institute of Arts and Sciences, Tokyo University of Science, 6-3-1 Niijuku, Katsushika-ku, Tokyo 125-8585, Japan}
\email{aishida@rs.tus.ac.jp}

\author[K. Merz]{Konstantin Merz}
\address[Konstantin Merz]{Institute for Theoretical Physics, ETH Zurich, Wolfgang-Pauli-Strasse 27, 8093 Zurich, Switzerland, and Institut f\"ur Analysis und Algebra, Technische Universit\"at Braunschweig, Universit\"atsplatz 2, 38106 Braun\-schweig, Germany}
\email{k.merz@tu-bs.de, k.merz@phys.ethz.ch}



\date{\version}

\begin{abstract}
  We prove asymptotic completeness of the wave operators in time-de\-pen\-dent scattering theory which intertwine the ordinary or fractional Laplacian and the corresponding Hardy operator, i.e., the fractional Laplacian with an added subcritical or critical Hardy potential, whenever the Hardy potential is short-range.
  Moreover, we prove asymptotic completeness of the wave operators intertwining Hardy operators with and without external short-range perturbations. We also show the nonexistence of the considered wave operators when the Hardy or external potentials are long-range.
  We arrive at these results using three different approaches.
  The first approach uses an abstract semigroup framework for asymptotic completeness of wave operators for Fourier multipliers with short-range potentials. The second approach uses Kato smoothness theory. In the third approach, we develop and apply new propagation estimates for Hardy operators with and without perturbations.
\end{abstract}

\maketitle
\tableofcontents


\section{Introduction and main results}

Let $d\in\N=\{1,2,...\}$ and $\alpha\in(0,d)\cap(0,2]$. Our principal goal in this paper is to advance the scattering theory of
\begin{align}
  L_\kappa = (-\Delta)^{\alpha/2}-\frac{\kappa}{|x|^\alpha} \quad \text{in} \ L^2(\R^d).
\end{align}
This operator is well-defined as a nonnegative, selfadjoint operator for all $\kappa\leq\kappa_{\rm c}$ with
\begin{align}
  \kappa_{\rm c} = \kappa_{\rm c}(d,\alpha) = \frac{2^\alpha\Gamma\left(\frac{d+\alpha}{4}\right)^2}{\Gamma\left(\frac{d-\alpha}{4}\right)^2}
\end{align}
being the sharp constant in the Hardy--Kato--Herbst inequality \cite{Kato1966,Herbst1977}, or Hardy inequality for short, which states
\begin{align}
  L_0 \geq \kappa_{\rm c}|x|^{-\alpha}
\end{align}
in the sense of quadratic forms on $L^2(\R^d)$; see also \cite{Kovalenkoetal1981,Yafaev1999,Franketal2008H,FrankSeiringer2008,Bogdanetal2016} for different proofs. We call $\kappa/|x|^\alpha$ the Hardy potential and $L_\kappa$ the Hardy operator.
The scattering theory of $L_\kappa$ for $\alpha=2$ is well understood \cite{Boutetetal1996,DerezinskiRichard2017,Mizutani2020S,Miaoetal2023}; see also the overview at the end of this introduction. While our primary interest concerns the nonlocal case $\alpha<2$, our methods also allow us to recover and extend existing results for $\alpha=2$.

\smallskip
The operator $L_\kappa$ is of great interest in both pure and applied mathematics. Indeed, since $L_\kappa$ is homogeneous under dilations, it provides an important test ground in harmonic analysis to investigate how much of Euclidean Fourier analysis remains valid in the absence of translation invariance. Moreover, $L_\kappa$ frequently arises as a scaling limit of more complicated operators, and one aims to analyze it in as much detail as possible to draw conclusions about the original operators from which $L_\kappa$ arises. For instance, for $\alpha=1$, $L_\kappa$ occurs in mathematical physics, where it provides a relativistic description of electrons in heavy atoms and stars, see, e.g., \cite{Sommerfeld1915,Franketal2008,Franketal2020P,Chandrasekhar1931,LiebThirring1984,LiebYau1987,LiebYau1988,FrohlichLenzmann2007D}.

\smallskip
In the next subsection, we recall the definition and some spectral properties of $L_\kappa$. In the ensuing subsections, we state the main results of the paper, first informally, then precisely, put them into context and discuss previous results, and give an outlook. We close the introduction with the organization of the rest of the paper.


\subsection{Definition and spectral properties of $L_\kappa$}

For $\kappa\leq\kappa_{\rm c}$, the operator $L_\kappa$ is defined as the Friedrichs extension of the corresponding expression on $C_c^\infty(\R^d)$. Its form domain is the inhomogeneous Sobolev space $H^{\alpha/2}(\R^d)$ if $\kappa<\kappa_{\rm c}$, while for $\kappa=\kappa_{\rm c}$ it is the closure of $C_c^\infty(\R^d)$ with respect to the norm $(\langle u,L_\kappa u\rangle + \|u\|_2^2)^{1/2}$. Here, $\langle\cdot,\cdot\rangle$ denotes the usual scalar product in $L^2(\R^d)$, which is linear in the second argument, and $\|\phi\|_2^2:=\langle\phi,\phi\rangle$. If there is no danger of confusion, we abbreviate $\|\phi\|_2$ by $\|\phi\|$.

For $\kappa<\kappa_{\rm c}$, the spectrum of $L_\kappa$ in $[0,\infty)$ is purely absolutely continuous, see \cite[Theorem~2.3]{Herbst1977}, \cite[Remark~1.4]{MizutaniYao2021}, and the remark after Lemma~\ref{lem5}. Moreover, zero is not an eigenvalue for all $\kappa\leq\kappa_{\rm c}$, which can be read off from the ground state representation \cite[Proposition~4.1]{Franketal2008H}, \cite[Proposition~5]{Bogdanetal2016}; see also \cite{Yafaev1999,BogdanMerz2024} for different proofs, which work in every separate angular momentum channel. Although zero is not an eigenvalue, $L_\kappa$ has infinitely many generalized ground states, i.e., functions $h\in L_{\rm loc}^2(\R^d)$ which do not belong to $L^2(\R^d)$ but for which $L_\kappa h=0$ in the distributional sense. For instance, $L_\kappa|x|^{-\sigma}=0$, where $\sigma=\sigma_\kappa\in(-M,(d-\alpha)/2]$ with $M=\alpha$ if $\alpha<2$ and $M=\infty$ if $\alpha=2$, is the unique solution to
\begin{align}
  \label{eq:defsigma}
  \kappa = \Phi(\sigma) := \frac{2^\alpha\Gamma\left(\frac{\sigma+\alpha}{2}\right) \Gamma\left(\frac{d-\sigma}{2}\right)}{\Gamma\left(\frac\sigma2\right)\Gamma\left(\frac{d-\alpha-\sigma}{2}\right)}, \quad \sigma\in\left(-M,\frac{d-\alpha}{2}\right],
\end{align}
see, e.g., \cite[(3.5)]{Franketal2008H} or \cite[Proposition~5]{Bogdanetal2016}. In particular, $\Phi(\sigma)$ is strictly monotonically increasing for $\sigma\in(-M,(d-\alpha)/2]$, and $\Phi(\sigma)$ satisfies $\lim_{\sigma\to-M}\Phi(\sigma)=-\infty$, $\Phi(0)=0$, and $\Phi((d-\alpha)/2)=\kappa_{\rm c}$. Note also that $L_{\kappa_{\rm c}}$ has a so-called virtual level or zero-energy resonance, i.e., $L_{\kappa_{\rm c}}-V$ has an eigenvalue for any $0\not\equiv V \geq0$ satisfying certain integrability conditions, cf.~\cite{EkholmFrank2006} for $\alpha=2$.


\subsection{Summary of main results and new tools}

With the most important notation being settled, let us give an overview over our main results and new tools that may be of independent interest.

\subsubsection*{Main results}

Our main results show asymptotic completeness, i.e., existence and completeness of (M\o ller) wave operators intertwining, possibly perturbed, Hardy operators with fractional Laplacians or Hardy operators for various parameter regimes, and absence of singular continuous spectrum of the perturbed operators. In the following, we introduce this terminology precisely.

For two selfadjoint operators $H$ and $H_0$ on a Hilbert space $\mathcal H$, the wave operators are defined as the strong limits
\begin{align}
  W_\pm(H,H_0) := \slim_{t\to\pm\infty}\me{itH}\me{-itH_0}P_{\rm ac}(H_0)
\end{align}
provided these limits exist. Here, $P_{\rm ac}(H_0)$ denotes the spectral projection on the absolutely continuous spectral subspace of $H_0$, denoted by $\mathcal H_{\rm ac}(H_0)$. As the analysis of $W_-(H,H_0)$ is under our assumptions the same as that of $W_+(H,H_0)$, we will from now on only consider the latter and denote it by $W(H,H_0)$. If $W(H,H_0)$ exists, then $\ran(W(H,H_0))$ is a closed subspace of $\mathcal H_{\rm ac}(H)$, we have the intertwining property $\phi(H)W(H,H_0)=W(H,H_0)\phi(H_0)$ for any bounded Borel function $\phi$, and $\|W(H,H_0)u\|=\|P_{\rm ac}(H_0)u\|$ for all $u\in\ch$, i.e., $W(H,H_0)$ is a partial isometry with initial subspace $\ch_{\rm ac}(H_0)$.
If $W(H,H_0)$ exists and $\ran(W(H,H_0))=\ch_{\rm ac}(H)$, it is called complete. Thus, if $W(H,H_0)$ exists and is complete, then for every initial datum $u\in\ch_{\rm ac}(H)$ there is an initial datum $u_+\in\ch_{\rm ac}(H_0)$ such that $\me{-itH}u-\me{-itH_0}u_+\to0$ as $t\to\infty$. Existence and completeness of $W(H,H_0)$ is equivalent to the simultaneous existence of $W(H,H_0)$ and $W(H_0,H)$; we recall the argument at the end of Section~\ref{s:asympcomplperturbedwosmallalpha_a}.
We say that $W(H,H_0)$ is {\em asymptotically complete} if in addition the singular continuous spectral subspace satisfies $\ch_{\rm sc}(H)=\{0\}$. We note that some literature uses ``asymptotic completeness'' already for $\ran(W(H,H_0))=\ch_{\rm ac}(H)$, e.g., \cite{Isozaki2023}, while our stronger convention of asymptotic completeness follows \cite{ReedSimon1978,Teschl2014}; it is also sometimes called {\em strong asymptotic completeness}, e.g., in \cite{AmreinGeorgescu1974}, \cite[Remark~4.4]{Kuroda1983}, or \cite[Sec.~5.5]{Amrein2009}.

Our main results concerning asymptotic completeness can then be listed as follows:
\begin{enumerate}
\item[(a)] If $\alpha\in(1,2]\cap(1,d)$ and $\kappa\leq\kappa_{\rm c}$, then $W(L_\kappa,L_0)$ is asymptotically complete. (Theorem~\ref{asympcomplhardy}(1))

\item[(b)] If $\alpha\in(1,2]\cap(1,d)$, $\kappa\leq\kappa_{\rm c}$, and $V^S$ is short-range, i.e., decays at least like $|x|^{-\rho_S}$ for some $\rho_S>1$, then $W(L_\kappa-V^S,L_0)$ and $W(L_\kappa-V^S,L_\kappa)$ are asymptotically complete. (Theorem~\ref{asympcomplhardyplusperturbation}(\ref{asympcomplhardyplusperturbation:acshortrange}))

\item[(c)] The previous two items even hold when $\kappa/|x|^\alpha$ is replaced with $V_\kappa$ obeying $\tilde\kappa/|x|^\alpha \leq V_\kappa(x) \leq \kappa/|x|^\alpha$ with $\tilde\kappa\leq\kappa\leq\kappa_{\rm c}$.

\item[(d)] If $1/2\leq\alpha\leq1$, $\kappa<\kappa_{\rm c}$, and $\kappa/|x|^\alpha+V^S$ obeys certain assumptions from Mourre theory for $L_0$ with added scaling-critical perturbations (to be made precise in Assumption~\ref{mizutaniyaoconditions} below), then $W(L_\kappa-V^S,L_\kappa)$ is asymptotically complete if $V^S$ decays like $|x|^{-\rho_S}$ with $\rho_S>2\alpha$; this result continues to hold for $1/3<\alpha<1/2$ if $\sigma=\Phi^{-1}(\kappa)<(d+\alpha-1)/2$ and $\rho_S>1$. Sufficient conditions on $V^S$ are that $V^S$ is $C^2$, small in some quantitative sense, and that $|\partial_x^\beta V^S(x)|$ decays like $|x|^{-\rho_S-|\beta|}$ with $\rho_S\geq\alpha$ for all multi-indices $\beta$ with $|\beta|\leq2$. (Theorem~\ref{asympcomplhardyplusperturbation}(\ref{asympcomplhardyplusperturbation:acshortrangesmallalpha}), Lemma~\ref{lem:MY-sufficient})
  
\item[(e)] For smooth Fourier multipliers $C_0(-\Delta)$ with discrete sets of critical points, and which have the same large-momentum behavior as $L_0$, the wave operators $W(C_\kappa,C_0)$, $W(C_\kappa-V^S,C_0)$, and $W(C_\kappa-V^S,C_\kappa)$ with $C_\kappa=C_0-\kappa/|x|^\alpha$ are asymptotically complete under the same assumptions as in (a) and (b); the same generalization as in (c) is possible, too. An important example is $C_0=(1-\Delta)^{\alpha/2}-1$, which is the pseudo-relativistic kinetic energy of a particle with unit mass when $\alpha=1$. (Corollary~\ref{cor:massiverelativisticexistence})

\end{enumerate}

For $\alpha=2$, the statement in (a) was proved in \cite{Boutetetal1996,Mizutani2020S} for $\kappa<\kappa_{\rm c}$. Using tools that are insensitive to virtual levels, we are able to treat also $\kappa=\kappa_{\rm c}$.

In some cases we are able to show that the above results are sharp by proving the nonexistence of wave operators for long-range perturbations.

\begin{enumerate}
\item[(f)] If $\alpha\leq1$, $\kappa\neq0$, and $\kappa\leq\kappa_{\rm c}$, then $W(L_\kappa,L_0)$ does not exist. (Theorem~\ref{asympcomplhardy}(2))

\item[(g)] If $V^L$ is long-range, i.e., it decays like $\pm |x|^{-\rho_L}$ for some $0<\rho_L\leq1$ with $\rho_L\neq\alpha$, then $W(L_\kappa-V^L,L_0)$ does not exist. If $\rho_L=\alpha$, we further assume that $|V^L+\kappa/|x|^\alpha|$ is bounded from below by a positive multiple of $|x|^{-\alpha}$ at large distances. Moreover, $W(L_\kappa-V^L,L_\kappa)$ does not exist if $0<\rho_L\leq1$ and $\alpha>1$. (Theorem~\ref{asympcomplhardyplusperturbationnonexistence}(\ref{asympcomplhardyplusperturbation:nonexistencefree}--\ref{asympcomplhardyplusperturbation:nonexistencerelativelargealpha}))

\item[(h)] If $1/2<\alpha\leq1$, $\kappa<\kappa_{\rm c}$, $0<\rho_L\leq\alpha$, and $V^L$ is smooth with $|\partial_x^\beta V^L|$ decaying like $|x|^{-\rho_L-|\beta|}$ for all multi-indices $\beta$, then $W(L_\kappa-V^L,L_\kappa)$ does not exist. If $d\geq2$, this result extends to $\alpha=1/2$. (Theorem~\ref{asympcomplhardyplusperturbationnonexistence}(\ref{asympcomplhardyplusperturbation:nonexistencelongrangesmallalpha}))

\item[(i)] The nonexistence results in (f)--(g) also hold when $L_0$ is replaced by more general Fourier multipliers $C_0$ discussed in (e). (Corollary~\ref{cor:massiverelativisticnonexistence})
\end{enumerate}

In these cases, the reference dynamics needs to be modified in order to construct suitable analogs of the standard wave operators. We will not discuss such long-range scattering here.

\subsubsection*{Used frameworks and new tools}

Our proofs of asymptotic completeness use three approaches.
For $\alpha>1$ we use short-range scattering theory for pairs of operators $(H_0,H)$ whose heat kernel difference is compact and obeys a certain weighted estimate. This framework has been established in \cite{DemuthSinha1999} for $H_0=-\Delta$ and in \cite{Demuthetal2001,Giere2004} for $H_0$ being a spherically symmetric, smooth Fourier multiplier with a discrete set of critical points. Thus, to prove (a)--(c) above, we use sharp heat kernel bounds for $\me{-t L_\kappa}$ by \cite{Bogdanetal2019,JakubowskiWang2020,Choetal2020}. These estimates factorize into $\me{-tL_0}$ and the locally square-integrable ground state $|x|^{-\sigma}$ for all $\kappa\leq\kappa_{\rm c}$, which is the reason why we can treat in particular the critical coupling constant.
Regarding asymptotic completeness of $W(L_\kappa-V^S,L_\kappa)$ for $\alpha\leq1$, one approach builds on and further develops Kato's smooth method with Kato--Yajima-type estimates for $L_\kappa$ in \cite{MizutaniYao2021}. Moreover, in this case, we avoid the small-momentum singularity of the group velocity $\alpha|D|^{\alpha-1}$ by replacing certain powers of $L_0$ with powers of $L_\kappa$, using that both operators generate equivalent homogeneous Sobolev norms for certain parameter ranges of $d,\alpha,\kappa$ \cite{Killipetal2018,Franketal2021}.
The second approach to asymptotic completeness for $\alpha\leq1$ relies on the Cook--Kuroda method. 

\smallskip
We also contribute new methodological infrastructure which may be of independent interest. Most importantly, we prove new propagation estimates for $L_\kappa$ and $L_\kappa-V$, which are important to run the Cook--Kuroda method and to show the nonexistence results.
Propagation estimates are classical tools in time-dependent scattering theory, see, e.g., the monographs \cite{DerezinskiGerard1997,Isozaki2023} and the references therein for $\alpha=2$. We discuss previous propagation estimates for fractional Schr\"odinger operators in Subsection~\ref{ss:previousresults} below; to our knowledge these do not incorporate scaling-critical potentials like $|x|^{-\alpha}$.
Our estimates will likely also be useful to develop a corresponding scattering theory for $N$-particle Hardy operators. In the nonrelativistic setting, $N$-particle scattering theory has been established in the foundational works \cite{SigalSoffer1987,Graf1990,Derezinski1993}.

To state our results, we denote the characteristic function of a set $\Omega$ by
\begin{align}
  \label{eq:defF}
  x \mapsto
  \one_\Omega(x) :=
  \begin{cases}
    1 & \quad \text{for} \ x \in \Omega, \\
    0 & \quad \text{for} \ x \notin \Omega.
  \end{cases}
\end{align}
For a selfadjoint operator, we use the same notation for the corresponding spectral projection.
For $x,\theta\in\R$, we sometimes write $\one(x\leq\theta)$ or $\one_{x\leq\theta}$ to mean $\one_{(-\infty,\theta]}(x)$. Let us also write $D:=-i\nabla$ and denote the generator of dilations by
\begin{align}
  A := \frac12\left(x\cdot D + D\cdot x\right).
\end{align}
The latter is a natural and important escape observable in scattering theory.
Let us recall a short motivation, in the spirit of \cite{Hunzikeretal1999}, for studying its time evolution. Note $i[H,|x|^2/2]=2A$ if $H=-\Delta-V$. Hence, a lower bound on $i[H,A]$ implies the lower bound $\langle\me{-itH}\psi, |x|^2 \me{-itH}\psi\rangle \gtrsim t^2$ as $t\to\infty$ when $\psi$ is an element of the continuous spectral subspace of $H$, i.e., (ballistic) escape from bounded regions.
While $A$ cannot be written as a commutator of $|D|^\alpha-V$ with a power of $|x|$ when $\alpha<2$, we still consider it as a relevant propagation observable in view of
$$
\frac{d}{dt}\left(\me{it(L_\kappa-V)} A \me{-it(L_\kappa-V)}\right)\Big|_{t=0}
= i[L_\kappa-V,A] = \alpha L_\kappa + x\cdot\nabla V.
$$
Since $L_\kappa \sim |D|^\alpha$ for $\kappa<\kappa_{\rm c}$ by Hardy's inequality, the time evolution of $A$ under the flow generated by $L_\kappa-V$ and that generated by $L_0$ are expected to be comparable when $V$ is not too rough.

Let us now list our new propagation estimates, the first two of which confirm our expectation. In all of the following estimates, the constants also depend on $d,\alpha,\kappa$. Since these parameters are fixed, we will generally not track their dependence explicitly.
\begin{enumerate}
\item[(a)] For all $\alpha\in(0,2]\cap(0,d)$ and $\kappa<\kappa_{\rm c}$, we show
  $$
  \left\| \one\left(\frac{A}{t}\leq\theta\right)f(H_\kappa) \me{-i t H_\kappa}\phi\right\|
  \leq C_{\rho,\theta,\nu,f} t^{-\nu}\|(1+A^2)^{\nu/2}\phi\|
  $$
  where $H_\kappa=L_\kappa-V$ and $V$ decays like $|x|^{-\rho}$ for some $\rho>0$, and $f\in C_c^\infty(0,\infty)$ with $\supp f \cap \sigma_{\rm p}(H_\kappa) = \emptyset$. In view of the above discussion of the time evolution of $A$, we call it a minimal escape velocity estimate.
  Indeed, the estimate shows that (after energy localization) the part of the state contained in $A/t\leq\theta$ decays rapidly. Thus, the state is rather located in that part of phase space, where free, outgoing particles are. The proof is an implementation of an abstract, robust result in \cite{Skibsted1991}.
  (Theorem~\ref{the6})
  
\item[(b)] For all $\alpha\in(1/2,2]\cap(1/2,d)$ and $\kappa<\kappa_{\rm c}$, we show outgoing phase space propagation estimates of the form
  $$
  \left\| f_1(K_1)\me{itK_1} V^S \chi\left(\frac At\right)^2 \me{-itK_0}f_0(K_0) \right\|
  \leq C_{\rho,\mu,f_0,f_1} t^{-\mu}
  $$
  with a ``switch function'' $0\leq\chi(r)\leq1$ with $\chi(r)=0$ if $r\leq1/4$ and $\chi(r)=1$ for $r\geq1/2$, where $f_0,f_1$ are bump functions, and where $K_0,K_1\in\{L_\kappa,L_\kappa-V^S\}$. The cutoff $\chi(A/t)$ selects that region in phase space where states escape like free outgoing particles.
  The $t^{-\mu}$-bound shows that when a state is restricted to this phase-space region, its overlap with the potential $V^S$ rapidly decays, i.e., the particle escapes the effective range of the potential in phase space. This estimate will be crucial in the application of the Cook--Kuroda method to establish asymptotic completeness.
  (Lemma~\ref{the11suff-unified})
  
\item[(c)] For $\alpha\in[1,2]\cap[1,d)$ and $\kappa<\kappa_{\rm c}$, we show weighted propagation estimates of the form
  $$
  \|(1+x^2)^{\mu/2} f(H_\kappa) \me{-i t H_\kappa} (1+ x^2)^{-\mu/2}\|
  \leq C_{\mu,f,\rho} (1+t^2)^{\mu/2}, \quad t\in\R
  $$
  for some $\mu\geq0$, where $H_\kappa=L_\kappa-V$ and $V$ decays like $|x|^{-\rho}$ for some $\rho$. Such bounds say that an initially spatially localized state develops spatial moments no faster than the ballistic scale. When $V=0$, we can also treat $\alpha>1/2$ and $\kappa<\kappa_{\rm c}$, as well as $\alpha=1/2$ when $d\geq2$.
  These estimates are used to prove maximal velocity estimates, which we prove in \cite{IshidaMerz2026M}.
  (Lemmas~\ref{lem9} and~\ref{lem9-small-alpha})
\end{enumerate}

We use (a) and (b) to give another proof of asymptotic completeness of $W(L_\kappa-V^S,L_\kappa)$ for $1/2<\alpha\leq1$, however for a slightly smaller class of potentials. The tools in (a) and (c) are the key to showing the above nonexistence result (h) for $\alpha\leq1$.
To prove the nonexistence results for $\alpha>1$, we develop a new nonexistence criterion for wave operators intertwining Fourier multipliers with their perturbed versions, which in particular works for singular perturbations. The proof of this criterion is inspired by \cite{IshidaWada2020} where a special class of long-range perturbations was considered.
The proofs of the propagation estimates (b)--(c) rely on commutator expansions and Helffer--Sj\"ostrand calculus for functions that do not decay at infinity, while (a) is an implementation of an abstract result in \cite{Skibsted1991}.

\smallskip
In the following subsections, we state our main results and tools precisely.


\subsection{Scattering theory for Hardy operators}

We first consider
$$
W(L_\kappa,L_0) = \slim_{t\to\infty}\me{i t L_\kappa}\me{- i t L_0},
$$
where we recall $\ch_{\rm ac}(L_0) = L^2(\R^d)$.
In fact, our positive result on its asymptotic completeness holds for a slightly larger class of potentials, namely those that are bounded from above and below by (sub-)critical multiples of $|x|^\alpha$.

\begin{theorem}
  \label{asympcomplhardy}
  Let $d\in\N$, $0<\alpha\in(0,d)\cap(0,2]$, and $\kappa\leq\kappa_{\rm c}$.
  \begin{enumerate}
  \item Let $\alpha>1$, $\tilde\kappa\leq\kappa$, and $V_\kappa:\R^d\to\R$ be such that
    $$
    \frac{\tilde\kappa}{|x|^\alpha} \leq V_\kappa(x) \leq \frac{\kappa}{|x|^\alpha}, \quad x\in \R^d.
    $$
    Then the wave operator $W(L_0-V_\kappa,L_0)$ is asymptotically complete.

  \item Let $\alpha\leq1$ and $\kappa\neq0$. Then $W(L_\kappa,L_0)$ does not exist.
  \end{enumerate}
\end{theorem}


\subsection{Scattering theory for perturbed Hardy operators}

We now study the scattering theory for $L_\kappa$ perturbed by real-valued short- or long-range potentials. In the following, let $V^S,V^L \in L^\infty(\R^d;\R)$ be such that
\begin{align}
  \label{eq:basicassumpshortrange}
  |V^S(x)| \leq C\langle x\rangle^{-\rho_S}
\end{align}
for some $C>0$, $\rho_S>0$, and
\begin{subequations}
  \label{eq:basicassumplongrange}
  \begin{align}
    \label{ass1}
    K_1\langle x\rangle^{-\rho_L} \leq V^L(x) \leq K_2 \langle x\rangle^{-\rho_L} \quad \text{or} \\
    \label{ass2}
    K_1\langle x\rangle^{-\rho_L} \leq -V^L(x) \leq K_2 \langle x\rangle^{-\rho_L}
    \end{align}
\end{subequations}
for some $0<K_1 \leq K_2$ and $\rho_L>0$. Often, but not always, we will assume $\rho_S>1$ and $\rho_L\leq1$.
Here, we write $\langle x\rangle:=(1+|x|^2)^{1/2}$ for $x\in\R^d$ or a selfadjoint operator. Since all potentials considered here are real-valued, we just write $V\in L^\infty(\R^d)$ when there is no danger of confusion.
Some of our statements require the following smoothness assumptions for the short- or long-range parts,
\begin{align}
  \label{ass5}
  |\partial_x^\beta V(x)| \lesssim \langle x\rangle^{-\rho-|\beta|}, \quad \beta\in\N_0^d.
\end{align}
Here and below, we write $A\lesssim B$ for functions $A,B\geq0$ to indicate that there is a constant $c$ such that $A\leq c B$. If $c$ depends on a parameter $\tau$, we may write $A\lesssim_\tau B$, but the dependence on $d,\alpha,\kappa$ is usually ignored in the notation. The notation $A\sim B$ means $A\lesssim B\lesssim A$. Furthermore, $\|A\|$ denotes the operator norm of a linear operator $A$ on a Hilbert space.

Studying the case $\alpha\leq1$ also involves the following assumptions arising in the Kato smoothness theory of nonlocal operators with Hardy-type potentials in \cite{MizutaniYao2021}.

\begin{assumption}[{\cite[Assumption~A]{MizutaniYao2021}}]
  \label{mizutaniyaoconditions}
  The function $V:\R^d\to\R$ is such that for all $j=0,1,2$, we have $V_j(x):=(-x\cdot\nabla)^j V \in L_{\rm loc}^1(\R^d)$ and the estimates
  \begin{align}
    \label{eq:mizutaniyaocond1}
    \| |(x\cdot\nabla)^j V|^{1/2}(L_0+1)^{-1/2}\| & \lesssim 1, \quad j=0,1,2, \\
    \label{eq:mizutaniyaocond2}
    \langle u,(L_0-V)u\rangle & \gtrsim \langle u,L_0 u\rangle, \\
    \label{eq:mizutaniyaocond3}
    \langle u,(\alpha L_0-V_1)u\rangle & \gtrsim \langle u,L_0 u\rangle, \\
    \label{eq:mizutaniyaocond4}
    |\langle u,(\alpha^2 L_0-V_2)u\rangle| & \lesssim \langle u,(\alpha L_0 - V_1) u\rangle
  \end{align}
  hold for all $u\in C_c^\infty(\R^d)$.
\end{assumption}

As shown in \cite[Example~1.2]{MizutaniYao2021}, $\kappa|x|^{-\alpha}$ with $\kappa<\kappa_{\rm c}$ or $\Psi(x/|x|)|x|^{-\alpha}$ with $\Psi\in L^\infty(\bs^{d-1})$ and $\sup\Psi<\kappa_{\rm c}$ satisfy Assumption~\ref{mizutaniyaoconditions}. In Lemma~\ref{lem:MY-sufficient}, we discuss further examples.

We are now ready to state our results for the scattering theory of $L_\kappa$ perturbed by short-range or long-range potentials. In the following theorem, we collect results concerning asymptotic completeness. In Theorem~\ref{asympcomplhardyplusperturbationnonexistence} below, we state nonexistence results.

\begin{theorem}
  \label{asympcomplhardyplusperturbation}
  Let $d\in\N$, $\alpha\in(0,d)\cap(0,2]$, $\kappa\leq\kappa_{\rm c}$, and $\sigma=\Phi^{-1}(\kappa)\in(-M,(d-\alpha)/2]$.
  Then, the following statements hold.
  \begin{enumerate}
  \item \label{asympcomplhardyplusperturbation:acshortrange}
    Assume in addition $\alpha>1$. Let $\tilde\kappa\leq\kappa$ and $V_\kappa:\R^d\to\R$ be such that
    $$
    \frac{\tilde\kappa}{|x|^\alpha} \leq V_\kappa(x) \leq \frac{\kappa}{|x|^\alpha}, \quad x\in \R^d.
    $$
    Assume also $V^S$ satisfies \eqref{eq:basicassumpshortrange} for some $\rho_S>1$.
    Then $W(L_0 - V_\kappa - V^S, L_0)$ and $W(L_0 - V_\kappa - V^S,L_0-V_\kappa)$ are asymptotically complete.

  \item \label{asympcomplhardyplusperturbation:acshortrangesmallalpha}
    Assume in addition $\alpha\leq1$ and one of the following assumptions:
    \begin{enumerate}
    \item[(a)] Assume $d\geq2$, $\alpha>1/3$, $\rho_S>\max\{1,2\alpha\}$.
      Assume $V^S$ satisfies \eqref{eq:basicassumpshortrange}, that $\kappa/|x|^\alpha+V^S(x)$ obeys Assumption~\ref{mizutaniyaoconditions}, and that there are $\tilde\kappa\leq\kappa'<\kappa_{\rm c}$ such that
      \begin{align}
        \label{ass6}
        \frac{\tilde\kappa}{|x|^\alpha}
        \leq \frac{\kappa}{|x|^\alpha} + V^S(x)
        \leq \frac{\kappa'}{|x|^\alpha}.
      \end{align}
      Let $\sigma'=\Phi^{-1}(\kappa')\in(-M,(d-\alpha)/2)$. If $1/3<\alpha<1/2$, assume in addition that $\sigma'<(d+\alpha-1)/2$.
      
    \item[(b)] Assume $\alpha>1/2$, $\sigma<(d-1)/2$, $\rho_S>1$, and $V^S\in C^\infty(\R^d)$ satisfies~\eqref{ass5} for some $\rho_S>1$ and all multi-indices $\beta$.
    \end{enumerate}
    Then $W(L_\kappa - V^S,L_\kappa)$ is asymptotically complete. Moreover, in case (a), the spectrum of $L_\kappa-V^S$ is absolutely continuous and equals $[0,\infty)$.
  \end{enumerate}
\end{theorem}

Note that for $\alpha>1$, we can prove asymptotic completeness for all $\kappa\leq\kappa_{\rm c}$, including the critical coupling constant $\kappa_{\rm c}$. This is not possible using perturbative results, such as \cite{MizutaniYao2021}, which we use to treat the case $\alpha\leq1$. Of course, one should not interpret the bounds $\rho_S>\max\{1,2\alpha\}$, $\sigma'<(d-1+\alpha)/2$, and $\alpha>1/3$ in Item~\ref{asympcomplhardyplusperturbation:acshortrangesmallalpha}(a) or $\sigma<(d-1)/2$ and $\alpha>1/2$ in Item~\ref{asympcomplhardyplusperturbation:acshortrangesmallalpha}(b) as physical obstructions. These are likely technical requirements arising from the rather slow growth of $(-\Delta)^{\alpha/2}$ at high momenta, which weakens the propagation. On a technical level they arise from our proof where we compare powers of $L_\kappa$ with powers of $L_0$ (via Proposition~\ref{eqsob} below).
From the proof it transpires that Item~(\ref{asympcomplhardyplusperturbation:acshortrangesmallalpha})(a) holds whenever there is $s\in\R$ satisfying
\begin{align}
  \begin{split}
    \max\left\{\frac12 , \alpha\right\} < s < \min\left\{\frac{\rho_S}{2},\frac{d}{2},1,\frac{3\alpha}{2},\frac{d+\alpha-\sigma'}{2} ,\frac{d+\alpha-2\sigma'}{2} \right\}, \quad \alpha<1, \\
    \frac12 < s < \min\left\{\frac{\rho_S}{2},\frac{d}{2},1,\frac{d+1-\sigma'}{2} ,\frac{d+1-2\sigma'}{2} \right\}, \quad \alpha=1.
  \end{split}
\end{align}
Note that this condition is satisfied for $1/2\leq\alpha\leq1$ when $\sigma'<(d-\alpha)/2$. When $1/3<\alpha<1/2$, the condition $\sigma'<(d+\alpha-1)/2$ guarantees the existence of such $s$.

Let us give a sufficient condition for $V^S$ so that $\kappa/|x|^\alpha+V^S$ obeys the assumptions in Theorem~\ref{asympcomplhardyplusperturbation}(\ref{asympcomplhardyplusperturbation:acshortrangesmallalpha})(a). In the following, we write $x_+=\max\{x,0\}$ for $x\in\R$.

\begin{lemma}
  \label{lem:MY-sufficient}
  Let $d\in\N$, $\alpha\in(0,2]\cap(0,d)$, $\kappa<\kappa_{\rm c}$. Let $V^S:\R^d\to\R$ be such that $V_j^S:=(-x\cdot\nabla)^j V^S \in L_{\rm loc}^1(\R^d)$, and assume that
  \begin{equation}
    \label{eq:MY-norm}
    M := \sup_{x\in\R^d\setminus\{0\}} |x|^\alpha \sum_{j=0}^2 |V_j^S(x)| <\infty.
  \end{equation}
  Define $W_\kappa(x) := \kappa/|x|^\alpha + V^S(x)$.
  Then the following statements hold.

  \begin{enumerate}
  \item We have
    \begin{align}
      \label{eq:natural-sandwich}
      \frac{\kappa-M}{|x|^\alpha} \leq W_\kappa(x) \leq \frac{\kappa+M}{|x|^\alpha}.
    \end{align}
    
  \item If $M < \min\{1,\alpha\} (\kappa_{\rm c}-\kappa_+)$, then $W_\kappa$ also satisfies Assumption~\ref{mizutaniyaoconditions}.

  \item Let $\tilde\kappa<\kappa<\kappa'<\kappa_{\rm c}$ be given. If
    \begin{align}
      \label{eq:MY-smallness-fixed}
      M < \min\left\{\kappa-\tilde\kappa,\, \kappa'-\kappa,\, \alpha(\kappa_{\rm c}-\kappa_+) \right\},
    \end{align}
    then $\tilde\kappa/|x|^\alpha \leq W_\kappa(x) \leq \kappa'/|x|^\alpha$, and $W_\kappa$ satisfies Assumption~\ref{mizutaniyaoconditions}.

  \item Let $\rho_S\geq\alpha$ and let $V^S\in C^2(\R^d;\R)$. Set
    \begin{align}
      \label{eq:A0}
      A_0 & := \sup_{x\in\R^d} \langle x\rangle^{\rho_S}|V^S(x)|, \\
      \label{eq:A1}
      A_1 & := \sup_{x\in\R^d} \langle x\rangle^{\rho_S+1}|\nabla V^S(x)|, \\
      \label{eq:A2}
      A_2 & := \sup_{x\in\R^d} \langle x\rangle^{\rho_S+2} \|{\rm Hess}(V^S(x))\|_{\R^d,\R^d},
    \end{align}
    where $\|{\rm Hess}(V^S(x))\|_{\R^d,\R^d}$ denotes the $\R^d\to\R^d$-operator norm of the Hessian of $V^S$ at $x$. Assume
    \begin{align}
      \label{eq:symbol-smallness-zero}
      A_0 & \leq \min\{\kappa-\tilde\kappa,\kappa'-\kappa\}, \\
      \label{eq:symbol-smallness-one}
      A_1 & < \alpha(\kappa_{\rm c}-\kappa_+), \\
      \label{eq:symbol-second-finite}
      A_2 & < \infty.
    \end{align}
    Then, $W_\kappa(x)$ satisfies $\tilde\kappa |x|^{-\alpha} \leq W_\kappa(x) \leq \kappa'|x|^{-\alpha}$ and Assumption~\ref{mizutaniyaoconditions}. In particular, \eqref{eq:symbol-smallness-zero}--\eqref{eq:symbol-second-finite} hold if
    \begin{align}
      \label{eq:symbol-common-constant}
      |\partial^\beta V^S(x)| \leq C_S\langle x\rangle^{-\rho_S-|\beta|}, \quad x\in\R^d, \ |\beta|\leq2
    \end{align}
    with
    \begin{align}
      \label{eq:common-constant-smallness}
      C_S < \min\left\{\kappa-\tilde\kappa,\, \kappa'-\kappa,\, \frac{\alpha(\kappa_{\rm c}-\kappa_+)}{\sqrt d} \right\}.
    \end{align}
  \end{enumerate}
\end{lemma}

We defer the proof of this lemma to Appendix~\ref{s:proof-MY-sufficient}. The condition $\rho_S \geq \alpha$ in Item~(4) is natural from the point of view of short-range scattering theory. Indeed, Theorem~\ref{asympcomplhardyplusperturbationnonexistence} below states nonexistence of $W(L_\kappa-V^L,L_\kappa)$ for $V^L$ decaying like $\langle x\rangle^{-\rho_L}$ with $\rho_L\leq\alpha$.

\smallskip
The smoothness assumption in Item~(\ref{asympcomplhardyplusperturbation:acshortrangesmallalpha})(b) is due to our use of the following minimal escape velocity estimates, which are also useful to show nonexistence of wave operators with long-range potentials. To state the estimates, we denote the set of eigenvalues of a linear operator $H$ with domain $\D(H)$ by $\sigma_{\rm p}(H)$.

\begin{theorem}
  \label{the6}
  Let $d\in\N$, $\alpha\in(0,2]\cap(0,d)$, $\kappa<\kappa_{\rm c}$, and $\rho>0$.
  Assume $V\in C^\infty(\R^d)$ satisfies \eqref{ass5} for all multi-indices $\beta$.
  Let $H_\kappa=L_\kappa-V$ and $f\in C_c^\infty(0,\infty)$ with $\supp f \cap \sigma_{\rm p}(H_\kappa) = \emptyset$.
  Then, for all $\nu\geq1$ there are $\theta>0$ and $T>0$ such that
  \begin{equation}
    \label{eq_the6_1}
    \left\| \one\left(\frac{A}{t}\leq\theta\right)f(H_\kappa) \me{-i t H_\kappa}\phi\right\|
    \lesssim_{\theta,\nu,f,T,V} t^{-\nu}\|\langle A\rangle^\nu\phi\|, \quad t\geq T,
  \end{equation}
  holds for all $\phi\in\D(\langle A\rangle^\nu)$.
\end{theorem}

The proof follows from an abstract result in \cite{Skibsted1991}, which relies on estimates for iterated commutators. Although only a bounded number of such commutators (and hence derivatives of $V$) need to be controlled---their number depends on the desired decay exponent $\nu$---we have not attempted to determine or bound this number.

\smallskip
We now state sufficient conditions for the nonexistence of certain wave operators.

\begin{theorem}
  \label{asympcomplhardyplusperturbationnonexistence}
  Let $d\in\N$, $\alpha\in(0,d)\cap(0,2]$, and $\kappa\leq\kappa_{\rm c}$. Assume that $V^L$ satisfies~\eqref{ass1} or \eqref{ass2} for some $0<\rho_L\leq1$. Then the following statements hold.

  \begin{enumerate}
  \item \label{asympcomplhardyplusperturbation:nonexistencefree}
    If $\rho_L=\alpha$, assume in addition that there is $\epsilon_{\rm s}\in\{-1,1\}$ and $c,R>0$ so that
    \begin{align}
      \label{eq:noncancellation-equal-order}
      \epsilon_{\rm s} \left(\frac{\kappa}{|x|^\alpha}+V^L(x) \right)
      \geq c\langle x\rangle^{-\alpha}, \quad |x|\geq R.
    \end{align}
    Then $W(L_\kappa-V^L,L_0)$ does not exist.

  \item \label{asympcomplhardyplusperturbation:nonexistencerelativelargealpha}
    Assume $\alpha>1$. Then $W(L_\kappa-V^L,L_\kappa)$ does not exist.

  \item \label{asympcomplhardyplusperturbation:nonexistencelongrangesmallalpha}
    Assume $\kappa<\kappa_{\rm c}$ and $V^L\in C^\infty(\R^d)$ satisfies \eqref{ass5} for all multi-indices $\beta$. Assume further that either $1/2<\alpha\leq1$ and $0<\rho_L\leq\alpha$, or $d\geq2$, $\alpha=1/2$, and $0<\rho_L\leq1/2$. Then $W(L_\kappa-V^L,L_\kappa)$ does not exist.
  \end{enumerate}
\end{theorem}

We assume~\eqref{eq:noncancellation-equal-order} in Item~(\ref{asympcomplhardyplusperturbation:nonexistencefree}) since for $\rho_L=\alpha$, the long-range parts of the Hardy potential and $V^L$ could cancel if they have opposite signs and the same magnitude. This assumption is not necessary in the other two items, where the reference dynamics are generated by $L_\kappa$.

\smallskip
The nonexistence statements in Theorem~\ref{asympcomplhardyplusperturbationnonexistence}(\ref{asympcomplhardyplusperturbation:nonexistencefree},\ref{asympcomplhardyplusperturbation:nonexistencerelativelargealpha}) follow from an abstraction and generalization of an argument of Ishida--Wada \cite{IshidaWada2020} and, in Item~(\ref{asympcomplhardyplusperturbation:nonexistencerelativelargealpha}), asymptotic completeness and the chain rule. In particular, these arguments require no regularity assumptions on $V^L$ beyond \eqref{eq:basicassumplongrange}.
Item~(\ref{asympcomplhardyplusperturbation:nonexistencelongrangesmallalpha}) is of a different nature. When $\alpha\leq1$, the wave operator $W(L_\kappa,L_0)$ itself does not exist for $\kappa\neq0$, so a chain rule argument is unavailable. We prove Item~(\ref{asympcomplhardyplusperturbation:nonexistencelongrangesmallalpha}) more directly using minimal escape velocity estimates, and weighted propagation estimates for the reference dynamics $L_\kappa$. To prove the latter, we compare powers of $L_0$ and $L_\kappa$. This comparison requires the (likely technical) restriction $\alpha>1/2$, or $\alpha=1/2$ when $d\geq2$.


\subsection{Scattering theory for more general fractional operators and Hardy potentials}

We can actually extend several of the results in Theorems~\ref{asympcomplhardy}, \ref{asympcomplhardyplusperturbation}, and~\ref{asympcomplhardyplusperturbationnonexistence} to a wider class of Fourier multipliers.
For $\alpha<2$, we consider Fourier multipliers $C_0(-\Delta)$ whose heat kernels arise via subordination of $\me{t\Delta}$ and which are bounded from below by $(-\Delta)^{\alpha/2}$ modulo an additive constant. To make this precise, let $\sigma_t^{(\alpha/2)}(\tau)$ be the $\frac{\alpha}{2}$-stable subordination transition density function, i.e., the probability density function for which
\begin{align}
  \label{eq:laplacetransform}
  \me{-t\lambda^{\alpha/2}} = \int_0^\infty \me{-\tau\lambda}\, \sigma_t^{(\alpha/2)}(\tau)\,d\tau, \quad t,\lambda > 0.
\end{align}
For a list of properties of $\sigma_t^{(\alpha/2)}(\tau)$, and references where they are proved, see, e.g., \cite[Appendix~B]{BogdanMerz2024}.
Then, we define the following class of kinetic energy symbols
\begin{align}
  \label{eq:fmclass}
  \begin{split}
    \ck_\alpha := \Big\{ & T:[0,\infty)\to[0,\infty):\, \text{There is a map $\R_+\ni t \mapsto \sigma_t^{(T)}(\cdot)$ so that}  \\ &  \text{for all $t>0$, $\sigma_t^{(T)}(\cdot)$ is a density function on $\R_+$ for which} \\ & \text{there is $c\geq0$ such that $\sigma_t^{(T)}(\tau)\leq \me{ct} \sigma_t^{(\alpha/2)}(\tau)$ for all $t,\tau>0$ and} \\ & \text{$\me{-tT(\lambda)}=\int_0^\infty \me{-\tau\lambda}\sigma_t^{(T)}(\tau)\,d\tau$.} \Big\}.
  \end{split}
\end{align}
Importantly, if $C_0\in\ck_\alpha$, then
\begin{align}
  \me{-t C_0(-\Delta)}(x,y) \leq \me{ct}\me{-t L_0}(x,y), \quad x,y\in\R^d.
\end{align}
Hence, by Trotter's formula,
\begin{align}
  \me{-t(C_0(-\Delta)-\kappa/|x|^\alpha)}(x,y) \leq \me{ct}\me{-t L_\kappa}(x,y), \quad x,y\in\R^d.
\end{align}
For every $t>0$, the function $\lambda\mapsto \me{-tC_0(\lambda)}$ is completely monotone by Bernstein's theorem \cite[Theorem~1.4]{Schillingetal2012}.
Recalling the set of Bernstein functions
\begin{align}
  \label{eq:defbernstein}
  \cb = \left\{ \Phi\in C^\infty(0,\infty;[0,\infty)) : \, \Phi(\lambda)\geq0, \,(-1)^k\frac{d^k\Phi}{d\lambda^k}\leq0, \, \lambda\in(0,\infty), \, k\in\N \right\},
\end{align}
we see that $\ck_\alpha\subset\cb$; see, e.g., \cite[Theorem~3.7]{Schillingetal2012}. In particular, by operator monotonicity of the logarithm, we have for $C_0\in\ck_\alpha$,
\begin{align}
  C_0(-\Delta) \geq (-\Delta)^{\alpha/2} - c,
\end{align}
i.e., $C_0(-\Delta) - \kappa/|x|^\alpha$ is bounded from below in the sense of quadratic forms for all $\kappa\leq\kappa_{\rm c}$. An important example of a function in $\ck_\alpha$ is $\R_+\ni\lambda\mapsto(1+\lambda)^{\alpha/2}-1 \in \ck_\alpha$. Indeed,
\begin{align}
  \me{-t[(\lambda+1)^{\alpha/2}-1]} = \me{t} \int_0^\infty \me{-\tau\lambda}\me{-\tau}\sigma_t^{(\alpha/2)}(\tau)\,d\tau, \quad t>0,\ \lambda\geq0,
\end{align}
and in this case, $c=1$. When $\alpha=1$, the corresponding Fourier multiplier provides a relativistic description of the kinetic energy of a particle in units where the particle mass and speed of light are set to $1$.
We remark that the notation $C_0$ is inspired by Chandrasekhar's use of $(1-\Delta)^{\alpha/2}-1-\kappa/|x|^\alpha$ as a model for studying the stability of stars \cite{Chandrasekhar1931}, in which case $\alpha=1$.

We are now ready to state the following result on asymptotic completeness of wave operators involving $C_0\in\ck_\alpha$.

\begin{corollary}
  \label{cor:massiverelativisticexistence}
  Let $d\in\N$, $\alpha\in(1,\min\{2,d\})$, and $\kappa\leq\kappa_{\rm c}$.
  Let $C_0\in\ck_\alpha$ and assume in addition that, with $h(k):=C_0(k^2)$, $k>0$, the set $M_2:=\{k>0:\,h''(k)=(h'(k))^2\}$ has no accumulation point at zero.
  Let further $\tilde\kappa\leq\kappa$, and $V_\kappa:\R^d\to\R$ be such that
  $$
  \frac{\tilde\kappa}{|x|^\alpha} \leq V_\kappa(x) \leq \frac{\kappa}{|x|^\alpha}, \quad x\in \R^d,
  $$
  and assume $V^S$ satisfies \eqref{eq:basicassumpshortrange} for some $\rho_S>1$. Then, the wave operators $W(C_0-V_\kappa,C_0)$, $W(C_0-V_\kappa-V^S,C_0-V_\kappa)$, and $W(C_0-V_\kappa-V^S,C_0)$ are asymptotically complete under the same assumptions as in Theorems~\ref{asympcomplhardy}(1) and~\ref{asympcomplhardyplusperturbation}(\ref{asympcomplhardyplusperturbation:acshortrange}).
\end{corollary}

Let us explain the assumption on the above set $M_2$. Our proof uses that the Fourier multiplier generating the reference dynamics has the property that $M_2$ and $M_1:=\{k>0:\,h'(k)=0\}$ are discrete in $\R_+$ with only possible accumulation point at $\infty$. In the proof we will show that $C_0$ is a Bernstein function implies that $M_1=\emptyset$, while $M_2$ has no accumulation point in $(0,\infty)$. Since this does not rule out accumulation at zero, we have to enforce this assumption by hand.
  Observe that $C_0=(1-\Delta)^{\alpha/2}-1$ obeys the assumptions of Corollary~\ref{cor:massiverelativisticexistence}. Indeed, $h'(k)=\alpha k(k^2+1)^{\alpha/2-1}>0$ for all $k>0$ and $h''(k)-(h'(k))^2=\alpha(k^2+1)^{\alpha/2-2}q(k^2)$, where $q(y)=1+(\alpha-1)y-\alpha y(y+1)^{\alpha/2}$. For $\alpha>1$, we have $q'(y)<0$. Since $q(0)=1$ and $q(y)\to-\infty$ as $y\to\infty$, $q$ has one positive zero, so $M_2$ consists just of one point.

In \cite[Theorem~5.1]{MizutaniYao2021}, the authors generalize the Kato smoothness theory used in our proof of Theorem~\ref{asympcomplhardyplusperturbation}(\ref{asympcomplhardyplusperturbation:acshortrangesmallalpha})(a) to a larger class of Fourier multipliers obeying \cite[(5.1)]{MizutaniYao2021}. This assumption states that the map $\xi\mapsto C_0(|\xi|^2)$ belongs to $C(\R^d)\cap C^2(\R^d\setminus\{0\})$ and there are constants $c_{1,j},c_{2,j},c>0$ such that
\begin{align}
  \label{eq:mizutaniyaocond5}
  c_{1,j}|\xi|^\alpha \leq (\xi\cdot\nabla)^j C_0(|\xi|^2) \leq c_{2,j} \langle \xi\rangle^\alpha, \quad
  |(\xi\cdot\nabla)^2 C_0(|\xi|^2)| \leq c |(\xi\cdot\nabla) C_0(|\xi|^2)|
\end{align}
holds for all $\xi\in\R^d$ and $j=0,1$. Hence, by an adaptation of our proof, one could also prove asymptotic completeness of $W(C_0-\kappa/|x|^\alpha-V^S,C_0-\kappa/|x|^\alpha)$ for such Fourier multipliers $C_0$ under the assumptions of Theorem~\ref{asympcomplhardyplusperturbation}(\ref{asympcomplhardyplusperturbation:acshortrangesmallalpha})(a). However, the massive pseudo-relativistic kinetic energy $C_0(|\xi|^2)=(|\xi|^2+1)^{\alpha/2}-1$ does not satisfy the lower bound in \eqref{eq:mizutaniyaocond5} for $\alpha<2$. (We correct here a minor inaccuracy in \cite[p.~609]{MizutaniYao2021}, which states this for all $0<\alpha<d$.) Since we are currently unable to prove asymptotic completeness of $W(C_0-\kappa/|x|^\alpha-V^S,C_0-\kappa/|x|^\alpha)$ with the physically important symbol $C_0(|\xi|^2)=(|\xi|^2+1)^{\alpha/2}-1$, we omit the statement and proof of asymptotic completeness of $W(C_0-\kappa/|x|^\alpha-V^S,C_0-\kappa/|x|^\alpha)$ with $C_0$ obeying \eqref{eq:mizutaniyaocond5}.

\smallskip
Finally, we also prove the following nonexistence results.

\begin{corollary}
  \label{cor:massiverelativisticnonexistence}
  Let $d\in\N$, $\alpha\in(0,d)\cap(0,2]$ and $\kappa\leq\kappa_{\rm c}$.
  Let $C_0\in\ck_\alpha$ and $C_\kappa:=C_0-\kappa/|x|^\alpha$. Assume that $V^L$ satisfies \eqref{eq:basicassumplongrange}. Then the following statements hold.
  \begin{enumerate}
  \item If $\alpha\leq1$ and $\kappa\neq0$, then $W(C_\kappa,C_0)$ does not exist.

  \item If $\alpha\neq\rho_L$, then $W(C_\kappa-V^L,C_0)$ does not exist.
    The same conclusion holds when $\alpha=\rho_L$ provided there are $\epsilon_{\rm s}\in\{-1,1\}$, $c>0$, and $R>0$ such that
    \begin{align}
      \label{eq:massive-equal-order-noncancellation}
      \epsilon_{\rm s} \left(\frac{\kappa}{|x|^\alpha}+V^L(x)\right)
      \geq c\langle x\rangle^{-\alpha}, \quad |x|\geq R.
    \end{align}

  \item Assume $\alpha>1$. Then $W(C_\kappa-V^L,C_\kappa)=\slim_{t\to\infty}\me{i t(C_\kappa-V^L)}\me{-i tC_\kappa}P_{\rm ac}(C_\kappa)$ does not exist.
  \end{enumerate}
\end{corollary}

Proving the nonexistence of $W(C_\kappa-V^L,C_\kappa)$ for $\alpha\leq1$ would also require propagation estimates for $C_\kappa$. Proving them goes beyond the scope of this manuscript. Let us conclude with an optimistic remark: We expect that the technical threshold $\alpha>1/2$ (with the endpoint $\alpha=1/2$ included when $d\geq2$) from the nonexistence Theorem~\ref{asympcomplhardyplusperturbationnonexistence}(\ref{asympcomplhardyplusperturbation:nonexistencelongrangesmallalpha}) is not necessary. This is best seen when considering the operator $(1-\Delta)^{\alpha/2}-1$ whose group velocity is linear in the momentum for small momenta.


\subsection{Discussion and comparison with earlier results}
\label{ss:previousresults}

In the following, we briefly survey results in scattering theory for fractional or homogeneous Schr\"odinger operators. This overview is not meant to be complete but to point to related earlier results.

\subsubsection*{Scattering theory for fractional Schr\"odinger operators}

Scattering theory for nonlocal Schr\"odinger operators appears to have begun in \cite{Schechter1976} and \cite{Simon1979P}, where asymptotic completeness was established for sufficiently smooth kinetic energy symbols, such as $(1-\Delta)^{\alpha/2}$, and sufficiently smooth short-range potentials. More than ten years later, \cite{Sigal1990} proved asymptotic completeness for suitably modified wave operators associated with such kinetic energy symbols in the presence of long-range potentials using propagation estimates. Another decade later, the short-range scattering theory for perturbations of $(-\Delta)^{\alpha/2}$ was developed, which is more delicate because of the singularity in Fourier space at low momenta. For $\sqrt{-\Delta}-V$ with $|V|\lesssim\langle x\rangle^{-1-\epsilon}$, \cite{Umeda2001,Umeda2006,Wei2007} established eigenfunction expansions, asymptotic completeness of the associated wave operators, and boundedness and large-distance asymptotics of generalized eigenfunctions. Subsequently, the scattering theory for $(-\Delta)^{\alpha/2}-V$ with $1\leq\alpha<2$ and long-range potentials was studied in \cite{Kitada2010,Kitada2011}, where modified wave operators were constructed in the spirit of the local theory $\alpha=2$ and shown to be asymptotically complete via propagation estimates for $L_0$. In \cite{Soffer2016}, maximal and minimal escape velocity bounds for $(-\Delta)^{1/2}-V$ in $d=3$ were proved under the assumption that $V$, $x\cdot\nabla V$, and $\langle x\rangle\Delta V$ decay faster than $\langle x\rangle^{-2}$. More recently, short-range scattering theory in the spirit of Agmon--H\"ormander \cite{AgmonHormander1976} has been developed in \cite{Zhangetal2021} for all $\alpha>0$. In \cite{IshidaWada2020}, the sharp threshold between existence and nonexistence of wave operators for a general class of nonlocal kinetic energy symbols with perturbations was investigated; this work, together with \cite{IshidaKawamoto2023}, served as inspiration for our nonexistence results. Finally, inverse scattering theory and propagation estimates for $(-\Delta)^{\alpha/2}$ with $1<\alpha\leq2$ are treated in \cite{Ishida2020} and minimal escape velocity bounds are proved for $(1-\Delta)^{\alpha/2}-1-V_{\rm sing}-V_{\rm short}-V_{\rm long}$ with $0<\alpha\leq2$ in \cite{Ishida2025,SigalWu2025}.

\subsubsection*{Scattering theory for homogeneous operators related to Hardy operators}

For $\alpha=2$, the asymptotic completeness of $W(L_\kappa,L_0)$ was proved in \cite[Corollary~2.1]{Boutetetal1996} by means of positive commutators and Mourre theory. This approach requires the Hardy potential to be form-bounded with respect to $L_0$ with relative bound strictly smaller than $1$, which requires $\kappa<\kappa_{\rm c}$. In contrast, our result in Theorem~\ref{asympcomplhardy}(1) also covers the critical coupling constant $\kappa=\kappa_{\rm c}$, since it does not rely on commutator estimates but instead uses the ground state of $L_\kappa$ in a more direct way. Put differently, the reason why we can also treat $\kappa_{\rm c}$ here is that the structure of the heat kernel, unlike that of the resolvent, is insensitive to a zero-energy resonance in the form of an incipient eigenvalue.

In \cite[Theorem~4.10]{DerezinskiRichard2017}, wave operators intertwining Bessel operators, formally given by $-\partial_r^2+(m^2-1/4)/r^2$ in $L^2(\R_+,dr)$ and indexed by $m\in\C$ with $\Re(m)>-1$, were explicitly constructed by the stationary approach to scattering theory and shown to coincide with the wave operators from the time-dependent approach considered here. For special values of the Bessel parameter $m$, these coincide with the radial operators arising in the angular momentum decomposition of $L_\kappa$.

In \cite{DAnconaetal2026}, the authors show existence and completeness of the wave operators $W(\mathcal{L}_{\vec{A},a},-\Delta)$ and $W(H_\alpha,H_0)$, where $\mathcal{L}_{\vec{A},a}=(-i\nabla+\vec A(\theta)/r)^2+a(\theta)/r^2$ and $H_\alpha=(-i\nabla+\alpha(-x_2,x_1)/|x|^2)^2$ is the Aharonov--Bohm operator, both realized in $L^2(\R^2)$; here $\vec{A}(\theta)=b(\theta)(-\sin\theta,\cos\theta)$ with $b\in C^1(\mathbb T,\R)$. Both magnetic vector potentials allow to accommodate a non-trivial Hardy potential when their fluxes are nonintegral \cite{LaptevWeidl1999}.

Finally, modified wave operators for the physically relevant Dirac--Coulomb operator have also been studied. Their existence was proved in \cite{DollardVelo1966}, while asymptotic completeness was established in \cite{Muthuramalingam1985} and \cite{EnssThaller1986}. More general long-range potentials were treated in \cite{Thaller1986}. For a textbook account and an extensive list of references, see \cite{Thaller1992}.


\subsection{Outlook}

Let us give two future research directions.
\begin{enumerate}
\item It would be interesting to study scattering for other homogeneous operators, such as $(-\Delta)^{\alpha/2}|_\Omega-\kappa/{\rm dist}(x,\partial\Omega)^\alpha$ with Dirichlet boundary conditions at the boundary $\partial\Omega$ of a connected, unbounded domain $\Omega\subset\R^d$. To this end, it seems advantageous to prove corresponding propagation estimates. The semigroup framework in \cite{Demuthetal2001,Giere2004} relies on the Fourier transform at many places, which is unavailable in this setting. It would be interesting to see to which extent this framework can be extended. Notably, sharp heat kernel bounds for this operator are proved in \cite{Choetal2020,Songetal2025,Kinzebulatovetal2021,KinzebulatovSemenov2023F}.

\item A more ambitious goal is the $N$-particle scattering theory for
  $$
  \sum_{j=1}^N(-\Delta_j)^{\alpha/2} - \sum_{j<k}\frac{\kappa}{|x_j-x_k|^\alpha},
  $$
  possibly with additional pair interactions. The operator could be considered on $L^2(\R^{dN})$ or on its totally symmetric (bosonic) or antisymmetric (fermionic) subspace. The range of admissible coupling constants $\kappa$ depends both on the chosen Hilbert space and on $N$. Sharp ranges are currently known only in the limit $N\to\infty$ \cite{LiebYau1987,Franketal2024} or for $N=3$ \cite{Mahadevan2026}. When $\alpha=1$, the above operator provides, e.g., a hyperrelativistic description of particles in gravitational systems like stars. The commutation relation with the many-particle-analogue of the generator of dilations suggests that the present propagation estimates can be extended to the many-particle case, which would be an important ingredient in this endeavor.
\end{enumerate}


\subsection{Organization}

In Section~\ref{s:wlkappal0complete}, we prove asymptotic completeness of $W(L_\kappa,L_0)$ when $\alpha>1$ (Theorem~\ref{asympcomplhardy}(1)).
In Section~\ref{s:wlkappavlkappacomplete}, we prove asymptotic completeness of $W(L_\kappa-V^S,L_\kappa)$ when $\alpha>1$ (Theorem~\ref{asympcomplhardyplusperturbation}(\ref{asympcomplhardyplusperturbation:acshortrange})).
In Sections~\ref{s:asympcomplperturbedwosmallalpha_k} and~\ref{s:asympcomplperturbedwosmallalpha_a}, we prove asymptotic completeness of $W(L_\kappa-V^S,L_\kappa)$ for certain $\alpha\leq1$ (Theorem~\ref{asympcomplhardyplusperturbation}(\ref{asympcomplhardyplusperturbation:acshortrangesmallalpha})).
In Section~\ref{s:technicalpreparations}, we prepare tools to prove the propagation estimates; their actual proofs are carried out in Section~\ref{s:propagationestimates}. In that section, we also prove an outgoing phase space propagation estimate, Lemma~\ref{the11suff-unified}, which is important for applying the Cook--Kuroda method.
In Section~\ref{s:proofcormassiverelativistic}, we prove Corollary~\ref{cor:massiverelativisticexistence}.
In Section~\ref{s:nonexistence}, we prove the nonexistence results in Theorems~\ref{asympcomplhardy}(2) and~\ref{asympcomplhardyplusperturbationnonexistence}, and Corollary~\ref{cor:massiverelativisticnonexistence}.
%


\section{Asymptotic completeness of $W(L_0-V_\kappa,L_0)$ when $1<\alpha\leq2$}
\label{s:wlkappal0complete}

Let $\alpha\in(1,2]$, $\alpha<d$, $\kappa\leq\kappa_{\rm c}$, and $V_\kappa$ be as in Theorem~\ref{asympcomplhardy}.
To prove the asymptotic completeness of $W(L_0-V_\kappa,L_0)$ in this setting, we use the semigroup framework in \cite{Demuthetal2001,Giere2004}, first developed for the free dynamics generated by $-\Delta$ in \cite{DemuthSinha1999}. Let us state this framework now. In passing we note that the main results in \cite{Demuthetal2001,Giere2004} are the same, with \cite{Giere2004} containing detailed proofs and an application to scattering by obstacles.

\begin{assumption}
  \label{assumheatdifference}
  Let $H_0$ and $H$ be two selfadjoint operators in $L^2(\R^d)$ that are bounded from below. We assume
  \begin{enumerate}
  \item[(i)] $\me{-t H}-\me{-t H_0}$ is compact for a fixed $t>0$, such as $t=1$, and
    
  \item[(ii)] there is $\gamma>1$ such that $(\me{-H}-\me{-H_0})\langle x\rangle^\gamma$ is a bounded operator on $L^2(\R^d)$.
  \end{enumerate}
\end{assumption}

\begin{theorem}[{\cite[Theorem~1.1]{Demuthetal2001}, \cite[Theorem~1.5]{Giere2004}}]
  \label{giere}
  Let $H_0=P(D)$ on $L^2(\R^d)$, where $P:\R^d\to\R$ is a real-valued, nonnegative, spherically symmetric Fourier multiplier, which is smooth away from the origin. Let $h(|\xi|):=P(\xi)$ and assume $h$ is nondecreasing with $h(k)\to\infty$ as $k\to\infty$.
  Assume further
  $$
  M_1 := \{ k>0:\, h'(k)=0\} \ \text{and}\ M_2 := \{ k>0:\, h''(k)=(h'(k))^2 \}
  $$
  are discrete in $\R_+$ with only possible accumulation point at infinity.
  Suppose Assumption~\ref{assumheatdifference} holds.  
  Then, the wave operators
  $$
  W(H,H_0)
  := \slim_{t\to\infty}\me{i t H}\me{-i t H_0}P_{ac}(H_0)
  $$
  are asymptotically complete. Specifically, $\ran(W(H,H_0))=\ch_{ac}(H)$ and $\ch_{\rm sc}(H)=\{0\}$.
\end{theorem}

\begin{remark}~
  \noindent
  (1) In \cite[Assumption~1.4]{Giere2004}, it is assumed that $S(t)-S_0(t)$ is compact for all $t\geq0$.
  However, it suffices to verify this assumption only for fixed $t_0>0$. Indeed, the compactness of $S(t_0)-S_0(t_0)$ for a fixed $t_0>0$ implies compactness of $S(t)-S_0(t)$ for every $t>0$. To see this, let $A:=S(t_0)$ and $A_0:=S_0(t_0)$, i.e., $\me{-tH}=A^{t/t_0}$ and $\me{-tH_0}=A_0^{t/t_0}$. Since $\R_+\ni s\mapsto s^{t/t_0}$ is continuous, it follows from functional calculus that also $A^{t/t_0}-A_0^{t/t_0}$ is compact. The fact that Assumption~\ref{assumheatdifference} is only needed for fixed $t$ is also stated explicitly in \cite{Demuthetal2001}.

  \noindent
  (2) In this section, we use $P(\xi)=|\xi|^\alpha$. In this regard, we note that \cite[p.~133]{Giere2004} states that $P$ is assumed to be $C^\infty$, but inspection of the proof shows that only $P\in C^\infty(\R^d\setminus\{0\})$ is used.

  \noindent
  (3) The assumption that the sets $M_1$ and $M_2$ in Theorem~\ref{giere} are discrete comes from running the non-stationary phase argument with phase function $\exp\left(ix\cdot\xi-it\me{-P(\xi)}\right)$.
\end{remark}

Let us write
\begin{align}
  S_\kappa(t)(x,y) := \me{-tL_\kappa}(x,y),
  \qquad
  S_V(t)(x,y) := \me{-t(L_0-V_\kappa)}(x,y).
\end{align}
To verify Assumption~\ref{assumheatdifference}, let us collect some properties of $S_\kappa(t)$. First, we record $\int_{\R^d}\me{-t L_0}(x,y)\,dy = 1$ and that $\me{-tL_\kappa}$ satisfies the Chapman--Kolmogorov equation and semigroup property, i.e.,
$$
\me{-t_1 L_\kappa} \me{-t_2 L_\kappa} = \me{-(t_1+t_2) L_\kappa} \quad \text{in $L^2$ and pointwise.}
$$
Next, we recall the following sharp heat kernel bounds for $L_\kappa$ in \cite{MilmanSemenov2004,Bogdanetal2019,JakubowskiWang2020,Choetal2020}: For all $t>0$ and $x,y\in\R^d$, we have
\begin{align}
  \label{eq:heatkernelbounds}
  S_\kappa(t)(x,y)
  & \sim_{d,\alpha,\kappa} \me{-t L_0}(x,y) \left(1\wedge \frac{|x|}{t^{1/\alpha}}\right)^{-\sigma} \left(1\wedge\frac{|y|}{t^{1/\alpha}}\right)^{-\sigma},
  \\
  \label{eq:heatkernelboundsfree}
  \begin{split}
    S_0(t)(x,y) & \sim_{d,\alpha} \frac{t}{(t^{1/\alpha}+|x-y|)^{d+\alpha}}, \quad \alpha<2, \\
    S_0(t)(x,y) & = (4\pi t)^{-d/2}\me{-|x-y|^2/(4t)}, \quad \alpha=2,
  \end{split}
\end{align}
where $\sigma=\sigma_\kappa\in(-M,(d-\alpha)/2]$ is the unique solution to the equation in \eqref{eq:defsigma}. Here and below, we write $A\vee B:=\max\{A,B\}$ and $A\wedge B:=\min\{A,B\}$ for $A,B \geq 0$.
In our analysis below, we will in fact only need these bounds for $\kappa>0$, while for $\kappa<0$ we content ourselves with $\me{-t L_\kappa}(x,y) \leq \me{-t L_0}(x,y)$.
To ease the notation, we write
\begin{align}
  S_\kappa := S_\kappa(1), \qquad S_V:=S_V(1).
\end{align}

We can now give the

\begin{proof}[{Proof of Theorem~\ref{asympcomplhardy}(1)}]
  By Trotter's formula,
  $$
  \me{-tL_{\tilde\kappa}}(x,y) \leq \me{-t(L_0-V_\kappa)}(x,y) \leq \me{-tL_\kappa}(x,y)
  $$
  and therefore
  $$
  \left| S_V(t)(x,y) - S_0(t)(x,y) \right|
  \leq \left|S_\kappa(t)-S_0(t)\right| + \left|S_{\tilde\kappa}(t)-S_0(t)\right|.
  $$
  To conclude the proof, we apply Theorem~\ref{giere} by verifying Assumption~\ref{assumheatdifference}. We do so in the following two subsections, where we just consider $t=1$.
  To verify (ii), it suffices, by the above estimate, to consider the difference $S_\kappa-S_0$. To verify (i) for $\me{-H}=S_V$ and $\me{-H_0}=S_0$, we first verify it for $\me{-H}=S_\kappa$ and then show how the compactness of $S_\kappa-S_0$ can be used to show compactness of $S_V-S_0$.
  Since our techniques make no use of the locality of $-\Delta$ or the Gaussian decay of heat kernels corresponding to the operators when $\alpha=2$, we will in the following only treat the case $0<\alpha<2$ explicitly. The generalizations to $\alpha=2$ are analogous.
\end{proof}

Assume $\kappa\neq0$ and let
\begin{align}
  D(t):=\kappa^{-1}(S_\kappa(t)-S_0(t)).
\end{align}
By Duhamel's formula,
\begin{align}
  D(t)(x,y) = \int_0^t ds\, \int_{\R^d}dz\, \me{-(t-s)L_\kappa}(x,z)|z|^{-\alpha} \me{-s L_0}(z,y).
\end{align}
Moreover, we write
\begin{align}
  D(x,y) := D(1)(x,y).
\end{align}
It will be useful to split the $z$-integral in Duhamel's formula to control the Hardy singularity. Thus, for $R>0$, we decompose
\begin{align}
  D(t)(x,y) & = D^{<R}(t)(x,y) + D^{>R}(t)(x,y), \\
  D^{<R}(t)(x,y) & := \int_0^t ds\, \int_{|z|<R}dz\, \me{-(t-s)L_\kappa}(x,z)|z|^{-\alpha}\me{-s L_0}(z,y), \\
  D^{>R}(t)(x,y) & := \int_0^t ds\, \int_{|z|>R}dz\, \me{-(t-s)L_\kappa}(x,z)|z|^{-\alpha}\me{-s L_0}(z,y).
\end{align}
Similarly as before, write $D^{>R}(x,y):=D^{>R}(1)(x,y)$ and $D^{<R}(x,y):=D^{<R}(1)(x,y)$. We emphasize that we will not use $-i\nabla$ in this section, which enables us to abuse notation by writing $D(x,y):=D(1)(x,y)$ for the kernel of $D:=D(1)$.
The Duhamel formula is only going to be used when we verify that $D\langle x\rangle^\gamma \in \mathcal B(L^2(\R^d))$, where $\mathcal B(L^2(\R^d))$ denotes the space of bounded linear operators on $L^2(\R^d)$.


\subsection{On the compactness assumption}
\label{s:compactness}

In this subsection, we verify (i) in Assumption~\ref{assumheatdifference}. To show compactness of $D(1)$, we approximate it by a sequence of Hilbert--Schmidt operators $D_r(1):=\chi_r D(1)$, where $\chi_r$ is the indicator function on the ball $B_r(0)$ centered at the origin with radius $r>0$. We also write $\chi_r^c:=1-\chi_r$ and $D_r:=D_r(1)$. Specifically, it suffices to show that for all $r>0$, the operator $D_r$ is compact and that $\|D-D_r\|\to0$ as $r\to\infty$.

\subsubsection{The case $\kappa<0$}

Here we use $\me{-L_\kappa}(x,y)\leq\me{-L_0}(x,y)$.
We first show compactness of $D_r$. Since $|D_r(x,y)|\leq 2|\kappa|^{-1}\chi_r(x)\me{-L_0}(x,y)$,
\begin{align}
  \|D_r\|_{\rm HS} \lesssim r^{d/2}
\end{align}
by the Seiler--Simon inequality \cite[Theorem~4.1]{Simon2005}; here, $\|\cdot\|_{\rm HS}$ denotes the Hilbert--Schmidt norm.

Let us now show that the operator norm of $\chi_r^c D$ vanishes as $r\to\infty$. To this end, we show
\begin{align}
  \lim_{r\to\infty} \left(\|\chi_r^c D \chi_{r/2}^c\| + \|\chi_r^c D \chi_{r/2}\| \right) = 0.
\end{align}
Since $|D(x,y)| \leq 2 |\kappa|^{-1} \me{-L_0}(x,y)$,
\begin{align}
  \|\chi_r^c D \chi_{r/2}\|
  \leq 2 \|\chi_r^c \me{-L_0} \chi_{r/2}\|
  \lesssim r^{-\alpha}
\end{align}
follows from the computations carried out in \cite[p.~149]{Giere2004}; see in particular the third displayed formula. On the other hand,
\begin{align}
  \|\chi_r^c D \chi_{r/2}^c\|
  \leq \|\chi_r^c D \langle x\rangle^\gamma\| \|\langle x\rangle^{-\gamma}\chi_{r/2}^c\|
  \lesssim r^{-\gamma}
\end{align}
in view of $\|\chi_r^c D \langle x\rangle^\gamma\| \leq \|D \langle x\rangle^\gamma\| \lesssim 1$, which we show in Subsection~\ref{s:weightedbound}, where we verify (ii) in Assumption~\ref{assumheatdifference}.

Thus, for $\kappa<0$, the difference of the heat kernels $D$ is compact.

\subsubsection{The case $\kappa>0$}

Recall that $0<\sigma<d/2$ since $\sigma\leq(d-\alpha)/2$.
Let $r>2$.
Note first that $\chi_r S_\kappa$ is Hilbert--Schmidt since
\begin{align}
  \begin{split}
    \|\chi_r S_\kappa\|_{\rm HS}^2
    & \lesssim \int_{\R^{d}}dx\, \chi_r(x)^2 \left(1\wedge |x|\right)^{-2\sigma} \int_{\R^d} dy\, \left(1\wedge |y|\right)^{-2\sigma} \frac{1}{(1 + |x-y|)^{2d+2\alpha}} \\
    & \lesssim \left(\int_{|x|<2r}dx\, (1 + |x|^{-2\sigma})\right) \sup_{|x|<2r}\int_{\R^d}dy\, (1 + |y|^{-2\sigma})(1 + |x-y|)^{-2d-2\alpha} \\
    & \lesssim r^d (1 + r^{-2\sigma}) \cdot \left( r^d (1 + r^{-2\sigma}) + r^{-d-2\alpha}\right)
      \lesssim r^{2d}.
  \end{split}
\end{align}

Next, we show that the operator norm of $(1-\chi_r)D$ vanishes as $r\to\infty$. To this end, we show
\begin{align}
  \lim_{r\to\infty} \left(\|\chi_r^c D \chi_{r/2}^c\| + \|\chi_r^c D \chi_{r/2}\| \right) = 0.
\end{align}
We start with the second summand and estimate $\|\chi_r^c D \chi_{r/2}\|$ using a weighted Schur test with weight $\left(\frac{1 \wedge |y|}{1\wedge |x|}\right)^{-\sigma}$. Using $|D(x,y)| \lesssim S_\kappa(x,y)$, we have
\begin{align}
  \begin{split}
    & \sup_{|x|>r}\int_{|y|<r/2}dy\, |D(x,y)| \cdot \left(\frac{1\wedge |y|}{1\wedge |x|}\right)^{-\sigma} \\
    & \quad \lesssim \sup_{|x|>r}\int_{|y|<r/2}dy\, \left(1 \wedge |x|\right)^{-\sigma}\left(1 \wedge |y|\right)^{-\sigma} \frac{1}{(1 + |x-y|)^{d+\alpha}} \cdot \left(\frac{1\wedge |y|}{1\wedge |x|}\right)^{-\sigma} \\
    & \quad \lesssim \sup_{|x|>r} \int_{|y|<r/2}(1+|y|^{-2\sigma}) |x|^{-d-\alpha}
      \lesssim r^{-\alpha},
  \end{split}
\end{align}
which vanishes as $r\to\infty$. On the other hand,
\begin{align}
  \begin{split}
    & \sup_{|y|<r/2}\int_{|x|>r}dx\, |D(x,y)| \cdot \left(\frac{1\wedge |x|}{1\wedge |y|}\right)^{-\sigma} \\
    & \quad \lesssim \sup_{|y|<r/2}\int_{|x|>r}dx\, \left(1\wedge |x|\right)^{-\sigma} \left(1\wedge |y|\right)^{-\sigma} \frac{1}{(1 + |x-y|)^{d+\alpha}} \cdot \left(\frac{1\wedge |x|}{1\wedge |y|}\right)^{-\sigma} \\
    & \quad \lesssim \int_{|x|>r}dx\, |x|^{-d-\alpha}
      \sim r^{-\alpha},
  \end{split}
\end{align}
which also vanishes as $r\to\infty$.

Finally, we show $\lim_{r\to\infty}\|\chi_r^c D \chi_{r/2}^c\|=0$. We have
\begin{align}
  \|\chi_r^c D \chi_{r/2}^c\|
  \leq \|\chi_r^c D \langle x\rangle^\gamma\| \|\langle x\rangle^{-\gamma}\chi_{r/2}^c\|
  \lesssim r^{-\gamma}
\end{align}
in view of $\|\chi_r^c D \langle x\rangle^\gamma\| \leq \|D \langle x\rangle^\gamma\| \lesssim 1$, which we show in Subsection~\ref{s:weightedbound}, where we verify (ii) in Assumption~\ref{assumheatdifference}.

\subsubsection{Conclusion of the verification of the compactness assumption}

We finally show compactness of
$$
D_V := S_V - S_0.
$$
Recall $\tilde\kappa/|x|^\alpha \leq V_\kappa(x) \leq \kappa/|x|^\alpha$. Thus, $S_{\tilde\kappa}(x,y) \leq S_V(x,y) \leq S_\kappa(x,y)$ by Trotter's formula. Hence,
\begin{equation}
  \label{eq:DV-endpoint-domination}
  |D_V(x,y)| \leq |S_\kappa(x,y)-S_0(x,y)| +  |S_{\widetilde\kappa}(x,y)-S_0(x,y)|.
\end{equation}
The two differences on the right-hand side have integral kernels with a fixed sign. Moreover, the above analysis shows that their localizations by $\chi_r$ are Hilbert--Schmidt and that the operator norms of their complements (i.e., the semigroup differences multiplied with $(1-\chi_r)$) converge to zero as $r\to\infty$. Thus, by \eqref{eq:DV-endpoint-domination}, we see that $\chi_rD_V$ is Hilbert--Schmidt for every $r>0$. Moreover, for $f \in L^2$,
$$
|\chi_r^c D_V f| \leq \chi_r^c \left( |S_\kappa-S_0|+|S_{\widetilde\kappa}-S_0| \right)|f|.
$$
Thus,
$$
\|\chi_r^c D_V\| \leq \|\chi_r^c |S_\kappa-S_0|\| + \|\chi_r^c|S_{\widetilde\kappa}-S_0|\| \to0.
$$
Thus, $D_V$ is a norm limit of the compact operators $\chi_rD_V$ and therefore also compact. Thus, we have verified Assumption~\ref{assumheatdifference}(i).


\subsection{On the boundedness assumption}
\label{s:weightedbound}

We verify Assumption~\ref{assumheatdifference}(ii) by showing $D\langle x\rangle^\gamma \in \mathcal B(L^2)$ for every $\gamma\in(1,\alpha)$, which is a nonempty interval since $\alpha>1$.

\subsubsection{The case $\kappa<0$}

Since $\langle x\rangle^\gamma \sim \one_{|x|<1}+|x|^\gamma\one_{|x|>1}$ and $|D(x,y)| \leq 2|\kappa|^{-1} \me{-L_0}(x,y)$ and $\me{-2L_0}$ is $L^2$-bounded, it suffices to check whether $D(x,y)|y|^\gamma\one_{|y|>1}$ generates an $L^2$-bounded operator.
To this end, we split the Duhamel integral in $D(x,y)$ at $|y|/2$. Then,
\begin{align}
  D^{>|y|/2}(x,y)|y|^\gamma \lesssim |y|^{\gamma-\alpha} \me{-L_0}(x,y).
\end{align}
Since $\me{-L_0}$ is $L^2$-bounded and $1<\gamma<\alpha$, the part $D^{>|y|/2}(x,y) \langle y\rangle^\gamma$ generates an $L^2$-bounded operator. We now show the $L^2$-boundedness of the operator generated by $D^{<|y|/2}(x,y)|y|^\gamma\one_{|y|>1}$. To this end, we split the operator into its contributions from $|x|<R_0$ and $|x|>R_0$ for some $R_0>0$ to be chosen below; for instance, $R_0=2$ works. Using $D^{<|y|/2}(x,y)\leq |\kappa|^{-1}\me{-L_0}(x,y)$, we see that $\one_{|x|<R_0}D^{<|y|/2}(x,y)|y|^\gamma\one_{|y|>1}$ generates an $L^2$-bounded operator in view of the first displayed bound in \cite[p.~147]{Giere2004}.

Thus, it remains to prove the $L^2$-boundedness of the operator with integral kernel
\begin{align}
  K(x,y) = \one_{|x|>R_0} \int_0^1 ds \int_{|z|<|y|/2}dz\, \me{-(1-s)L_0}(x,z)|z|^{-\alpha}\me{-s L_0}(z,y)|y|^\gamma \one_{|y|>1}.
\end{align}
To this end, we use a Schur test.
We start with the $\sup_y \int_{\R^d}dx$ Schur test. Using that $\int_{\R^d}\me{-L_0}(x)\,dx=1$, we have, since $1<\gamma<\alpha$,
\begin{align}
  \begin{split}
    \sup_y \int_{\R^d}dx\, K(x,y)
    & \leq \int_0^1 ds \int_{|z|<|y|/2}dz\, \me{-s L_0}(z,y)|z|^{-\alpha}|y|^\gamma \one_{|y|>1} \\
    & \lesssim |y|^{-d-\alpha} \cdot |y|^{d-\alpha} \cdot |y|^\gamma \one_{|y|>1} \lesssim 1.
  \end{split}
\end{align}
We now turn to the $\sup_x\int_{\R^d}dy$ Schur test. Since $1<\gamma<\alpha$,
\begin{align}
  \begin{split}
    & \one_{|x|>R_0}\int_0^1 ds \int_{|y|>1} dy \int_{|z|<|y|/2}dz\, \me{-(1-s)L_0}(x,z)|z|^{-\alpha} |y|^\gamma \me{-s L_0}(z,y) (\one_{|z|\leq1}+\one_{|z|>1}) \\
    & \quad \leq \one_{|x|>R_0}\int_0^1 ds\, s\int_{|z|<1} dz\, \me{-(1-s)L_0}(x,z)|z|^{-\alpha} + C.
  \end{split}
\end{align}
Here, we used that $|z|^{-\alpha}\leq1$ and the pointwise decay the free heat kernel together with $\alpha>\gamma$ to treat the integral over $|z|\geq1$.
Choosing $R_0=2$, we can bound the contribution from the integral over $|z|<1$ from above by
\begin{align}
  \begin{split}
    & \one_{|x|>2} \int_0^1 ds\, s\int_{|z|<1} dz\, \me{-(1-s)L_0}(x,z)|z|^{-\alpha} \\
    & \quad \lesssim \one_{|x|>2} \int_0^1 ds\, s\int_{|z|<1} dz\, \frac{1-s}{((1-s)^{1/\alpha}+|x|)^{d+\alpha}}|z|^{-\alpha} \\
    & \quad \lesssim \int_0^1 ds\, s(1-s) \int_{|z|<1}dz\, |z|^{-\alpha}
    \lesssim 1.
  \end{split}
\end{align}
This concludes the $L^2$-boundedness of $D\langle\cdot\rangle^\gamma$ for $1<\gamma<\alpha$.

\subsubsection{The case $\kappa>0$}

Since $\langle x\rangle^\gamma \sim \one_{|x|<1}+|x|^\gamma\one_{|x|>1}$ and $D(x,y)\leq \kappa^{-1}\me{-L_\kappa}(x,y)$ and $\me{-L_\kappa}$ is $L^2$-bounded, it suffices to check whether $D(x,y)|y|^\gamma\one_{|y|>1}$ generates an $L^2$-bounded operator.
To this end, we split the Duhamel integral in $D(x,y)$ at $|y|/2$. Then,
\begin{align}
  D^{>|y|/2}(x,y)|y|^\gamma
 \lesssim |y|^{\gamma-\alpha} \me{-L_{\kappa}}(x,y).
\end{align}
Since $\me{-L_\kappa}$ is $L^2$-bounded and $1<\gamma<\alpha$, the part $D^{>|y|/2}(x,y) \langle y\rangle^\gamma$ generates an $L^2$-bounded operator. We now show the $L^2$-boundedness of the operator generated by $D^{<|y|/2}(x,y)|y|^\gamma\one_{|y|>1}$. To this end, we split the operator into its contributions from $|x|<R_0$ and $|x|>R_0$ for some $R_0>1/2$ to be chosen below. For instance, $R_0=2$ works.

We start with establishing the $L^2$-boundedness of $\chi_{|x|<R_0}D^{<|y|/2}(x,y)|y|^\gamma\one_{|y|>1}$. For $|x|\in(1/2,R_0)$, the kernel is bounded by $\chi_{|x|<R_0}S_0|y|^\gamma\one_{|y|>1}$, which we already considered in the case $\kappa<0$.
Thus, it remains to consider
\begin{align}
  \begin{split}
    K(x,y)
    & := \one_{|x|<1/2}|x|^{-\sigma}\frac{1}{(1 + |y|)^{d+\alpha}}(|y|^\gamma + |y|^{\gamma-\sigma})\one_{|y|>1} \\
    & \ \lesssim \one_{|x|<1/2}|x|^{-\sigma}\frac{1}{(1 + |y|)^{d+\alpha}}|y|^\gamma\one_{|y|>1}.
  \end{split}
\end{align}
This operator is even Hilbert--Schmidt since $\sigma<d/2$.

Thus, it remains to prove the $L^2$-boundedness of the operator with integral kernel
\begin{align}
  \begin{split}
    & \one_{|x|>R_0} \int_0^1 ds \int_{|z|<|y|/2}dz\, \me{-(1-s)L_\kappa}(x,z)|z|^{-\alpha}\me{-s L_0}(z,y)|y|^\gamma \one_{|y|>1} \\
    & \ \lesssim \one_{|x|>R_0} \int_0^1 ds \int\limits_{|z|<|y|/2}dz\, \me{-(1-s)L_0}(x,z)\left(1+\frac{1}{|z|^\sigma}\right) \frac{1}{|z|^\alpha}\me{-s L_0}(z,y)|y|^\gamma \one_{|y|>1} \\
    & \ =: K(x,y).
  \end{split}
\end{align}
To this end, we use a Schur test.
We start with the $\sup_y \int_{\R^d}dx$ Schur test. Using $1<\gamma\leq\alpha$, $\sigma<d-\alpha$, and $\int_{\R^d}\me{-L_0}(x)\,dx=1$, we have
\begin{align}
  \begin{split}
    & \sup_y \int_{\R^d}dx\, K(x,y) \\
    & \ \lesssim \int_0^1 ds \int_{|z|<|y|/2}dz\, \me{-s L_0}(z,y) (1+|z|^{-\sigma})|z|^{-\alpha}|y|^\gamma \one_{|y|>1} \\
    & \ \lesssim |y|^{-d-\alpha} \cdot |y|^{d-\alpha} \cdot |y|^\gamma \one_{|y|>1}
      \lesssim 1.
  \end{split}
\end{align}

We now turn to the $\sup_x\int_{\R^d}dy$ Schur test. Since $1<\gamma<\alpha$,
\begin{align}
  \begin{split}
    & \one_{|x|>R_0}\int_0^1 ds \int_{|y|>1} dy \int_{|z|<|y|/2}dz\, \me{-(1-s)L_0}(x,z)(1+|z|^{-\sigma})|z|^{-\alpha} |y|^\gamma \me{-s L_0}(z,y) \\
    & \quad \leq \one_{|x|>R_0}\int_0^1 ds\, s \int_{\R^d} dz\, \me{-(1-s)L_0}(x,z)(1+|z|^{-\sigma})|z|^{-\alpha} (1\vee |z|)^{\gamma-\alpha}.
  \end{split}
\end{align}
Thus, with $R_0=2$, we can bound this contribution further from above by
\begin{align}
  \begin{split}
    & \one_{|x|>2} \int_0^1 ds\, s\int_{\R^d} dz\, \me{-(1-s)L_0}(x,z)(1+|z|^{-\sigma})|z|^{-\alpha} (1\vee |z|)^{\gamma-\alpha} \\
    & \quad \lesssim \one_{|x|>2} \int_0^1 ds\, s\int_{|z|<1} dz\, \frac{1-s}{((1-s)^{1/\alpha}+|x|)^{d+\alpha}} |z|^{-\sigma-\alpha} \\
    & \qquad + \one_{|x|>2} \int_0^1 ds\, s\int_{|z|>1} dz\, \me{-(1-s)L_0}(x,z) |z|^{\gamma-2\alpha}.
  \end{split}
\end{align}
The integral over $|z|<1$ gives
\begin{align}
  \begin{split}
    & \one_{|x|>2} \int_0^1 ds\, s\int_{|z|<1} dz\, \frac{1-s}{((1-s)^{1/\alpha}+|x|)^{d+\alpha}} |z|^{-\sigma-\alpha} \\
    & \quad \lesssim \int_0^1 ds\, s(1-s) \int_{|z|<1}dz\, |z|^{-\sigma-\alpha}
      \lesssim 1.
  \end{split}
\end{align}
Next, since $\int\me{-\tau L_0}(x,z)\,dz=1$,
\begin{align}
  \label{eq:ahelp}
  \one_{|x|>2}\int_{|z|>1} dz\ \me{-\tau L_0}(x,z) |z|^{\gamma-2\alpha}
  \lesssim 1, \tau >0
\end{align}
since $|z|^{\gamma-2\alpha}\leq1$.
This concludes the proof of \eqref{eq:ahelp}.


\section{Asymptotic completeness of $W(L_0-V_\kappa - V^S,L_0-V_\kappa)$ and $W(L_0-V_\kappa - V^S,L_0)$ when $\alpha>1$}
\label{s:wlkappavlkappacomplete}

By recycling the arguments of the previous section, we give the

\begin{proof}[{Proof of Theorem~\ref{asympcomplhardyplusperturbation}(\ref{asympcomplhardyplusperturbation:acshortrange})}]
  Let $H:=L_0-V_\kappa-V^S$. By the assumptions on $V_\kappa$ and Trotter's formula,
  \begin{align}
    \label{eq:perturbedheatkernel}
    \me{-t (L_{\tilde \kappa}-V^S)}(x,y)
    \leq \me{-tH}(x,y)
    \leq \me{-t (L_\kappa-V^S)}(x,y).
  \end{align}
  Moreover,
  \begin{align}
    \label{eq:perturbedheatkernel2}
    0\leq \me{-tH}(x,y) \leq \me{t\|(V^S)_+\|_\infty}\me{-tL_\kappa}(x,y).
  \end{align}
  We verify Assumption~\ref{assumheatdifference} with $H_0=L_0$.

  \smallskip
  \noindent
  {\bf On the boundedness assumption.}
  Choose $\gamma\in(1,\alpha\wedge\rho_S)$. By Duhamel's formula,
  $$
  \me{-H}-\me{-L_0} = \int_0^1 \me{-(1-s)H}(V_\kappa+V^S)\me{-sL_0}\,ds.
  $$
  Since $|V_\kappa(x)|\lesssim |x|^{-\alpha}$, the contribution containing $V_\kappa$ is bounded after composing from the right with $\langle \cdot\rangle^\gamma$ by the computations of Subsection~\ref{s:weightedbound} using \eqref{eq:perturbedheatkernel2}.

  For the contribution containing $V^S$, we use $\langle y\rangle^\gamma \lesssim \langle z\rangle^\gamma+|y-z|^\gamma$. Since $\gamma<\rho_S$, we have $\|V^S\langle z\rangle^\gamma\|_\infty<\infty$. Since $\gamma<\alpha$, we have
  $$
  \int_{\R^d}|z|^\gamma\me{-sL_0}(z,0)\,dz \lesssim s^{\gamma/\alpha}.
  $$
  Hence,
  $$
  \|V^S\me{-sL_0}\langle x\rangle^\gamma\| \lesssim 1+s^{\gamma/\alpha},
  \quad 0<s\leq1.
  $$
  Together with the $L^2$-boundedness of $\me{-(1-s)H}$, this yields
  $$
  \|(\me{-H}-\me{-L_0})\langle x\rangle^\gamma\|<\infty.
  $$

  \smallskip
  \noindent
  {\bf On the compactness assumption.}
  By \eqref{eq:perturbedheatkernel2},
  $$
  |\me{-H}(x,y)-\me{-L_0}(x,y)| \lesssim \me{-L_\kappa}(x,y)+\me{-L_0}(x,y).
  $$
  Hence, the arguments of Subsection~\ref{s:compactness} extend to the present situation. Indeed, as before, we can check that $\chi_r(\me{-H}-\me{-L_0})$ is Hilbert--Schmidt. On the other hand, $\chi_r^c(\me{-H}-\me{-L_0})\chi_{r/2}$ tends to zero by the heat-kernel bounds and
  $$
  \|\chi_r^c(\me{-H}-\me{-L_0})\chi_{r/2}^c\| \lesssim r^{-\gamma}
  $$
  follows from the above-proved weighted estimate. Thus, $\me{-H}-\me{-L_0}$ is compact.

  \smallskip
  We can therefore apply Theorem~\ref{giere} and obtain the asymptotic completeness of $W(L_0-V_\kappa-V^S,L_0)$. By the chain rule and Theorem~\ref{asympcomplhardy}(1), we obtain also the asymptotic completeness of $W(L_0-V_\kappa-V^S,L_0-V_\kappa)$.
\end{proof}


\section{Asymptotic completeness of $W(L_\kappa-V^S,L_\kappa)$ for $\alpha>1/3$ \\ using Kato smoothness theory}

\label{s:asympcomplperturbedwosmallalpha_k}

In this section, we show Theorem~\ref{asympcomplhardyplusperturbation}(\ref{asympcomplhardyplusperturbation:acshortrangesmallalpha})(a).
Thus, we consider short-range potentials $V^S:\R^d\to\R$ such that $|V^S(x)|\lesssim\langle x\rangle^{-\rho_S}$ with $\rho_S>1 \vee 2\alpha$. We also assume that
\begin{align}
  W_\kappa(x) := \kappa/|x|^\alpha + V^S
\end{align}
obeys Assumption~\ref{mizutaniyaoconditions} and that there are $\tilde\kappa\leq\kappa'<\kappa_{\rm c}$ such that
\begin{align}
  \frac{\tilde\kappa}{|x|^\alpha}
  \leq W_\kappa(x)
  \leq \frac{\kappa'}{|x|^\alpha};
\end{align}
in particular, $\kappa\leq\kappa'<\kappa_{\rm c}$.
This restriction will be important below to control the heat kernel of $(L_0-W_\kappa)^k \me{-t(L_0-W_\kappa)}$ pointwise for $k\in\{0,1\}$, and powers of $L_0$ by powers of $(L_0-W_\kappa)$. The restriction that our coupling constants need to be subcritical is perhaps technical and due to the use of the Kato smoothing estimates below, which rely on positive commutator estimates. In the following, let $\sigma'=\Phi^{-1}(\kappa')\in (-M,(d-\alpha)/2)$ and $D=-i\nabla$.

\subsection{Kato smoothness theory with heat kernels, proof of Theorem~\ref{asympcomplhardyplusperturbation}(\ref{asympcomplhardyplusperturbation:acshortrangesmallalpha})(a)}
Our proof proceeds by applying Kato's smooth approach to scattering theory (see, e.g., \cite[Section~XIII.7]{ReedSimon1978}), which we recall in the following.
For $j\in\N$, let $\Omega_j\subset\R_+=(0,\infty)$ be bounded and such that $[0,\infty)\setminus\bigcup_{j\in\N}\Omega_j$ has Lebesgue measure zero. Then, the existence and completeness of $W(L_\kappa-V^S,L_\kappa)$ follows once we show
\begin{align}
  \begin{split}
    & \sup_{z\in\C\setminus[0,\infty)} \left(\|(V^S)^{1/2}E_{L_\kappa - V^S}(\Omega_j)\, (L_\kappa - V^S - z)^{-1} E_{L_\kappa - V^S}(\Omega_j)\,|V^S|^{1/2}\| \right. \\ & \qquad\qquad \left. + \|(V^S)^{1/2} E_{L_\kappa}(\Omega_j)\,(L_\kappa-z)^{-1} E_{L_\kappa}(\Omega_j)\,|V^S|^{1/2}\|\right) < \infty;
  \end{split}
\end{align}
here, $V^{1/2}(x):=V(x)/|V(x)|^{1/2}$ whenever $V(x)\neq0$.
By \cite[Theorem~1.3]{MizutaniYao2021}, for all $d\geq2$, $\alpha\in(0,d)$, $s\in(1/2,d/2)$, and $V$ satisfying Assumption~\ref{mizutaniyaoconditions},
\begin{align}
  \label{eq:mizutaniyaosmoothing}
  \sup_{z\in\C\setminus[0,\infty)}\||x|^{-s}|D|^{-s+\alpha/2}(L_0 - V-z)^{-1}|D|^{-s+\alpha/2}|x|^{-s}\|<\infty,
\end{align}
i.e., $|x|^{-s}|D|^{-s+\alpha/2}$ is $(L_0-V)$-supersmooth in the sense of Kato--Yajima \cite{KatoYajima1989}; in the following we merely write ``$H$-smooth'' instead of ``$H$-supersmooth''.
Thus, to show the asymptotic completeness of $W(L_\kappa-V^S,L_\kappa)$, it would suffice to prove, for $\beta:=s-\alpha/2\in(\frac{1-\alpha}{2},\frac{d-\alpha}{2})$,
\begin{align}
  \|(V^S)^{1/2} E_{L_0 - W_\kappa}(\Omega)\,|D|^\beta |x|^{s}\|
  + \|(V^S)^{1/2} E_{L_\kappa}(\Omega)\,|D|^\beta |x|^{s}\|
  \lesssim_{d,\alpha,\kappa,s} 1.
\end{align}

It does not seem to be straightforward to prove this estimate directly as our understanding of the spectral projectors is rather limited. Thus, we attempt to construct an approach using heat kernels, which in our situation are better understood. Before we do so, observe that if $A$ is $H$-smooth and $C$ is bounded, then $CA$ is also $H$-smooth. However, $AC$ is generally not $H$-smooth, unless $C$ commutes with $H$.

\begin{lemma}
  \label{heatkernelsmoothness}
  Let $H_0$ and $H$ be two selfadjoint operators in a Hilbert space $\ch$ that are bounded below, and assume $H_0$ has (for simplicity) purely absolutely continuous spectrum. Assume $H-H_0=A^*B$ where $A$ is $H$-bounded and $A\me{-t_0H}$ is $H$-smooth for every $t_0>0$ and $B$ is $H_0$-bounded and $B\me{-t_0H_0}$ is $H_0$-smooth for every $t_0>0$. Then the wave operators
  \begin{align}
    W(H,H_0) := \slim_{t\to\infty}\me{i t H}\me{-i t H_0}
    \quad \text{and} \quad
    W(H_0,H) := \slim_{t\to\infty}\me{i t H_0}\me{-i t H}P_{\rm ac}(H)
  \end{align}
  exist and are complete.
\end{lemma}

\begin{proof}
  Let $E_H(\cdot)$ and $E_{H_0}(\cdot)$ denote the projection-valued spectral measures for $H$ and $H_0$, respectively.
  Let $\Omega\subset\R$ be an arbitrary bounded Borel set. Thus, for any $t_0>0$, the operator $\me{t_0 H} E_H(\Omega)$ is bounded on $\ch$ and commutes with $H$. We factor
  $$
  A E_H(\Omega) = \left(A \me{-t_0 H}\right)\cdot \left(\me{t_0 H} E_H(\Omega)\right).
  $$
  By assumption, $A\me{-t_0 H}$ is $H$-smooth. Thus, $A E_H(\Omega)$ is $H$-smooth as the product of an $H$-smooth operator and a bounded operator commuting with $H$. Since $\Omega$ is arbitrary, $A$ is locally $H$-smooth. The same argument also shows that $B$ is locally $H_0$-smooth. Thus, by local Kato smoothness theory \cite[Section~XIII.7]{ReedSimon1978}, it follows that $W(H,H_0)$ and $W(H_0,H)$ exist and are complete.
\end{proof}

We are now ready to give the
\begin{proof}[{Proof of Theorem~\ref{asympcomplhardyplusperturbation}(\ref{asympcomplhardyplusperturbation:acshortrangesmallalpha})(a)}]
  We aim to apply Lemma~\ref{heatkernelsmoothness} with $H=L_\kappa-V^S$, $H_0=L_\kappa$, $A=|V^S|^{1/2}$, and $B=-(V^S)^{1/2}$. Let $\beta=s-\alpha/2$ and recall that $G:=|x|^{-s}|D|^{-\beta}$ is supersmooth by \eqref{eq:mizutaniyaosmoothing}. In Lemma~\ref{vhsmoothestimate} below, we will show that, for fixed $t>0$,
  \begin{align}
    \label{eq:acshortrangesmallalphasufficient}
    C := |V^S|^{1/2}\me{-tH}|D|^\beta|x|^s \in \mathcal B(L^2(\R^d)).
  \end{align}
  Therefore, $CG = |V^S|^{1/2}\me{-tH}$ is $H$-smooth as the product of a bounded operator and a $H$-smooth operator.
  Analogously, $(V^S)^{1/2}\me{-tH_0}$ is $H_0$-smooth. Applying Lemma~\ref{heatkernelsmoothness} shows existence and completeness of $W(H,H_0)$. Moreover, $G=|x|^{-s}|D|^{-\beta}$ has trivial kernel. By the standard spectral consequence of Kato smoothness, $\overline{\ran G^*}\subset\mathcal H_{\rm ac}(H)$. Since $\ker G=\{0\}$, we have $\overline{\ran G^*}=(\ker G)^\perp=\mathcal H$, and therefore $H$ has purely absolutely continuous spectrum. Therefore, $W(H,H_0)$ is also unitary. Since $H_0$ has purely absolutely continuous spectrum $[0,\infty)$, the intertwining property yields that also $H$ has purely absolutely continuous spectrum $[0,\infty)$.
\end{proof}

The rest of this section is dedicated to proving \eqref{eq:acshortrangesmallalphasufficient}. This will require us to be able to find $s\in\R$ satisfying
\begin{align}
  \begin{split}
    \max\left\{\frac12 , \alpha\right\} < s < \min\left\{\frac{\rho_S}{2},\frac{d}{2},1,\frac{3\alpha}{2},\frac{d+\alpha-\sigma'}{2} ,\frac{d+\alpha-2\sigma'}{2} \right\}, \quad \alpha<1, \\
    \frac12 < s < \min\left\{\frac{\rho_S}{2},\frac{d}{2},1,\frac{d+1-\sigma'}{2} ,\frac{d+1-2\sigma'}{2} \right\}, \quad \alpha=1
  \end{split}
\end{align}
as noted in the introduction.


\subsection{Technical preparations for the proof of (\ref{eq:acshortrangesmallalphasufficient})}

To prove \eqref{eq:acshortrangesmallalphasufficient}, we prepare two technical tools. The first one states that powers of $L_0-W_\kappa$ generate $L^2$-Sobolev norms that are comparable with the ordinary homogeneous Sobolev norms. These are also crucial later to prove a phase space propagation estimate. For $\alpha=2$, it is proved in \cite[Theorem~1.2]{Killipetal2018}, while $\alpha<2$ is treated in \cite[Theorems~1.1, 4.1]{Franketal2021}.

\begin{proposition}[Equivalence of Sobolev norms]
  \label{eqsob}
  Let $d\in\N$, $\alpha\in(0,2] \cap (0,d)$, $\tilde\kappa\leq\kappa\leq\kappa_{\rm c}$ and let $\sigma=\Phi^{-1}(\kappa)\in(-M,(d-\alpha)/2]$. Let $\tilde\kappa/|x|^\alpha \leq V(x) \leq \kappa/|x|^\alpha$ and $s\in(0,2]$.
  \begin{enumerate}
  \item If $s<\frac{d-2\sigma}{\alpha}$, then
    \begin{align}
      \label{eq:eqsob1}
      \|(L_0)^{s/2} f\|_{2}
      \lesssim_{d,\alpha,\kappa,s} \| (L_0 - V)^{\frac{s}{2}}f\|_{2}
      \quad \text{ for all } f\in C_c^\infty(\R^d).
    \end{align}
  \item If $s<\frac{d}{\alpha}$, then
    \begin{align}
      \label{eq:eqsob2}
      \|(L_0 - V)^{\frac{s}{2}}f\|_{2}
      \lesssim_{d,\alpha,\kappa,s} \| (L_0)^{s/2}f\|_{2}
      \quad \text{for all}\ f\in C_c^\infty(\R^d).
    \end{align}
  \end{enumerate}
\end{proposition}

\begin{remark}
  \begin{enumerate}
  \item In \cite[Theorem~1.2]{Killipetal2018}, only the case $V(x)=\kappa/|x|^2$ is treated. However, the same arguments as in the proof of \cite[Theorem~4.1]{Franketal2021} allow us to cover all potentials mentioned in Proposition~\ref{eqsob}.

  \item As noted in \cite[Theorem~A.1]{Franketal2018} (the latest arXiv version of \cite{Franketal2021}), we have that if $s<(d-2\sigma)/\alpha$, then $\D((L_\kappa)^{s/2}) \subset \D((L_0)^{s/2})$ and \eqref{eq:eqsob1} holds for all $f\in \D((L_\kappa)^{s/2})$. Moreover $C_c^\infty(\R^d)$ is an operator core for $(L_\kappa)^{s/2}$. Similarly, if $s<\frac{d}{\alpha}$, then $\D((L_0)^{s/2}) \subset \D((L_\kappa)^{s/2})$ and \eqref{eq:eqsob2} holds for all $f\in\D((L_0)^{s/2})$. Moreover, $C_c^\infty(\R^d)$ is an operator core for $(L_0)^{s/2}$.
  \end{enumerate}
\end{remark}

\begin{remark}
  A statement analogous to that in Proposition~\ref{eqsob} is proved for the $L^2$-Sobolev norms generated by powers of $(-\Delta)^{\alpha/2}|_{\R_+^d}-\kappa/x_d^\alpha$ on the half-space $\R_+^d=\R^{d-1}\times(0,\infty)$ in \cite{FrankMerz2023}, where $(-\Delta)^{\alpha/2}|_{\R_+^d}$ means the regional fractional Laplacian \cite{Bogdanetal2003}. Moreover, the equivalence of the corresponding $L^p$-Sobolev norms has been proved for these operators in \cite{Merz2021,BuiDAncona2023,BuiMerz2026}.
\end{remark}

The second tool is reminiscent of Kato--Ponce theory and is used to treat $\alpha=1$ in Lemma~\ref{vhsmoothestimate}.

\begin{lemma}
  \label{calderoncommutator}
  Let $d\geq 1$ and $0<s\leq 1$. Assume, if $s<1$, that $1-s < d/2$. Then,
  \begin{align}
    \label{eq:commutator-weight}
    \| [|D|,|x|^s]g \|_2 \lesssim_{d,s} \| |D|^{1-s}g \|_{2}.
  \end{align}
  At the endpoint $s=1$, this is understood as $\| [|D|,|x|]g \|_2 \lesssim_d \|g\|_2$.

  In particular, if $d=1$ and $1/2<s\leq1$ or $d\in\{2,3,...\}$ and $0<s<1$, then
  $$
  \| [|D|,|x|^s]g \|_2 \lesssim_{d,s} \| |D|^{1-s}g \|_2+\|g\|_2.
  $$
\end{lemma}

We defer the proof of this lemma to Appendix~\ref{s:calderoncommutator}.


\subsection{Proof of (\ref{eq:acshortrangesmallalphasufficient})}

Let us first prove \eqref{eq:acshortrangesmallalphasufficient} in the case where $|D|^\beta$ is absent.

\begin{lemma}
  \label{vhsmoothestimateaux}
  Let $d\in\N$, $\alpha\in(0,d)\cap(0,2)$, $\kappa\leq\kappa_{\rm c}$, $\sigma=\Phi^{-1}(\kappa)$, and $t>0$.
  Assume there is $\tilde\kappa\leq\kappa$ such that $\tilde\kappa/|x|^\alpha \leq V(x)\leq \kappa/|x|^\alpha$.
  \begin{enumerate}
  \item Let $0\leq s<d/2-\sigma$ and $0\leq\gamma\leq s$.
    Then,
    \begin{align}
      \label{eq:vhsmoothestimateaux2}
      \||x|^{-s} \me{-t(L_0-V)} |x|^\gamma\| \lesssim_{d,\alpha,\kappa,s,\gamma} t^{(\gamma-s)/\alpha}.
    \end{align}
    
  \item Let $0\leq \gamma\leq s$. Assume either $s<d/2-\sigma$, or, when $s\geq d/2-\sigma$ also $\gamma+s \leq d+\alpha-\sigma$. Then, for every $t>0$, there is $C_t>0$ such that
    \begin{align}
      \label{eq:vhsmoothestimateaux4}
      \|\langle x\rangle^{-s} \me{-t(L_0-V)} |x|^\gamma\|
      \leq C_t.
    \end{align}  
  \end{enumerate}
\end{lemma}

All identities in the following proof involving unbounded spatial weights or fractional powers of $|D|$ are understood first on a common dense core (for instance $C_c^\infty(\R^d\setminus\{0\})$) and, when necessary, after bounded regularization of the weights and Fourier multipliers. The estimates below are uniform in these regularizations; hence the displayed factorizations pass to the corresponding closed operators by density and closedness. We suppress these standard regularizations in the notation.

\begin{proof}
  By Trotter's formula, $\me{-t(L_0-V)}(x,y)\leq\me{-t L_\kappa}(x,y)$.
  We begin with the proof of Item~(1).
  By scaling, it suffices to consider $t=1$.
  We use Schur tests for the four regions $|x|,|y|\leq1$, $|x|\leq1\leq|y|$, $|y|\leq1\leq|x|$, and $|x|,|y|\geq1$.
  
  \smallskip
  {\bf Case $|x|,|y|\leq1$:}
  The relevant kernel is $|x|^{-s-\sigma}|y|^{\gamma-\sigma}$. We use Schur tests with weight $(|x|/|y|)^{s+\sigma}$. Then,
  \begin{align}
    |x|^{-s-\sigma}\int_{|y|\leq1}|y|^{\gamma-\sigma} \left(\frac{|x|}{|y|}\right)^{s+\sigma}
    \lesssim1, \\
    |y|^{\gamma+s}\int_{|x|\leq1}dx\, |x|^{-2(s+\sigma)}
    \lesssim 1,
  \end{align}
  where we used $s-\gamma<d-2\sigma$ (which is satisfied since $s<d/2-\sigma$ and $\gamma\geq 0$) in the first Schur test, and $s+\gamma\geq0$ and $s<d/2-\sigma$ in the second one.
  Had we not worked with the potential $|x|^{-s}$ but rather with $\langle x\rangle^{-s}$, then this Schur test would have worked with the weight $(|x|/|y|)^\sigma$, and we would not encounter the restriction $s<d/2-\sigma$.
  
  \smallskip
  {\bf Case $|x|\leq1\leq|y|$:}
  The relevant kernel is $|x|^{-s-\sigma}|y|^\gamma(1+|x-y|)^{-d-\alpha}$. We use Schur tests with weight $(|x|/|y|)^{s+\sigma}$. Then,
  \begin{align}
    \int_{|y|>1}dy\, |y|^{-\sigma+\gamma-s}(1+|x-y|)^{-d-\alpha} \lesssim 1,
  \end{align}
  where we used $s-\gamma>-\alpha-\sigma$ (which is satisfied because $\gamma\leq s$ and $\sigma>-\alpha$), and
  \begin{align}
    \begin{split}
      & |y|^{s+\gamma+\sigma}\int_{|x|\leq1}dx\, |x|^{-2(s+\sigma)}(1+|x-y|)^{-d-\alpha}\left(\one_{|x|\leq1/2}+\one_{1/2\leq|x|\leq1}(\one_{|y|\in[1,2]}+\one_{|y|>2})\right) \\
      & \quad \lesssim |y|^{s+\gamma +\sigma-d-\alpha} + 1
        \lesssim 1.
    \end{split}
  \end{align}
  where we used $s<d/2-\sigma$ and $s+\gamma\leq d+\alpha-\sigma$ (which is satisfied since $s+\gamma\leq d-2\sigma<d+\alpha-\sigma$).

  \smallskip
  {\bf Case $|y|\leq1\leq|x|$:}
  The relevant kernel is $|x|^{-s}|y|^{\gamma-\sigma}(1+|x-y|)^{-d-\alpha}$. We use Schur tests with weight $(|x|/|y|)^{-\gamma+\sigma}$. Then,
  \begin{align}
    \begin{split}
      |x|^{\sigma-s-\gamma} \int_{|y|\leq1}dy\, |y|^{2\gamma-2\sigma}(1+|x-y|)^{-d-\alpha}(\one_{|x|\in[1,2]}+\one_{|x|\geq2})
      \lesssim1,
    \end{split}
  \end{align}
  where we used $\sigma-s-\gamma \leq d+\alpha$ and $2\gamma>2\sigma-d$ (which is satisfied since $-\alpha<\sigma\leq(d-\alpha)/2$ and $s,\gamma\geq0$), and
  \begin{align}
    \begin{split}
      \int_{|x|\geq1}dx\, |x|^{-\sigma+\gamma-s}(1+|x-y|)^{-d-\alpha} \lesssim1
    \end{split}
  \end{align}
  since $s-\gamma>-\alpha-\sigma$ (which is satisfied because $\gamma\leq s$ and $\sigma>-\alpha$).

  \smallskip
  {\bf Case $|x|,|y|\geq1$:}
  The relevant kernel is $|x|^{-s}|y|^{\gamma}(1+|x-y|)^{-d-\alpha}$. We use Schur tests with weight $(|x|/|y|)^\delta$ with $\gamma-\alpha<\delta<d+\alpha-s$; such a $\delta$ exists because $\gamma+s<d-2\sigma<d+2\alpha$, which follows from $\sigma>-\alpha$. Then
  \begin{align}
    \begin{split}
      & |x|^{\delta-s}\int_{|y|>1}dy\,(1+|x-y|)^{-d-\alpha}|y|^{\gamma-\delta} \left(\one_{|y|\leq|x|/2}+\one_{|x|/2\leq|y|\leq2|x|}+\one_{|y|\geq2|x|}\right) \\
      & \quad \lesssim |x|^{\delta-s} \cdot |x|^{\gamma-\delta} \lesssim 1
    \end{split}
  \end{align}
  since $s-\gamma\geq0>-\alpha$, and
  \begin{align}
    \begin{split}
      & |y|^{\delta+\gamma} \int_{|x|\geq1}dx\, (1+|x-y|)^{-d-\alpha}|x|^{-s-\delta}\left(\one_{|x|\leq|y|/2}+\one_{|y|/2\leq |x|\leq2|y|} + \one_{|x|\geq2|y|}\right) \\
      & \quad \lesssim |y|^{\delta+\gamma} ( |y|^{-\alpha-s-\delta}+|y|^{-s-\delta})
        \lesssim 1.
    \end{split}
  \end{align}
  This concludes the proof of Item~(1).

  \smallskip
  We now prove Item~(2). When $s<d/2-\sigma$, this follows from the previous Schur tests and $\|\langle x\rangle^{-s}\me{-t L_\kappa}|x|^\gamma\| \leq \||x|^{-s}\me{-t L_\kappa}|x|^\gamma\|$. When $s\geq d/2-\sigma$, we re-inspect the Schur test in the region $|x|\leq t^{1/\alpha}$, which is the only region where this restriction was used.
  When also $|y|\leq t^{1/\alpha}$, we use the weight $(|x|/|y|)^\sigma$ and get
  \begin{align}
    t^{2\sigma/\alpha-d/\alpha} \langle x\rangle^{-s}\int_{|y|<t^{1/\alpha}}|y|^{\gamma-2\sigma}
    & \lesssim t^{\gamma/\alpha}, \\
    t^{2\sigma/\alpha-d/\alpha} |y|^{\gamma} \int_{|x|<t^{1/\alpha}}\langle x\rangle^{-s}|x|^{-2\sigma}
    & \lesssim t^{\gamma/\alpha}.
  \end{align}
  where we used $\sigma\leq(d-\alpha)/2<d/2$.
  In the region $|x|\leq t^{1/\alpha}\leq|y|$ we use the weight $(|x|/|y|)^{\sigma}(\langle x_t\rangle/\langle y_t\rangle)^s$ with $x_t=x/t^{1/\alpha}$ and $y_t=y/t^{1/\alpha}$. The $\sup_{|x_t|<1}\int_{|y_t|>1}dy$-Schur test is the same, using $\gamma\leq s$. The other Schur test is
  \begin{align}
    \begin{split}
      & t^{1+d/\alpha}|y_t|^{s+\gamma+\sigma}\int_{|x_t|\leq1}dx\, \langle x_t\rangle^{-s}\langle x\rangle^{-s} |x_t|^{-2\sigma}(t^{1/\alpha}+|x-y|)^{-d-\alpha} \\
      & \quad \lesssim_t |y_t|^{s+\gamma+\sigma-d-\alpha}+1 \lesssim1, \quad |y_t|>1,
    \end{split}
  \end{align}
  where now we just used $s+\gamma\leq d+\alpha-\sigma$.
  This concludes the proof.
\end{proof}

We now repeat these estimates, but include $|D|^{s-\alpha/2}$ to conclude the proof of~\eqref{eq:acshortrangesmallalphasufficient}.

\begin{lemma}
  \label{vhsmoothestimate}
  Let $d\in\N$, $\alpha\in(0,d)\cap(0,1]$, $\kappa<\kappa_{\rm c}$, $\sigma=\Phi^{-1}(\kappa)$, and $t>0$.
  Assume there is $\tilde\kappa\leq\kappa$ such that $\tilde\kappa/|x|^\alpha \leq V(x)\leq \kappa/|x|^\alpha$.
  Assume also that either
  \begin{align*}
    \alpha < s \leq \min\{1,3\alpha/2\}
  \end{align*}
  or
  \begin{align*}
    \alpha=1 \quad \text{and} \quad \frac12<s\leq1.
  \end{align*}

  \begin{enumerate}
  \item Let $s<d/2-\sigma$. Then,
    \begin{align}
      \label{eq:vhsmoothestimate1}
      \||x|^{-s}\me{-t(L_0-V)}|D|^{s-\alpha/2}|x|^s\| \lesssim_{d,\alpha,\kappa,s} t^{-(s-\alpha/2)/\alpha}.
    \end{align}
    
  \item Assume
    \begin{align}
      \label{eq:vhsmooth-sobolev-condition}
      s < (d + \alpha - 2\sigma)/2, \\
      \label{eq:vhsmooth-kernel-condition}
      s < (d + \alpha - \sigma)/2.
    \end{align}
    Then, for every $t>0$, there is $C_t>0$ such that
    \begin{align}
      \label{eq:vhsmoothestimate2}
      \|\langle x\rangle^{-s}\me{-t(L_0-V)}|D|^{s-\alpha/2}|x|^s\|
      \lesssim_{d,\alpha,\kappa,s} C_t.
    \end{align}
    In particular, \eqref{eq:acshortrangesmallalphasufficient} holds.
  \end{enumerate}
\end{lemma}

The assumption \eqref{eq:vhsmooth-sobolev-condition} is used to apply the equivalence of Sobolev norms result (Proposition~\ref{eqsob}), while \eqref{eq:vhsmooth-kernel-condition} controls weighted heat kernels (cf.~Lemma~\ref{vhsmoothestimateaux}(2)).

\begin{proof}[Proof of Lemma~\ref{vhsmoothestimate}]
  {\bf Case $\alpha<1$:}
  We first show Item (1). By scaling, it suffices to consider $t=1$. Let $\beta = s - \alpha/2$. By our assumptions, $\beta \in (\alpha/2, \min\{1,3\alpha/2\}-\alpha/2]$.
  We consider
  \begin{align}
    T := |x|^s |D|^\beta \me{-(L_0-V)}|x|^{-s}.
  \end{align}
  We commute $|D|^\beta$ to the left and write $T = T_1 - T_2$, where
  \begin{align}
    T_1 & := |D|^\beta|x|^s\me{-(L_0-V)}|x|^{-s}, \\
    T_2 & := [|D|^\beta, |x|^s] \me{-(L_0-V)}|x|^{-s}.
  \end{align}

  \smallskip
  \noindent {\bf Bound for $T_2$:}
  The integral kernel of $[|D|^\beta, |x|^s]$ is bounded by $c_{d,\beta} \frac{||x|^s - |y|^s|}{|x-y|^{d+\beta}}$. Since $s \in (0, 1]$, we can use $||x|^s - |y|^s| \leq |x-y|^s$. Thus, the kernel is bounded pointwise by a multiple of $|x-y|^{-(d+\beta-s)}$. Since $\beta - s = -\alpha/2$, this is the convolution kernel of the Riesz potential $|D|^{-\alpha/2}$. Because the constituent operators possess positive integral kernels, we have the pointwise bound
  \begin{align}
    |T_2 f|
    \lesssim |D|^{-\alpha/2}|x|^{-\alpha/2} \left(|x|^{\alpha/2} \me{-(L_0-V)}|x|^{-s} |f| \right).
  \end{align}
  Therefore,
  \begin{align}
    \|T_2 f\|_2
    \lesssim \left\| |x|^{\alpha/2} \me{-(L_0-V)} |x|^{-s} |f| \right\|_2
    \lesssim \|f\|_2
  \end{align}
  by Hardy's inequality and \eqref{eq:vhsmoothestimateaux2} in Lemma~\ref{vhsmoothestimateaux} (applied with $\gamma = \alpha/2$ there), which requires $\alpha/2 \leq s$ and $s < d/2 - \sigma$; this is possible due to our assumption $\sigma < (d-\alpha)/2$. Thus, $\|T_2\| \lesssim 1$.

  \smallskip\noindent {\bf Bound for $T_1$:}
  To estimate $T_1$, we use the equivalence of Sobolev norms in Proposition~\ref{eqsob}, i.e., $|D|^{2\beta} \lesssim (L_0-V)^{2\beta/\alpha}$ whenever $\beta < (d-2\sigma)/2$ and $\beta\leq \alpha$. The assumption $\beta\leq\alpha$ is ensured by $s\leq3\alpha/2$.
  The assumption $\beta < (d-2\sigma)/2$ is equivalent to $s<(d-2\sigma+\alpha)/2$; this is justified here since $s<d/2-\sigma$.
  Thus, by Proposition~\ref{eqsob},
  \begin{align}
    \begin{split}
      \|T_1 f\|_2
      & = \left\| |D|^\beta \left( |x|^s \me{-(L_0-V)}|x|^{-s} f \right) \right\|_2 \\
      & \lesssim \left\| (L_0-V)^{\beta/\alpha} \left(|x|^s \me{-(L_0-V)}|x|^{-s} f\right) \right\|_2.
    \end{split}
  \end{align}
  Next, by the spectral theorem with $\beta\leq\alpha$,
  $$
  \| (L_0-V)^{\beta/\alpha} u \|_2
  \leq \| u \|_2^{1-\beta/\alpha} \| (L_0-V) u \|_2^{\beta/\alpha}
  \lesssim \|u\|_2 + \| (L_0-V) u\|_2.
  $$
  Thus, letting
  $$
  B = |x|^s\me{-(L_0-V)}|x|^{-s},
  $$
  it suffices to prove that $\|(L_0-V) B\| \lesssim 1$ since $\|B\| \lesssim 1$ follows from Lemma~\ref{vhsmoothestimateaux}.

  We now show $\|(L_0-V) B\| \lesssim 1$ by commuting $L_0-V$ next to $\me{-(L_0-V)}$. We have $(L_0-V) (|x|^s u) = [ (-\Delta)^{\alpha/2}, |x|^s ] u + |x|^s (L_0-V) u$. Letting $u = \me{-(L_0-V)} |x|^{-s} f$, we split $(L_0-V) B = W + C$, where
  \begin{align}
    W f & = [ (-\Delta)^{\alpha/2}, |x|^s ] \me{-(L_0-V)} |x|^{-s} f, \\
    C f & = |x|^s (L_0-V) \me{-(L_0-V)} |x|^{-s} f.
  \end{align}
  To bound $C$, we use that for every $\epsilon>0$, we have
  \begin{align}
    \label{eq:derivativeheatkernel}
    \left| (L_0-V) \me{-t(L_0-V)}(x,y) \right|
    \lesssim_\epsilon t^{-1-\frac{d}{\alpha}} \left(1\wedge\frac{|x|}{t^{1/\alpha}}\right)^{-\sigma} \left(1\wedge\frac{|y|}{t^{1/\alpha}}\right)^{-\sigma} \left(1+\frac{|x-y|}{t^{1/\alpha}}\right)^{-d-\alpha+\epsilon}
  \end{align}
  by the bounds for $\me{-t(L_0-V)}$, Cauchy's integral formula, and a Phragm\'en--Lindel\"of argument \cite{DuongRobinson1996,Merz2022}; we give a full proof in Proposition~\ref{thm-ptk}.
  Accepting \eqref{eq:derivativeheatkernel} and taking $\epsilon<\min\{\alpha,\alpha+\sigma\}$, we can argue exactly as in the proof of Lemma~\ref{vhsmoothestimateaux} to infer $\|C\|\lesssim1$ whenever $s<d/2-\sigma$.
    
  To bound $W$, we note that the kernel of $[ (-\Delta)^{\alpha/2}, |x|^s ]$ is bounded by $|x-y|^{-d-\alpha+s}$, which is the kernel of the Riesz potential $|D|^{\alpha-s}$ since $\alpha<s$. Applying the pointwise bound and Hardy's inequality (using $s\leq3\alpha/2$) yields
  \begin{align}
    \begin{split}
      \| W f \|_2
      & \lesssim \left\| |D|^{\alpha-s} \left( \me{-(L_0-V)} |x|^{-s} |f| \right) \right\|_2 \\
      & \lesssim \left\| |x|^{s-\alpha} \me{-(L_0-V)} |x|^{-s} |f| \right\|_2
        \lesssim \|f\|_2
    \end{split}
  \end{align}
  by \eqref{eq:vhsmoothestimateaux2} in Lemma~\ref{vhsmoothestimateaux} (with $\gamma=s-\alpha$ there).
  Thus, we have proved $\|(L_0-V) B\| \lesssim 1$, which yields $\|T_1\| \lesssim 1$. Combining the bounds for $T_1$ and $T_2$ yields $\|T\| \lesssim 1$. This concludes the proof of Item~(1).

  \smallskip 
  The proof of Item~(2) is analogous, but uses \eqref{eq:vhsmoothestimateaux4} in Lemma~\ref{vhsmoothestimateaux}. Let us give the details. Here we define
  \begin{align}
    T & := |x|^s|D|^\beta \me{-(L_0-V)}\langle x\rangle^{-s}, \\
    B & := |x|^s \me{-(L_0-V)}\langle x\rangle^{-s}.
  \end{align}
  First, let us recall the decomposition of $T=T_1-T_2$. The operator $T_2$ is bounded as before using \eqref{eq:vhsmoothestimateaux4} in Lemma~\ref{vhsmoothestimateaux} (with $\gamma=\alpha/2$ there) and the analysis of $T_1$ is reduced to that of $B$ using Proposition~\ref{eqsob}, which requires $s<(d+\alpha-2\sigma)/2$.
  Using \eqref{eq:vhsmoothestimateaux4} in Lemma~\ref{vhsmoothestimateaux} (with $\gamma=s$ there), we have $\|B\|\lesssim1$ when $s<(d+\alpha-\sigma)/2$. Moreover, we have $\||x|^s(L_0-V)\me{-(L_0-V)}\langle x\rangle^{-s}\|\lesssim1$ (the analog of $\|C\|\lesssim1$) when $s<(d+\alpha-\sigma)/2$; here we need to choose $\epsilon<\min\{\alpha,\alpha+\sigma,d+\alpha-\sigma-2s\}$ in \eqref{eq:derivativeheatkernel}. Finally, we have $\|[(-\Delta)^{\alpha/2},|x|^s]\me{-(L_0-V)}\langle x\rangle^{-s}\|\lesssim1$ using \eqref{eq:vhsmoothestimateaux4} in Lemma~\ref{vhsmoothestimateaux} (with $\gamma=s-\alpha$ there) when $s<(d+2\alpha-\sigma)/2$.

  \smallskip
  {\bf Case $\alpha=1$:} From the above proof for $\alpha<1$, it transpires that we only need to re-verify the boundedness of $W$, where we used $s>\alpha$ when $\alpha<1$. Let us show that when $\alpha=1$, it suffices to assume $s>1/2$.
  We apply Lemma~\ref{calderoncommutator} with $g=\me{-(L_0-V)}\omega^{-s}f$, where $\omega(x)=|x|$ in Item~(1) and $\omega(x)=\langle x\rangle$ in Item~(2).

  Since $\sigma<(d-1)/2$ and $s>1/2$, we have $2(1-s)<1<d-2\sigma$. Hence, Proposition~\ref{eqsob}, applied with exponent $2(1-s)$, gives
  $$
  \| |D|^{1-s}g\|_2 \lesssim \|(L_0-V)^{1-s}g\|_2.
  $$
  Therefore, by the spectral theorem,
  \begin{align*}
    \|(L_0-V)^{1-s}\me{-(L_0-V)} \omega^{-s}f\|_2
    \leq \|(L_0-V)^{1-s}\me{-(L_0-V)/2}\| \|\me{-(L_0-V)/2} \omega^{-s}f\|_2
    \lesssim \|f\|_2;
  \end{align*}
  here we used \eqref{eq:vhsmoothestimateaux2} in Lemma~\ref{vhsmoothestimateaux} (with $\gamma=0$ there) when $\omega(x)=|x|$ for Item~(1) and $\langle x\rangle^{-s}\leq1$ for Item~(2). Thus,
  $$
  \|[|D|,|x|^s]\me{-(L_0-V)}|x|^{-s}f\|_2
  + \|[|D|,|x|^s]\me{-(L_0-V)}\langle x\rangle^{-s}f\|_2
  \lesssim \|f\|_2.
  $$
  This shows the boundedness of $W$ and concludes the proof for $\alpha=1$ since the estimates of $T_2$ and $C$, as well as the final interpolation argument for $T_1$ are the same.
\end{proof}

\begin{remark}
  \label{rem:technicaldescription}
  The above proof is quite cumbersome. We first commuted $|D|^\beta$ out, then used the equivalence of Sobolev norms, only to commute $L_\kappa$ back to be next to $\me{-t L_\kappa}$. The reason for doing so is that we do not know how to control $|D|^\beta\me{-t L_\kappa}$. This leads to the technical assumption $s\leq 3\alpha/2$. (Note that a similar technical problem arises in \cite{Buietal2026}, where derivative bounds for heat kernels would also be beneficial.)
  The assumption $s>\alpha$ ensures that $[|D|^\alpha,|x|^s]$ is smoothing. As we explained above, this back-and-forth commuting is likely not necessary, hence also this assumption is likely technical, too.
  However, the assumption $\sigma<(d-\alpha)/2$ is not too surprising, because of the zero-energy resonance of $L_\kappa$, which is particularly severe when $\alpha\leq1$.
\end{remark}

\begin{remark}
  Since for $\alpha\in(1,d)$ we also have, according to \cite[Theorem~1.3]{MizutaniYao2021},
  \begin{align}
    \label{eq:mizutaniyaosmoothingshortrange}
    \sup_{z\in\C\setminus[0,\infty)}\||x|^{-\alpha/2}(L_0 - V - z)^{-1}|x|^{-\alpha/2}\| < \infty,
  \end{align}
  Kato's smoothness approach also yields another proof of asymptotic completeness of $W(L_\kappa,L_0)$ for $\alpha\in(1,2]$ when $\kappa<\kappa_{\rm c}$. The advantage of our above proof using the semigroup framework is that we can also treat $\kappa=\kappa_{\rm c}$.
\end{remark}


\section{Preparations for propagation estimates}
\label{s:technicalpreparations}

This section serves as a technical preparation for proving propagation estimates for $H_\kappa=L_\kappa-V^S$.
In the first two subsections, we recall commutator classes, Helffer--Sj\"ostrand formulae, and commutator expansions.
In the third subsection, we prepare properties of the generator of dilations $A = \frac{1}{2} (D \cdot x + x \cdot D)$.


\subsection{Commutator classes}

The following notation will be helpful in the ensuing commutator expansions.
Let $A$ be selfadjoint on a Hilbert space $\mathcal H$. For $B\in\mathcal B(\mathcal H)$, we write
$$
B\in C^1(A)
$$
if the map
$$
\R\ni t\mapsto \me{-itA}B\me{itA}
$$
is strongly continuously differentiable. Equivalently, $B\mathcal D(A)\subset\mathcal D(A)$ and the commutator
$$
[B,A]=BA-AB,
\qquad
\mathcal D([B,A])=\mathcal D(A),
$$
extends to a bounded operator on $\mathcal H$.

More generally, for $k\in\N$, we write $B\in C^k(A)$ if
$$
\R\ni t\mapsto\me{-itA}B\me{itA}
$$
is strongly $C^k$.

If $H$ is selfadjoint (not necessarily bounded), we write $H\in C^k(A)$ if, for one (and hence every) $z\in\rho(H)$ (the resolvent set of $H$),
$$
(z-H)^{-1}\in C^k(A).
$$
Whenever the commutator $[H,A]$ is represented by an operator for which the following products are meaningful, one has
$$
[(z-H)^{-1},A] = (z-H)^{-1}[H,A](z-H)^{-1}.
$$


\subsection{Helffer--Sj\"ostrand formulae and commutator expansions}
\label{ss:functionalcalculus}

We recall the Helffer--Sj\"ostrand formula \cite[Proposition~7.2]{HelfferSjostrand1989}; see also the textbooks \cite[Section~2.2]{Davies1996}, \cite[Theorem~8.1]{DimassiSjostrand1999}, \cite[Proposition~C.2.1]{DerezinskiGerard1997}, and \cite[Theorem~1.17]{Isozaki2023}. It allows us to control functions of selfadjoint operators in terms of their resolvents. To state the formula, we recall the notation
\begin{align}
  \bar\partial
  = \frac{\partial}{\partial\overline z}
  = \frac12\left(\frac{\partial}{\partial x} + i \frac{\partial}{\partial y}\right),
  \qquad
  z=x+iy,
\end{align}
and $dz\wedge d\overline z=-2i dx\,dy$, which is $-2i$ times the Lebesgue measure on $\C$.

\begin{proposition}
  \label{prop:analyticextension}
  Let $f\in C_c^\infty(\R)$. Then there exists $\tilde f\in C_c^\infty(\C)$, called an {\em almost analytic extension of $f$} such that
  $$
  \tilde f\big|_{\R} = f
  $$
  and for every $N\in\N$ there is a constant $C_N>0$ such that
  $$
  \left| \frac{\partial \tilde f}{\partial\overline z}(z) \right|
  \leq C_N\, |\im(z)|^{N},  \quad z\in\C.
  $$
  Moreover, for every selfadjoint operator $A$ and every such $f$, we have the Helffer--Sj\"ostrand formula
  \begin{align}
    \label{eq:hs}
    f(A) = \frac{1}{2\pi i} \int_{\C}\frac{\partial\tilde f}{\partial \overline z}(z)(z-A)^{-1} \, dz\wedge d\overline z,
  \end{align}
  where the integral converges absolutely in operator norm.
  More generally, for $n\in\N$,
  \begin{align}
    \label{eq:hsderivative}
    f^{(n)}(A) = \frac{n!}{2\pi i}\int_{\C}\frac{\partial\tilde f}{\partial \overline z}(z)(z-A)^{-n-1}\, dz\,\wedge d\overline z.
  \end{align}
\end{proposition}

\begin{remark}
  One construction of an almost analytic extension of $f\in C^\infty(\R:\,\C)$ of order $N$ is
  $$
  \tilde f(z) := \left(\sum_{r=0}^N f^{(r)}(x)(iy)^r/r!\right) \eta\left(\frac{y}{\langle x\rangle}\right),
  $$
  where $\eta \in C_c^\infty(\R;[0,1])$ satisfies $\eta(s)=1$ for $|s|\leq1$, $\eta(s)=0$ for $|s|>2$, and is smooth in between; see \cite{Davies1996,DimassiSjostrand1999}. The precise values of $N$ and choice of $\eta$ are not relevant.
  It is sometimes useful to indicate the order $N$ of the above construction, in which case we write more explicitly $\widetilde f_N(z)$ instead of just $\tilde f(z)$.
  Note that $\tilde f(z)$ is smooth as a function of $x$ and $y$, but not analytic, and that $\tilde f(z)|_\R = f(x)$. It follows from this construction that
  $$
  \left|\frac{\partial \tilde f}{\partial\overline z}\right|
  = \mathcal{O}(|\im(z)|^N) \quad \text{as} \ y \to 0 \ \text{for all} \ x\in\R.
  $$
  In particular,
  $$
  \frac{\partial \tilde f}{\partial\overline z}=0, \quad z\in \R,
  $$
  which is why one calls $\tilde f$ an almost analytic extension of $f$.
  The idea and above construction of almost analytic extensions was introduced by H\"ormander \cite[Section~4.2]{Hormander1969}. Further constructions were introduced by Mather \cite{Mather1971} and later by Jensen and Nakamura \cite{JensenNakamura1994}. See Sj\"ostrand's final comments in \cite{Hormander1969} and \cite[Ch.~8]{DimassiSjostrand1999} for more interesting historical remarks.
  It is also possible to prove \eqref{eq:hs} and formulae for higher-order derivatives $f^{(n)}(A)$ even if $f$ is not compactly supported, but decays at infinity. Such versions are contained, e.g., in \cite[Theorem~1.17, Corollary~1.18, Lemma~1.19]{Isozaki2023}.
\end{remark}

In the proofs below of the phase space propagation estimates, we frequently encounter commutators of the form $[|D|^\alpha,\chi(x)]$ for a radial switch function $\chi \in C^\infty(\R^d:[0,1])$ with $\chi(x)=0$ if $|x|\leq1/2$ and $\chi(x)=1$ if $|x|\geq1$, say. When $\alpha=2$, this commutator can be evaluated explicitly and is seen to depend only on $\nabla\chi$, which has support in $|x|\in[1/2,1]$. However, when $\alpha<2$, the commutator has a kernel of the form $(\chi(x)-\chi(y))|x-y|^{-d-\alpha}$, which is not localized to the support of $\nabla\chi$, and requires more care.
To prepare ourselves for such situations, we prove the following commutator expansion involving functions of operators which do not decay at infinity.

\begin{lemma}
  \label{lem:functional-calculus-commutator}
  Let $\mathcal H$ be a Hilbert space, let $A$ be selfadjoint on $\mathcal H$, and let $X\in\mathcal B(\mathcal H)$. Let $f\in L^\infty(\R)\cap C^\infty(\R)$ be such that
  \begin{align}
    \label{eq:f-S0}
    \sup_{x\in\R} \langle x\rangle^k |f^{(k)}(x)|<\infty, \quad k\in\N_0.
  \end{align}
  For $N\in\N$ let $\widetilde f_N$ be an almost analytic extension of $f$.
  Let also $R_t(z) := \left(z-A/t\right)^{-1}$.

  \begin{enumerate}
  \item Assume that $X\in C^1(A)$, i.e.,
    $$
    X \D(A) \subset \D(A)
    $$
    and that the commutator
    $$
    [X,A] = XA-AX, \quad \D([X,A]) = \D(A),
    $$
    extends to a bounded operator on $\mathcal H$, denoted again by $[X,A]$. Then, for every $N\geq2$,
    \begin{align}
      \label{eq:first-commutator-HS-corrected}
      \left[X,f\left(\frac{A}{t}\right)\right]
      = \frac{1}{2\pi i\,t} \int_{\C} \overline{\partial} \widetilde f_N(z) R_t(z)[X,A]R_t(z)\, dz\wedge d\overline z
    \end{align}
    with
    \begin{align}
      \label{eq:first-functional-commutator}
      \left\|\left[X,f\left(\frac At\right)\right]\right\|
      \lesssim_f \frac1t\|[X,A]\|,
      \qquad t>0.
    \end{align}

  \item Assume in addition that $X\in C^2(A)$, i.e.,
    $$
    [X,A] \D(A) \subset \D(A)
    $$
    and that
    $$
    [[X,A],A] = [X,A]A-A[X,A], \quad \D([[X,A],A]) = \D(A),
    $$
    extends to a bounded operator on $\mathcal H$, denoted again by $[[X,A],A]$. Then, for every $N\geq3$,
    \begin{align}
      \label{eq:second-functional-commutator}
      \left[X,f\left(\frac At\right)\right]
      = \frac1t f'\left(\frac At\right)[X,A] + \ce_f(t)
    \end{align}
    where
    \begin{align}
      \label{eq:explicit-Rf}
      \ce_f(t) := \frac{1}{2\pi i\,t^2} \int_{\C} \overline\partial\widetilde f_N(z) R_t(z)^2[[X,A],A]R_t(z)\, dz\wedge d\overline z.
    \end{align}
    The integral in \eqref{eq:explicit-Rf} is an absolutely convergent Bochner integral in $\mathcal B(\mathcal H)$ and
    \begin{align}
      \label{eq:second-functional-remainder}
      \|\ce_f(t)\| \lesssim_f \frac1{t^2}\|[[X,A],A]\|, \quad t>0.
    \end{align}
  \end{enumerate}
\end{lemma}

The assumption on $f$ means that $f$ belongs to the symbol class $S^0(\R)$. In particular, the lemma applies to $f\in C^\infty(\R)\cap L^\infty(\R)$ with $\supp(f')$ compact.
More general commutator expansions are available for functions in symbol classes $S^s(\R)$, although their precise formulation and mapping properties depend on the order $s\in\R$; see, e.g., \cite[p.~187--188]{Isozaki2023} and related specialized commutator expansions in \cite[Lemma~2.4]{Ishida2025}. For the present argument, only the $S^0$ case is needed, but in Lemma~\ref{lem:fractional-A-commutator} below, we will consider such a generalization for $s<1$; higher orders can be considered similarly.

\begin{proof}[Proof of Lemma~\ref{lem:functional-calculus-commutator}]
  Let us fix $N\in\{2,3,4,...\}$ and choose $\eta\in C_c^\infty(\R)$ such that
  $$
  \eta(s)=1 \quad \text{for} \ |s| \leq 1,
  \qquad
  \eta(s)=0 \quad \text{for} \ |s| \geq 2.
  $$
  We define the almost analytic extension
  \begin{align}
    \label{eq:explicit-AA-extension}
    \widetilde f_N(x+iy)
    := \eta\left(\frac{y}{\langle x\rangle}\right) \sum_{k=0}^N \frac{f^{(k)}(x)(iy)^k}{k!},
  \end{align}
  which obeys
  $$
  \widetilde f_N|_{\R}=f, \qquad
  \supp\widetilde f_N \subset \{x+iy:|y|\leq2\langle x\rangle\}.
  $$
  Moreover, by \eqref{eq:f-S0},
  \begin{align}
    \label{eq:AA-extension-S0-corrected}
    |\overline\partial\widetilde f_N(x+iy)|
    \leq C_{f,N}\langle x\rangle^{-N-1}|y|^N \one_{|y|\leq2\langle x\rangle}.
  \end{align}
  Indeed, for those $x,y$ for which $\eta(y/\langle x\rangle)=1$, the terms of order strictly less than $N$ vanish when $\overline{\partial}$ is applied. On the other hand, on the support of the derivatives of $\eta$, we have $|y|\sim\langle x\rangle$. Thus, for $m\in\{2,3,...,N\}$, we obtain from \eqref{eq:AA-extension-S0-corrected},
  \begin{align}
    \label{eq:Im-finite-corrected}
    \begin{split}
      \int_{\C} |\overline\partial\widetilde f_N(z)| |\im(z)|^{-m}\,dx\,dy
      & \lesssim_{f,N} \int_{\R} \langle x\rangle^{-N-1} \int_0^{2\langle x\rangle} y^{N-m}\,dy\,dx \\
      & \lesssim \int_{\R}\langle x\rangle^{-m}\,dx < \infty.
    \end{split}
  \end{align}

  We next justify the Helffer--Sj\"ostrand commutator formula despite the fact that $f$ need not decay. Let $\chi\in C_c^\infty(\R)$ be such that
  $$
  \chi=1 \text{ on }[-1,1],
  \qquad
  \supp\chi\subset[-2,2],
  $$
  and, for $L\geq1$, set
  $$
  f_L(x) := \chi(x/L)f(x).
  $$
  Then $f_L\in C_c^\infty(\R)$,
  $$
  f_L(x) \to f(x) \quad\text{for every }x\in\R, \quad
  \sup_{L\geq1}\|f_L\|_\infty<\infty,
  $$
  and the family $(f_L)_{L\geq1}$ is bounded in $S^0(\R)$; indeed,
  \begin{align}
    \label{eq:uniform-S0-truncations}
    \sup_{L\geq1}\sup_{x\in\R} \langle x\rangle^k|f_L^{(k)}(x)|<\infty, \quad k\in\N_0.
  \end{align}
  Let $\widetilde f_{L,N}$ be defined from $f_L$ by the same construction as in \eqref{eq:explicit-AA-extension}. The estimate \eqref{eq:AA-extension-S0-corrected} then holds uniformly in $L$, and
  $$
  \overline{\partial} \widetilde f_{L,N}(z) \to \overline{\partial} \widetilde f_N(z)
  \quad\text{for every }z\in\C.
  $$
  Let
  $$
  B_t := \frac At \quad \text{and} \quad R_t(z):=(z-B_t)^{-1}.
  $$
  Since $B_t$ is selfadjoint,
  \begin{align}
    \label{eq:resolvent-bound}
    \|R_t(z)\| \leq |\im(z)|^{-1}, \quad z\in\C \setminus \R.
  \end{align}
  Under the hypotheses of Item~(1), the resolvent commutator identity is
  \begin{align}
    \label{eq:resolvent-commutator-first-corrected}
    [X,R_t(z)] = R_t(z)[X,B_t]R_t(z) = \frac1t R_t(z)[X,A]R_t(z).
  \end{align}
  To verify it, note that $R_t(z) \mathcal H \subset \D(A)$ and $X \D(A) \subset \D(A)$. Hence, all expressions are well-defined, initially on $\mathcal H$, and the identity follows from
  $$
  [X,T^{-1}] = -T^{-1}[X,T]T^{-1}, \quad T=z-B_t.
  $$
  Since $f_L$ is compactly supported, the usual Helffer--Sj\"ostrand formula in Proposition~\ref{prop:analyticextension} and \eqref{eq:resolvent-commutator-first-corrected} give
  \begin{align}
    \label{eq:truncated-commutator-HS}
    [X,f_L(B_t)]
    = \frac{1}{2\pi i\,t} \int_{\C} \overline{\partial}\widetilde f_{L,N}(z) R_t(z) [X,A] R_t(z)\, dz\wedge d\overline z.
  \end{align}
  By the spectral theorem, $f_L(B_t) \to f(B_t)$ strongly, and hence
  $$
  [X,f_L(B_t)] \to [X,f(B_t)] \quad\text{strongly}.
  $$
  On the other hand, by \eqref{eq:resolvent-bound} and \eqref{eq:Im-finite-corrected} with $m=2$, dominated convergence shows that the right-hand side of \eqref{eq:truncated-commutator-HS} converges in operator norm. This shows \eqref{eq:first-commutator-HS-corrected} and
  \begin{align*}
    \|[X,f(B_t)]\|
    & \lesssim \frac1t \|[X,A]\| \int_{\C} |\overline\partial\widetilde f_N(z)| |\im(z)|^{-2}\,dx\,dy
      \lesssim_f \frac1t\|[X,A]\|.
  \end{align*}
  This proves Item~(1).

  We now prove Item~(2), assuming $N\geq3$. Set
  $$
  C_t := [X,B_t]=\frac1t[X,A],
  \quad \text{and} \quad
  D_t := [C_t,B_t] = \frac1{t^2}[[X,A],A].
  $$
  Then,
  \begin{align}
    \label{eq:C-resolvent-corrected}
    [C_t,R_t(z)] = R_t(z)D_tR_t(z).
  \end{align}
  Thus, $C_t R_t(z) = R_t(z)C_t+R_t(z)D_tR_t(z)$, and therefore
  \begin{align}
    \label{eq:resolvent-second-expansion-corrected}
    R_t(z)C_tR_t(z) = R_t(z)^2C_t+R_t(z)^2D_tR_t(z).
  \end{align}
  Substituting this into \eqref{eq:first-commutator-HS-corrected} yields
  \begin{align}
    \label{eq:second-commutator-HS-corrected}
    \begin{split}
      [X,f(B_t)]
      & = \frac{1}{2\pi i} \left(\int_{\C} \overline{\partial}\widetilde f_N(z) R_t(z)^2 \,dz \wedge d\overline z\right) C_t \\
      & \quad + \frac{1}{2\pi i} \int_{\C} \overline{\partial}\widetilde f_N(z) R_t(z)^2D_tR_t(z)\, dz \wedge d\overline z.
    \end{split}
  \end{align}
  It remains to identify the first integral. For every $L\geq1$, differentiation of the scalar Helffer--Sj\"ostrand formula for $f_L=\chi(\cdot/L)f$ gives
  $$
  f_L'(B_t)
  = \frac{1}{2\pi i} \int_{\C} \overline{\partial}\widetilde f_{L,N}(z) R_t(z)^2 \,dz \wedge d\overline z.
  $$
  We have
  $$
  f_L'(x)-f'(x) = (\chi(x/L)-1)f'(x) + \frac1L\chi'(x/L)f(x),
  $$
  and therefore
  $$
  \|f_L'-f'\|_\infty \lesssim_f L^{-1}.
  $$
  Hence, $f_L'(B_t)\to f'(B_t)$ in operator norm. Applying dominated convergence once more, using $m=2$ in \eqref{eq:Im-finite-corrected} yields
  \begin{align}
    \label{eq:derivative-HS-corrected}
    f'(B_t) = \frac{1}{2\pi i} \int_{\C} \overline{\partial}\widetilde f_N(z) R_t(z)^2 \,dz\wedge d\overline z.
  \end{align}
  Thus, the first term in \eqref{eq:second-commutator-HS-corrected} equals
  $$
  f'(B_t)C_t = \frac1t f'(B_t)[X,A].
  $$
  Since $D_t=t^{-2}[[X,A],A]$, the second term is
  $$
  \ce_f(t) = \frac{1}{2\pi i\,t^2} \int_{\C} \overline{\partial}\widetilde f_N(z) R_t(z)^2[[X,A],A]R_t(z)\, dz\wedge d\overline z.
  $$
  Finally,
  $$
  \|R_t(z)^2[[X,A],A]R_t(z)\| \leq |\im(z)|^{-3} \|[[X,A],A]\|.
  $$
  Therefore, by \eqref{eq:Im-finite-corrected} with $m=3$,
  \begin{align*}
    \|\ce_f(t)\|
    & \lesssim \frac1{t^2} \|[[X,A],A]\| \int_{\C} |\overline{\partial}\widetilde f_N(z)| |\im(z)|^{-3}\,dx\,dy \\
    & \lesssim_f \frac1{t^2}\|[[X,A],A]\|.
  \end{align*}
  This proves Item~(2).
\end{proof}


\subsection{Estimates involving the generator of dilations}
\label{ss:preliminaries-generator-of-dilations}

We first compute the commutator $i[L_\kappa,A]$.

\begin{lemma}
  \label{lem5}
  Let $\alpha\in(0,2]\cap(0,d)$ and $\kappa\leq\kappa_{\rm c}$. Let $\mathcal Q_\kappa:=\D(L_\kappa^{1/2})$ be equipped with the form norm $\|u\|_{\mathcal Q_\kappa}:=(\|L_\kappa^{1/2}u\|_2^2+\|u\|_2^2)^{1/2}$, and let $\mathcal Q_\kappa^*$ denote its anti-dual.
  Then, the dilation group $U_t:=\me{itA}$, given by $(U_tu)(x)=\me{dt/2}u(\me{t}x)$, preserves both $\mathcal Q_\kappa$ and $\D(L_\kappa)$, and we have
  \begin{align}
    \label{eq_lem5_0}
    U_t L_\kappa U_{-t} = \me{-\alpha t} L_\kappa, \quad t\in\R.
  \end{align}
  The form commutator $i[L_\kappa,A]$ is well-defined as a bounded operator mapping $\mathcal Q_\kappa$ to $\mathcal Q_\kappa^*$ and satisfies
  \begin{equation}
    \label{eq_lem5_1}
    \langle \psi, i[L_\kappa,A] \varphi\rangle
    = \alpha \langle L_\kappa^{1/2}\psi,L_\kappa^{1/2}\varphi\rangle,
    \quad \psi,\varphi\in\mathcal Q_\kappa.
  \end{equation}
  Moreover,
  \begin{equation}
    \label{eq_lem5_2}
    \langle \psi, i[L_\kappa,A] \varphi\rangle
    = i\langle L_\kappa\psi,A\varphi\rangle - i \langle A\psi,L_\kappa\varphi\rangle
    = \alpha \langle \psi,L_\kappa\varphi\rangle
  \end{equation}
  holds for all $\psi,\varphi\in \D(A)\cap\D(L_\kappa)$.
\end{lemma}

Note that if $\kappa<\kappa_{\rm c}$, then $\mathcal Q_\kappa=H^{\alpha/2}(\R^d)$ with equivalent norms, so $i[L_\kappa,A]$ is also a bounded operator from $H^{\alpha/2}(\R^d)\to H^{-\alpha/2}(\R^d)$.

\begin{proof}
  On the form core $C_c^\infty(\R^d)$ (see, e.g., \cite[Theorem~5.4]{Bogdanetal2019} for $\alpha<2$ and $\kappa=\kappa_{\rm c}$), the nonnegative form defining the Friedrichs realization $L_\kappa$ is
  \begin{equation}
    q_\kappa(\psi,\varphi)
    = \langle |D|^{\alpha/2}\psi,|D|^{\alpha/2}\varphi\rangle - \kappa \langle |x|^{-\alpha/2}\psi,|x|^{-\alpha/2}\varphi\rangle.
  \end{equation}
  We use the same notation for its closure on $\mathcal Q_\kappa$, where $q_\kappa(\psi,\phi)=\langle L_\kappa^{1/2}\psi,L_\kappa^{1/2}\phi\rangle$. (Note that if $\kappa=\kappa_{\rm c}$, then the two above integrals may be separately infinite, although $q_\kappa(\psi,\phi)$ remains well-defined as a closed form.) Since $C_c^\infty(\R^d)$ is invariant under $U_t$, we get
  \begin{equation}
    \label{eq_lem5_3}
    q_\kappa(\me{-i t A} \psi, \me{-i t A} \varphi) = \me{-\alpha t} q_\kappa(\psi,\varphi).
  \end{equation}
  for all $\psi,\varphi\in C_c^\infty(\R^d)$. In particular, $\|U_{-t}u\|_{\mathcal Q_\kappa}^2=\|u\|_2^2+\me{-\alpha t}q_\kappa(u,u)$ on $C_c^\infty(\R^d)$. Hence, $U_{-t}$ extends to a bounded bijection of $\mathcal Q_\kappa$, with inverse $U_t$, and \eqref{eq_lem5_3} extends to all $\psi,\phi\in\mathcal Q_\kappa$ by closure. The representation theorem for closed forms then yields invariance of $\D(L_\kappa)$ and \eqref{eq_lem5_0}.

  \smallskip
  We now define the form commutator by
  \begin{equation}
    \label{eq_lem5_4}
    i[L_\kappa,A]_q(\psi,\varphi)
    := -\frac{d}{dt}q_\kappa(U_{-t}\psi, U_{-t}\varphi)\big|_{t=0}
    = \alpha q_\kappa(\psi,\varphi).
  \end{equation}
  It obeys
  $$
  |i[L_\kappa,A]_q(\psi,\phi)|
  \leq \alpha \|\psi\|_{\mathcal Q_\kappa} \|\phi\|_{\mathcal Q_\kappa}
  $$
  and therefore defines a bounded operator from $\mathcal Q_\kappa$ to $\mathcal Q_\kappa^*$ with
  \begin{equation}
    \label{eq_lem5_5}
    i [L_\kappa,A]_q(\psi,\varphi)
    = \langle \psi,i[L_\kappa,A]\varphi\rangle.
  \end{equation}
  This shows \eqref{eq_lem5_1}.

  \smallskip
  Finally, \eqref{eq_lem5_0} implies
  \begin{align*}
    \langle L_\kappa\psi,U_t\phi\rangle
    = \me{\alpha t} \langle U_{-t}\psi, L_\kappa \phi\rangle, \quad \psi,\phi\in\D(L_\kappa).
  \end{align*}
  If also $\psi,\phi\in\D(A)$, then both sides can be differentiated at $t=0$, with the resulting identity being \eqref{eq_lem5_2}.
\end{proof}

\begin{remark}
  \label{rem:Lkappa-ac-spectrum}
  Using Lemma~\ref{lem5}, let us present an argument showing that
  \begin{align}
    \label{eq:hardy-purelyac}
    \sigma(L_\kappa)=\sigma_{\rm ac}(L_\kappa)=[0,\infty), \qquad
    \sigma_{\rm pp}(L_\kappa)=\sigma_{\rm sc}(L_\kappa)=\emptyset.
  \end{align}  
  for all $\kappa\leq\kappa_{\rm c}$. These assertions are well-known for $\kappa<\kappa_{\rm c}$, see, e.g., \cite{Herbst1977} for $d=3$ and $\alpha=1$, \cite{MizutaniYao2021} for $\alpha<2$, and \cite{Bruneauetal2011} for $\alpha=2$. For $\alpha<2$, \eqref{eq:hardy-purelyac} can be inferred from the abstract classification of homogeneous selfadjoint operators in \cite[Proposition~2.7]{Bruneauetal2011}; preservation of homogeneity by the Friedrichs extension is also recorded in \cite[Proposition~2.6]{Bruneauetal2011}. Since we are otherwise unaware of a proof of \eqref{eq:hardy-purelyac} when $\alpha<2$ and $\kappa=\kappa_{\rm c}$ (even though it may be folklore), we give a self-contained proof.

  \smallskip
  \noindent
  {\bf Absence of zero eigenvalues:}
  Let us first show that zero is not an eigenvalue. If $\kappa<\kappa_{\rm c}$, this follows from $L_\kappa\gtrsim|D|^\alpha\geq0$ by Hardy's inequality. To see that zero is also not an eigenvalue for $\kappa=\kappa_{\rm c}$, we recall the ground state representation of $L_\kappa$, see \cite[Proposition~4.1]{Franketal2008H} or \cite[Proposition~5]{Bogdanetal2016} for $\alpha<2$. For $\alpha<2$, $\kappa=\kappa_{\rm c}$, $v\in C_c^\infty(\R^d\setminus\{0\})$ (which is also a form core of $L_\kappa$ according to \cite[Theorem~5.4]{Bogdanetal2019}), and $h(x)=|x|^{-(d-\alpha)/2}$, it states
  $$
  q_{\kappa_{\rm c}}(v,v) = c_{d,\alpha}\int_{\R^d}\int_{\R^d}\frac{|v(x)/h(x)-v(y)/h(y)|^2}{|x-y|^{d+\alpha}}h(x)h(y)\,dx\,dy.
  $$
  Now, let $u\in\ker L_\kappa$ and choose a sequence $(u_n)_{n\in\N}\subset C_c^\infty(\R^d\setminus\{0\})$ with $u_n\to u$ in $\mathcal Q_\kappa$ so that $q_\kappa(u_n,u_n)\to0$. Thus, passing to an almost-everywhere convergent subsequence and applying Fatou's lemma shows that $u/h$ is constant almost everywhere. Since $h\notin L^2(\R^d)$, we see $\ker L_\kappa$ is trivial also for $\kappa=\kappa_{\rm c}$. For $d\geq3$ and $\alpha=2$, the ground state representation states, with $\kappa_{\rm c}=(d-2)^2/4$ and $h(x)=|x|^{-(d-2)/2}$,
  $$
  q_{\kappa_{\rm c}}(v,v) = \int_{\R^d}h(x)^2\left|\nabla\left(\frac{v(x)}{h(x)}\right)\right|^2\,dx, \quad v\in C_c^\infty(\R^d).
  $$
  Let $u\in\ker L_{\kappa_{\rm c}}$ and choose $u_n\in C_c^\infty(\R^d)$ such that $u_n\to u$ in $\mathcal Q_{\kappa_{\rm c}}$. Since $q_{\kappa_{\rm c}}(u_n,u_n)\to0$, we obtain $\nabla(u_n/h)\to0$ in $L^2_{\rm loc}(\R^d\setminus\{0\})$. Moreover, $u_n/h\to u/h$ in $L^2_{\rm loc}(\R^d\setminus\{0\})$. Thus, $\nabla(u/h)=0$ in the distributional sense on $\R^d\setminus\{0\}$. Since this set is connected for $d\geq3$, we conclude that $u=ch$ for some $c\in\C$. As $h\notin L^2(\R^d)$, we have $c=0$. Hence, $\ker L_{\kappa_{\rm c}}=\{0\}$ also for $\alpha=2$.

  \smallskip
  \noindent
  {\bf Absence of positive eigenvalues:}
  Next, by Lemma~\ref{lem5} and the separability of $L^2(\R^d)$, $L_\kappa$ cannot have positive eigenvalues. Indeed, if $\lambda>0$ were an eigenvalue with eigenfunction $u$, then by scaling $L_\kappa U_t u = \me{\alpha t}\lambda U_t u$. Thus, as $t$ ranges over $\R$, this would produce uncountably many distinct eigenvalues and hence an uncountable orthogonal family of nonzero vectors, which contradicts the separability of $L^2(\R^d)$.

  \smallskip
  \noindent
  {\bf Absolute continuity on $(0,\infty)$:}
  Let us now show $\sigma_{\rm sc}(L_\kappa)=\emptyset$. For $u\in L^2(\R^d)$ and a Borel set $B\subset\R$, let $E(B):=\one_B(L_\kappa)$ and $\mu_u(B):=\langle u,E(B)u\rangle$. Then, by Lemma~\ref{lem5},
  \begin{align*}
    U_{-t}E(B)U_t = E(\me{-\alpha t}B).
  \end{align*}
  Let $N\subset(0,\infty)$ be a Borel set with zero Lebesgue measure and $u\in L^2(\R^d)$. We claim $\|E(N)u\|_2=0$. To this end, it suffices to show $\|E(N)U_tu\|_2=0$ for almost all $t$. Indeed, since $U_t$ is strongly continuous, we would then obtain that $\|E(N)U_tu\|_2$ is continuous and hence vanishes for every $t$. In particular, taking $t=0$ would then yield $E(N)u=0$ for every $u$, which shows that $\sigma(L_\kappa)$ is absolutely continuous on $(0,\infty)$. We now show $\|E(N)U_tu\|_2=0$ for almost all $t$. First, we write
  \begin{align*}
    \|E(N)U_tu\|_2^2
    = \mu_u(\me{-\alpha t}N)
    = \int_{(0,\infty)}\one_N(\me{\alpha t}\lambda)\,d\mu_u(\lambda).
  \end{align*}
  For every $\lambda>0$, the substitution $r=\me{\alpha t}\lambda$ yields
  \begin{align*}
    \int_\R\one_N(\me{\alpha t}\lambda)\,dt = \frac{1}{\alpha}\int_N\frac{dr}{r}=0.
  \end{align*}
  Hence, by Tonelli's theorem,
  $$
  \int_\R \|E(N)U_tu\|_2^2\,dt = 0
  $$
  which shows $\|E(N)U_tu\|_2=0$ for almost all $t$ and hence the absolute continuity of $\sigma(L_\kappa)$ on $(0,\infty)$.

  \smallskip
  \noindent
  {\bf Conclusion:}
  Finally, by Lemma~\ref{lem5}, we have $\sigma(L_\kappa)=\me{-\alpha t}\sigma(L_\kappa)$ for every $t\in\R$. Thus, since $L_\kappa$ is nonzero and closed, we obtain \eqref{eq:hardy-purelyac}.
\end{remark}

\medskip

For the proof of the phase space propagation estimate, we also use the following weighted estimate on $A$.

\begin{lemma}
  \label{lem12}
  Let $\mu\in[0,2]$. Then,
  \begin{equation}
    \langle D\rangle^{-\mu/2}\langle x\rangle^{-\mu/2}\langle A\rangle^{\mu}\langle x\rangle^{-\mu/2}\langle D\rangle^{-\mu/2} \in \mathcal B(L^2(\R^d)).
  \end{equation}
\end{lemma}

\begin{proof}
  It suffices to prove that $\langle A\rangle^{\mu/2}\langle x\rangle^{-\mu/2}\langle D\rangle^{-\mu/2}$ is bounded for $0\leq\mu\leq2$. We define $F(z)=\langle A\rangle^z\langle x\rangle^{-z}\langle D\rangle^{-z}$ for $z\in\C$ with $0\leq \re(z) \leq1$ and apply Stein's complex interpolation theorem (see, e.g., \cite[Theorem~1.3.7]{Grafakos2014C}).
  All expressions involving $F(z)$ are initially interpreted on the Schwartz space $\mathcal S(\R^d)$. For $\phi,\psi\in\mathcal S(\R^d)$, the function $z\mapsto\langle\phi,F(z)\psi\rangle$ is analytic in the open strip and continuous on its closure. The boundary estimates below are uniform on this common core and therefore give bounded closures on the two boundary lines. Stein's interpolation theorem then yields the bounded operator in the interior, and the resulting estimate extends from $\mathcal S(\R^d)$ to $L^2$ by density.

  We first note that $\langle A\rangle\langle x\rangle^{-1}\langle D\rangle^{-1}$ is bounded. Indeed, since $A=x\cdot D-i\cdot d/2$,
  $$
  A \langle x\rangle^{-1} \langle D\rangle^{-1}
  = x \langle x\rangle^{-1} D \langle D\rangle^{-1}
  + [A,\langle x\rangle^{-1}] \langle D\rangle^{-1} - \frac{id}{2} \langle x\rangle^{-1}\langle D\rangle^{-1},
  $$
  which is bounded since $[A,\langle x\rangle^{-1}]=-ix\cdot\nabla\langle x\rangle^{-1}$ is also bounded.
   
  Next, $z\mapsto\langle \phi,F(z)\psi\rangle$ is analytic in the open strip $0<\re(z)<1$ and continuous on $0\leq \re(z) \leq1$ for $\phi,\psi\in C_c^\infty(\R^d)$. Indeed, $\langle x\rangle^{-z}=\me{-z\log\langle x\rangle}$ is a multiplication operator, $\langle D\rangle^{-z}$ is a Fourier multiplier, and $\langle A\rangle^z$ is defined by the functional calculus of the selfadjoint operator $A$.

  Since each factor of $F(i\lambda) = \langle A\rangle^{i\lambda}\langle x\rangle^{-i\lambda}\langle D\rangle^{-i\lambda}$ is unitary, $\|F(i\lambda)\|=1$ on the line $\re(z)=0$. 
  On the right end of the strip $\re(z)=1$, we have
  \begin{align}
    \label{eq:F1ilambda}
    \begin{split}
      F(1 + i\lambda)
      & = \langle A\rangle^{1+i\lambda}\langle x\rangle^{-1-i\lambda} \langle D\rangle^{-1-i\lambda} \\
      & = \langle A\rangle^{i\lambda} \cdot \langle A\rangle\langle x\rangle^{-i\lambda} \langle A\rangle^{-1} \cdot \langle A\rangle\langle x\rangle^{-1}\langle D\rangle^{-1} \cdot \langle D\rangle^{-i\lambda}.
    \end{split}
  \end{align}
  It remains to show
  \begin{align}
    \label{eq:AxAinverse}
    \|\langle A\rangle\langle x\rangle^{-i\lambda}\langle A\rangle^{-1}|\|\lesssim\langle\lambda\rangle.
  \end{align}
  To this end, we claim
  \begin{align}
    \label{eq:AxAinverse-squared}
    \| \langle A\rangle^2\langle x\rangle^{-i\lambda}\langle A\rangle^{-2} \|
    \lesssim \langle\lambda\rangle^2.
  \end{align}
  To show \eqref{eq:AxAinverse-squared}, let us denote the multiplication operator $\langle x\rangle^{-i\lambda}$ by $M_\lambda$ and observe
  \begin{align*}
    [A,M_\lambda] & = -ix\nabla M_\lambda, \\
    \|[A,M_\lambda]\| & \lesssim |\lambda|, \\
    \|[A,[A,M_\lambda]]\| & \lesssim |\lambda| + |\lambda|^2.
  \end{align*}
  Thus, \eqref{eq:AxAinverse-squared} follows by expanding the operator in this formula and using the above estimates.
  Applying complex interpolation for the family $\langle A\rangle^z M_\lambda \langle A\rangle^{-z}$ with $0\leq\re(z)\leq2$ then yields \eqref{eq:AxAinverse}.

  Thus, $\|F(1+i\lambda)\|\lesssim\langle\lambda\rangle$ for all $\lambda\in\R$, which concludes the proof.
\end{proof}

The following estimates will be useful to show the nonexistence of $W(L_\kappa-V^L,L_\kappa)$ for $0<\rho_L\leq\alpha\leq1$.

\begin{lemma}
  \label{lem:commutators-spectral-cutoff}
  Let $d\in\N$, $\alpha\in(0,2]\cap(0,d)$, $\kappa<\kappa_{\rm c}$, and $\rho>0$. Assume that $V\in C^2(\R^d)$ satisfies \eqref{ass5} for all multi-indices $\beta\in\N_0^d$ with $|\beta|\leq2$. Let $H:=L_\kappa-V$ and $\lambda\in\R$ be such that $H+\lambda\geq\delta>0$ for some $\delta>0$.
  Let $\chi\in C^\infty(\R)$ satisfy, for some $\theta>0$,
  $$
  \chi(s)=0
  \quad\text{for }s\leq\frac{\theta}{2},
  \qquad
  \chi(s)=1
  \quad\text{for }s\geq\theta.
  $$
  Then,
  \begin{align}
    \label{eq:chi-VL-commutator}
    \|[\chi(A/t),V]\| & \lesssim_\chi (\|V\|_\infty+\|x\cdot\nabla V\|_\infty) t^{-1}, \quad t>0.
  \end{align}
  Moreover, $R_\lambda:=(\lambda+H)^{-1}\in C^1(A)$ and
  \begin{align}
    \label{eq:chi-resolvent-commutator}
    \|[\chi(A/t),R_\lambda]\| & \lesssim_{\chi,V,\delta} t^{-1}, \quad t>0.
  \end{align}
\end{lemma}

\begin{proof}
  First notice that $\chi$ satisfies the assumptions of Lemma~\ref{lem:functional-calculus-commutator}.

  Let us prove \eqref{eq:chi-VL-commutator}. By assumption, $V,x\cdot\nabla V \in L^\infty(\R^d)$. Multiplication by $V$ preserves $\mathcal D(A)$ and, on $\D(A)$, we have $[V,A]=i\,x\cdot\nabla V$. Hence, $V\in C^1(A)$ and, for $t\geq1$, Lemma~\ref{lem:functional-calculus-commutator} yields
  $$
  \|[\chi(A/t),V]\| \lesssim_\chi t^{-1}\|x\cdot\nabla V\|_\infty \lesssim t^{-1}.
  $$
  For $0<t<1$,
  $$
  \|[\chi(A/t),V]\| \leq2\|\chi\|_\infty \|V\|_\infty.
  $$
  This proves \eqref{eq:chi-VL-commutator}.

  We turn to \eqref{eq:chi-resolvent-commutator}. Set
  $$
  B := i[H,A].
  $$
  We first verify that $R_\lambda\in C^1(A)$. Let $U(s):=\me{isA}$ for $s\in\R$. By the dilation covariance of $L_\kappa$ (Lemma~\ref{lem5}), we have $U(s)L_\kappa U(-s) = \me{-\alpha s}L_\kappa$. Moreover, $U(s)VU(-s)=V(\me{s}x)$. Hence, with $H_s:=U(s)HU(-s)$, we have $H_s=\me{-\alpha s}L_\kappa-V(\me{s} x)$. Since $V$ is bounded, $\mathcal D(H)=\mathcal D(L_\kappa)$, and this domain is invariant under $U(s)$. Set $F(x):=x\cdot\nabla V(x)$. By~\eqref{ass5}, $x\cdot\nabla F\in L^\infty(\R^d)$. Moreover, for $s\neq0$
  $$
  \frac{V(\me{s} x)-V(x)}{s} = \int_0^1F(e^{\theta s}x)\,d\theta.
  $$
  Since $\|F(\me{s}\cdot) - F\|_\infty \leq |s|\,\|x\cdot\nabla F\|_\infty$,
  $$
  \lim_{s\to0}\frac{V(\me{s}\cdot)-V}{s} = F \quad\text{in }L^\infty(\R^d).
  $$
  Thus,
  \begin{align}
    \label{eq:H-dilation-derivative}
    \lim_{s\to0} \frac{H_s-H}{s} = -\alpha L_\kappa-x\cdot\nabla V = -B \quad\text{in }\mathcal B(\mathcal D(H),L^2)
  \end{align}
  where $\mathcal D(H)$ is equipped with the graph norm. Let
  $$
  R_{\lambda,s} := U(s)R_\lambda U(-s) = (\lambda+H_s)^{-1}.
  $$
  Since $H_s$ is unitarily equivalent to $H$, we have $\|R_{\lambda,s}\| = \|R_\lambda\| \leq\delta^{-1}$. Using
  $$
  \frac{R_{\lambda,s}-R_\lambda}{s} = -R_{\lambda,s} \frac{H_s-H}{s} R_\lambda
  $$
  and \eqref{eq:H-dilation-derivative}, we obtain
  $$
  \lim_{s\to0}\frac{R_{\lambda,s}-R_\lambda}{s} = R_\lambda B R_\lambda \quad\text{in }\mathcal B(L^2).
  $$
  Thus, $R_\lambda\in C^1(A)$. In particular,
  \begin{align}
    \label{eq:resolvent-A-commutator}
    R_\lambda\mathcal D(A)\subset\mathcal D(A),
    \qquad
    [R_\lambda,A] = iR_\lambda B R_\lambda.
  \end{align}
  By Lemma~\ref{lem5},
  $$
  B = i[H,A] = \alpha L_\kappa+x\cdot\nabla V = \alpha H+W,
  $$
  where $W:=\alpha V+x\cdot\nabla V \in\mathcal B(L^2)$. Therefore,
  $$
  R_\lambda B R_\lambda = \alpha(R_\lambda-\lambda R_\lambda^2) + R_\lambda W R_\lambda \in\mathcal B(L^2).
  $$
  Hence, $[R_\lambda,A]\in\mathcal B(L^2)$. Thus, for $t\geq1$, Lemma~\ref{lem:functional-calculus-commutator} yields
  $$
  \|[\chi(A/t),R_\lambda]\| \lesssim_\chi t^{-1}\|[R_\lambda,A]\| \lesssim_{\delta,V} t^{-1}.
  $$
  Finally, for $0<t<1$,
  $$
  \|[\chi(A/t),R_\lambda]\| \leq 2\|\chi\|_\infty\|R_\lambda\| \lesssim t^{-1}.
  $$
  This proves \eqref{eq:chi-resolvent-commutator}.
\end{proof}

\begin{lemma}
  \label{lem:fractional-A-commutator}
  Let $\mathcal H$ be a Hilbert space, let $A$ be selfadjoint
  on $\mathcal H$, and let $B\in\mathcal B(\mathcal H)$.
  Assume that $B\in C^1(A)$, i.e.,
  $$
  B\mathcal D(A)\subset\mathcal D(A)
  $$
  and that
  $$
  [B,A] = BA-AB,
  \qquad
  \mathcal D([B,A])=\mathcal D(A),
  $$
  extends to a bounded operator on $\mathcal H$.
  Then, for every $0<s<1$,
  $$
  B\mathcal D(\langle A\rangle^s) \subset \mathcal D(\langle A\rangle^s)
  $$
  and $[B,\langle A\rangle^s]$, initially defined on $\D(\langle A\rangle^s)$, extends to a bounded operator on $\mathcal H$ with
  \begin{align}
    \label{eq:fractional-A-abstract-comm-bound}
    \|[B,\langle A\rangle^s]\| \lesssim_s \|[B,A]\|.
  \end{align}
\end{lemma}

\begin{proof}
  Set $f(x) := \langle x\rangle^s$. Then $f\in S^s(\R)$, i.e.,
  $$
  |f^{(k)}(x)| \lesssim_{s,k} \langle x\rangle^{s-k}, \quad k\in\N_0.
  $$
  Choose $\eta\in C_c^\infty(\R)$ such that
  $$
  \eta(r)=1\quad\text{for }|r|\leq1,
  \qquad
  \eta(r)=0\quad\text{for }|r|\geq2,
  $$
  and, for $N\in\{2,3,...\}$, define the almost analytic extension
  $$
  \widetilde f_N(x+iy) := \eta\left(\frac{y}{\langle x\rangle}\right) \sum_{k=0}^N \frac{f^{(k)}(x)(iy)^k}{k!}.
  $$
  Then,
  \begin{align}
    \label{eq:fractional-AA-symbol}
    |\overline\partial\widetilde f_N(x+iy)|
    \lesssim_{s,N} \langle x\rangle^{s-N-1}|y|^N \one_{|y|\leq2\langle x\rangle}.
  \end{align}
  Thus,
  \begin{align}
    \label{eq:fractional-AA-integrability}
    \int_{\C} |\overline\partial\widetilde f_N(z)| |\im z|^{-2}\,dx\,dy 
    \lesssim \int_{\R} \langle x\rangle^{s-N-1} \int_0^{2\langle x\rangle} y^{N-2}\,dy\,dx
    \lesssim \int_{\R}\langle x\rangle^{s-2}\,dx
    < \infty
  \end{align}
  since $s<1$. As in Lemma~\ref{lem:functional-calculus-commutator}, we can represent $[B,f(A)]$ as a Helffer--Sj\"ostrand-type integral by a truncation argument. To this end, let $\chi\in C_c^\infty(\R)$ be such that
  $$
  0\leq\chi\leq1,\qquad
  \chi=1\ \text{on }[-1,1],\qquad
  \supp\chi\subset[-2,2],
  $$
  and set
  $$
  f_L(x):=\chi(x/L)f(x), \quad L\geq1.
  $$
  Then $f_L\in C_c^\infty(\R)$, and $(f_L)_{L\geq1}$ is bounded in $S^s(\R)$, uniformly in $L$. Let $\widetilde f_{L,N}$ be the corresponding almost analytic extensions. Since $f_L$ is compactly supported, the ordinary Helffer--Sj\"ostrand formula yields
  \begin{align}
    \label{eq:fractional-truncated-comm}
    [B,f_L(A)] = \frac{1}{2\pi i} \int_{\C} \overline\partial\widetilde f_{L,N}(z) (z-A)^{-1}[B,A](z-A)^{-1} \,dz\wedge d\overline z.
  \end{align}
  Hence, by $\|(z-A)^{-1}\|\leq|\im z|^{-1}$, \eqref{eq:fractional-AA-integrability}, and dominated convergence, the right-hand side of \eqref{eq:fractional-truncated-comm} converges in operator norm as $L\to\infty$ to a bounded operator $K$ satisfying
  \begin{align}
    \label{eq:fractional-K-bound}
    \|K\| \lesssim_s \|[B,A]\|.
  \end{align}

  It remains to identify $K$ with $[B,\langle A\rangle^s]$ and show the domain invariance. Let $u\in\mathcal D(\langle A\rangle^s)$. By the spectral theorem,
  $$
  f_L(A)u \to f(A)u = \langle A\rangle^s u \quad \text{as } L\to\infty.
  $$
  Moreover,
  \begin{align}
    \label{eq:fractional-domain-limit}
    f_L(A)Bu = Bf_L(A)u-[B,f_L(A)]u.
  \end{align}
  The right-hand side converges in $\mathcal H$. In particular,
  $$
  \sup_{L\geq1}\|f_L(A)Bu\|<\infty.
  $$
  Since $f_L(x)\to f(x)$ pointwise, Fatou's lemma applied to the measure $\langle Bu,dE_A(\cdot)Bu\rangle$ yields
  $$
  \int_\R |f(\lambda)|^2\, \langle Bu,dE_A(\lambda)Bu\rangle < \infty.
  $$
  Hence, $Bu\in\mathcal D(f(A)) = \mathcal D(\langle A\rangle^s)$ and by the spectral theorem,
  $$
  f_L(A)Bu \to f(A)Bu \quad \text{as }L\to\infty.
  $$
  Taking the limit in \eqref{eq:fractional-domain-limit} shows
  $$
  B\langle A\rangle^s u - \langle A\rangle^s Bu = Ku.
  $$
  Therefore,
  $$
  [B,\langle A\rangle^s]=K
  $$
  on $\mathcal D(\langle A\rangle^s)$, and \eqref{eq:fractional-A-abstract-comm-bound} follows from \eqref{eq:fractional-K-bound}.
\end{proof}

With the previous lemma, we obtain the following boundedness of $\langle A\rangle^\rho V$ when $V$ decays like $|x|^{-\rho}$ in Sobolev spaces whose minimal regularity depends on $\rho$.

\begin{lemma}
  \label{lem:fractional-A-sandwich}
  Let $0<\rho<1$, $\rho\leq\alpha<2$, and suppose that $V\in C^\infty(\R^d)$ satisfies \eqref{ass5} for all multi-indices $\beta\in\N_0^d$. Then,
  \begin{align}
    \label{eq:fractional-sandwich}
    \langle D\rangle^{-\alpha/2} \langle A\rangle^\rho V \langle D\rangle^{-\alpha/2} \in\mathcal B(L^2(\R^d)).
  \end{align}
\end{lemma}

\begin{proof}
  Set
  $$
  w(x):=\langle x\rangle^{-\rho/2},
  \qquad
  b(x):=\langle x\rangle^\rho V(x).
  $$
  Then $V=w^2b$. By assumption, $b\in W^{1,\infty}(\R^d)$. Moreover, $x\cdot\nabla w\in L^\infty(\R^d)$. Thus, by Lemma~\ref{lem:fractional-A-commutator} (using $\rho<1$),
  \begin{align}
    \label{eq:fractional-w-comm}
    [\langle A\rangle^\rho,w] \in\mathcal B(L^2).
  \end{align}
  The following identities are first understood on $C_c^\infty(\R^d)$ (or another common core), and the estimates show that the operator sandwiched between the $\langle D\rangle^{-\alpha/2}$ extends uniquely to a bounded operator on $L^2$. Thus, on $C_c^\infty(\R^d)$,
  \begin{align}
    \label{eq:fractional-sandwich-decomp}
    \langle A\rangle^\rho V
    = \langle A\rangle^\rho w^2b
    = w\langle A\rangle^\rho wb + [\langle A\rangle^\rho,w] wb.
  \end{align}
  The second term is bounded by \eqref{eq:fractional-w-comm} since $w$ and $b$ are bounded. Hence,
  $$
  \langle D\rangle^{-\alpha/2} [\langle A\rangle^\rho,w] w b \langle D\rangle^{-\alpha/2} \in \mathcal B(L^2).
  $$
  For the first term in \eqref{eq:fractional-sandwich-decomp}, Lemma~\ref{lem12}, applied with $\mu=\rho$, yields
  \begin{align}
    \label{eq:fractional-lem12-use}
    T_\rho := \langle D\rangle^{-\rho/2} w\langle A\rangle^\rho w\langle D\rangle^{-\rho/2} \in\mathcal B(L^2).
  \end{align}
  Therefore,
  \begin{align}
    \label{eq:fractional-first-factorization}
    \langle D\rangle^{-\alpha/2} w\langle A\rangle^\rho w b \langle D\rangle^{-\alpha/2}
    = \langle D\rangle^{-(\alpha-\rho)/2} T_\rho \left(\langle D\rangle^{\rho/2} b \langle D\rangle^{-\alpha/2}\right).
  \end{align}
  The first two factors on the right-hand side are bounded. We claim that the last factor, i.e., $\langle D\rangle^{\rho/2}b\langle D\rangle^{-\alpha/2}$, is also bounded. Indeed, multiplication by a $W^{1,\infty}$ function (like $b$ here) is bounded on $H^r(\R^d)$ for every $0\leq r\leq1$. Indeed, the statement is immediate for $r=0$, while for $r=1$ it follows from $\nabla(bu) = b\nabla u+(\nabla b)u$; the intermediate values of $r$ follow by complex interpolation. Since $\rho/2<1$, we obtain
  \begin{align}
    \label{eq:fractional-b-Hq}
    \langle D\rangle^{\rho/2} b \langle D\rangle^{-\alpha/2}
    = \left( \langle D\rangle^{\rho/2} b \langle D\rangle^{-\rho/2} \right) \langle D\rangle^{-(\alpha-\rho)/2} \in \mathcal B(L^2)
  \end{align}
  since $\rho\leq\alpha$.
  Thus, also the first term in \eqref{eq:fractional-sandwich-decomp}, after sandwiching, is bounded.
  This concludes the proof.
\end{proof}

\begin{corollary}
  \label{cor:fractional-A-resolvent-V}
  Let $0<\rho\leq\alpha<1$, $\kappa<\kappa_{\rm c}$, and assume $V\in C^\infty(\R^d)$ satisfies \eqref{ass5} for all multi-indices $\beta\in\N_0^d$. Let $H_\kappa=L_\kappa-V$ and $\lambda\in\R$ be such that $H_\kappa+\lambda\geq\delta>0$ for some $\delta>0$. Then, for every $f\in C_c^\infty(0,\infty)$,
  \begin{align}
    \label{eq:fractional-A-resolvent-V}
    \langle A\rangle^\rho (\lambda+H_\kappa)^{-1} V f(L_\kappa) \in\mathcal B(L^2).
  \end{align}
\end{corollary}

\begin{proof}
  Set $R_\lambda:=(\lambda+H_\kappa)^{-1}$. By Lemma~\ref{lem:commutators-spectral-cutoff}, $R_\lambda\in C^1(A)$. Hence, Lemma~\ref{lem:fractional-A-commutator} with $B=R_\lambda$ yields
  \begin{align}
    \label{eq:fractional-R-comm}
    [\langle A\rangle^\rho,R_\lambda] \in\mathcal B(L^2).
  \end{align}
  Therefore, first on a common core, and then hence by bounded extension,
  \begin{align}
    \label{eq:fractional-resolvent-split}
    \langle A\rangle^\rho R_\lambda V f(L_\kappa)
    = R_\lambda \langle A\rangle^\rho V f(L_\kappa) + [\langle A\rangle^\rho,R_\lambda] V f(L_\kappa).
  \end{align}
  The second term is bounded by \eqref{eq:fractional-R-comm}.
  The first term is
  \begin{align}
    \label{eq:fractional-resolvent-factor}
    R_\lambda\langle A\rangle^\rho V f(L_\kappa)
    & = R_\lambda\langle D\rangle^{\alpha/2} \left( \langle D\rangle^{-\alpha/2} \langle A\rangle^\rho V \langle D\rangle^{-\alpha/2} \right) \langle D\rangle^{\alpha/2} f(L_\kappa).
  \end{align}
  By Lemma~\ref{lem:fractional-A-sandwich}, the middle factor is bounded.
  As $\kappa<\kappa_{\rm c}$, the form domain of $L_\kappa$ is $H^{\alpha/2}(\R^d)$. Thus, $\langle D\rangle^{\alpha/2} f(L_\kappa) \in\mathcal B(L^2)$. Moreover, since $V$ is bounded, $H_\kappa$ also has form domain $H^{\alpha/2}(\R^d)$ and hence, $R_\lambda\langle D\rangle^{\alpha/2} \in\mathcal B(L^2)$. Thus, every factor on the right-hand side of \eqref{eq:fractional-resolvent-factor} is bounded, as claimed.
\end{proof}


\section{Propagation estimates}
\label{s:propagationestimates}

In this section, we prove the minimal velocity estimate in Theorem~\ref{the6}, and additionally, outgoing phase space propagation and weighted propagation estimates, for the Schr\"odinger propagation with respect to $H_\kappa = L_\kappa - V$.
We will use these estimates to show the asymptotic completeness claimed in Theorem~\ref{asympcomplhardyplusperturbation}(\ref{asympcomplhardyplusperturbation:acshortrangesmallalpha})(b), and the nonexistence claim in Theorem~\ref{asympcomplhardyplusperturbationnonexistence}(\ref{asympcomplhardyplusperturbation:nonexistencelongrangesmallalpha}) for $\alpha<1$. These estimates are also of independent interest and may be useful to study corresponding many-particle scattering problems.

This section is structured as follows.
In Subsection~\ref{s:minimalvelocity}, we show the minimal escape velocity estimates. While it is intuitively clear that such bounds are useful to show existence (and completeness) of wave operators, they will also be important to prove nonexistence of wave operators involving long-range perturbations.
In Subsection~\ref{s:cookkurodaestimates} we prove a phase space propagation estimate, which is going to be crucial to carry out the Cook--Kuroda method.
In Subsection~\ref{s:weightedmaximalvelocity}, we show weighted propagation estimates which are useful to show the nonexistence claims.


\subsection{Minimal escape velocity estimates}
\label{s:minimalvelocity}

In this subsection, we prove the minimal escape velocity estimates in Theorem~\ref{the6}.
The proof relies on \cite[Corollaries~2.5--2.6]{Skibsted1991}, where the following definition and assumptions are used.

\begin{definition}[{\cite[Definition~2.1]{Skibsted1991}}]
  \label{definitionskibsted}
  Given $\beta,\alpha\geq0$ and $\epsilon>0$, let $\F_{\beta,\alpha,\epsilon}$ denote the set of functions $g,g(x,\tau)=g_{\beta,\alpha,\epsilon}(x,\tau) = -\tau^{-\beta} (-x)^\alpha \chi(\frac{x}{\tau})$, defined for $(x,\tau)\in\R\times\R_+$ and for $\chi\in C^\infty(\R)$ with the properties
  \begin{align*}
    \chi(x) = 1 \quad \text{for} \ x < -2 \epsilon, \qquad \chi(x)=0 \quad \text{for} \ x > -\epsilon, \\
    \frac{d}{dx}\chi(x) \leq 0 \quad \text{and} \quad \alpha\chi(x) + x \frac{d}{dx}\chi(x) = \tilde\chi(x)^2
  \end{align*}
  where $\tilde\chi(x)\geq0$ and $\tilde\chi \in C^\infty(\R)$.
\end{definition}

As noted in \cite{Skibsted1991}, $\F_{\beta,\alpha,\epsilon}\neq\emptyset$ for any parameters $\beta,\alpha,\epsilon$ as in Definition~\ref{definitionskibsted}. In the following, we use the iterated commutator notation
\begin{align}
  {\rm ad}_{A}^{0}(H) = H
  \quad \text{and} \quad
  i^n{\rm ad}_{A}^n(H) = i[i^{n-1} {\rm ad}_{A}^{n-1}(H),A]
  \quad \text{for } n\in\N.
\end{align}

\begin{assumption}[{\cite[Assumption~2.2]{Skibsted1991}}]
  \label{assumskibsted}
  Let $n_0\in\{2,3,...\}$, $t_0>0$, $k_0\geq0$, $\beta_0>0$, $n_0-1/2>\alpha_0>0$, $f,f_2\in C_c^\infty(\R)$, $f_2$ real-valued with $f f_2 = f$, and $H$, $A(\tau)$, $B$ be selfadjoint operators on a Hilbert space $\ch$. Assume with $\tau = t + t_0$, $t \geq 0$, that the operators $A(\tau)$ have a common domain $\D$, $\D(H)\cap\D$ is dense in $\D(H)$, $B \geq 1$, $H$ is bounded from below, and with $A = A(t_0)$ that $\langle A\rangle^{n_0/2} B^{-n_0/2}\in\mathcal B(\mathcal H)$. Assume also the following statements.
  \begin{enumerate}
  \item With ${\rm ad}_{A(\tau)}^{0}(H)=H$ and $1 \leq n \leq n_0$, the (commutator-)form (defined iteratively) $i^n{\rm ad}_{A(\tau)}^n(H)$ on $\D(H)\cap\D$ extends to a symmetric operator with domain $\D(H)$.

  \item If $A$ is unbounded, then $\sup_{|s|<1}\|H\me{i A(\tau)s}\psi\|<\infty$ for all $\psi\in\D(H)$ and $\tau \geq t_0$.

  \item For any $\tau_1,\tau_2 \geq t_0$, $A(\tau_1)-A(\tau_2)$ is bounded, and the derivative $d_{\tau}A(\tau) = \frac{d}{d\tau} A(\tau)$ exists as a bounded operator.

    For $n \leq n_0-1$ and $\tau \geq t_0$ the form (defined iteratively)
    $$
    i^n {\rm ad}_{A(\tau)}^{n}(d_\tau A(\tau)) = i [i^{n-1}{\rm ad}_{A(\tau)}^{n-1}(d_\tau A(\tau)),A(\tau)]
    $$
    on $\D$ extends to a bounded selfadjoint operator on $\ch$.

  \item For $n\leq n_0$, ${\rm ad}_{A(\tau)}^{n}(H) (H-i)^{-1}$ and ${\rm ad}_{A(\tau)}^{n-1}(d_\tau A(\tau))$ are continuous, $\cb(\ch)$-valued functions of $\tau \geq t_0$.

  \item
    \begin{enumerate}
    \item[(a)] ${\rm ad}_{A(\tau)}^{n_0-1}(d_\tau A(\tau)) = \mathcal{O}(\tau^{k_0})$ in $\cb(\ch)$ as $\tau\to\infty$.

    \item[(b)] For $n\leq n_0$, we have ${\rm ad}_{A(\tau)}^{n-1}(d_\tau A(\tau))(H-i)^{-1} = \mathcal{O}(1)$ in $\cb(\ch)$ as $\tau\to\infty$.

    \item[(c)] For $n\leq n_0$, we have ${\rm ad}_{A(\tau)}^{n}(H)(H-i)^{-1} = \mathcal{O}(1)$ in $\cb(\ch)$ as $\tau\to\infty$.
    \end{enumerate}

  \item The statement $q(\beta_0,\alpha_0)$, defined as follows, holds:

    With $\alpha_0'=\max\{m\in\N:\,m<\alpha_0\}$ (for $\alpha_0>1$) the following (positivity) estimate holds for $(\beta,\alpha)=(0,1),...,(0,\alpha_0'),(\beta_0,\alpha_0)$(=$(\beta_0,\alpha_0)$ if $0<\alpha_0\leq1$):

    Let $\mathbb D_HA(\tau)$ denote the symmetric operator $i[H,A(\tau)] + d_\tau A(\tau)$, the Heisenberg derivative. Given $\epsilon>0$ and $g(x,\tau)\in \F_{\beta,\alpha,\epsilon}$, there is $C>0$ so that with
    $$
    \zeta(t) = (g^{(1)}(A(\tau),\tau))^{1/2}\me{-i t H}f(H)B^{-\alpha/2}\phi,
    $$
    one has
    $$
    \int_0^\infty dt \left(\langle \zeta(t), f_2(H) \mathbb D_H A(\tau) f_2(H)\zeta(t)\rangle\right)_-
    \leq C \|\phi\|^2, \qquad \phi\in\ch.
    $$
  \end{enumerate}
\end{assumption}

\begin{proposition}[{\cite[Corollary~2.5]{Skibsted1991}}]
  \label{skibstedcor1}
  Suppose Assumption~\ref{assumskibsted} and in addition
  \begin{align}
    \label{eq:skibstedcor1_1}
    & k_0+\alpha_0+\frac32 < \beta_0+n_0, \\
    \label{eq:skibstedcor1_2}
    & k_0+\alpha_0'+\frac32 < n_0 \quad \text{(if $\alpha_0>1$)}, \\
    \label{eq:skibstedcor1_3}
    & \frac{\alpha_0}{2} + 2 < n_0+\beta_0, \\
    \label{eq:skibstedcor1_4}
    & \frac{\alpha_0'}{2}+2 \leq n_0 \quad \text{(implied by \eqref{eq:skibstedcor1_2})}.
  \end{align}
  Then for $(\beta,\alpha)=(0,1),...,(0,\alpha_0'),(\beta_0,\alpha_0)$, any $\epsilon>0$, $g(x,\tau)\in\F_{\beta,\alpha,\epsilon}$ and $\theta\in[0,1]$,
  \begin{align}
    (-g_{0,\alpha(1-\theta),\epsilon}(A(\tau),\tau))^{1/2}\me{-i t H}f(H)B^{-\alpha/2} = \mathcal{O}(\tau^{(\beta-\alpha\theta)/2}) \quad \text{in $\cb(\ch)$ for $\tau\to\infty$.}
  \end{align}
\end{proposition}

The following remark is taken from \cite[Remark~2.3]{Skibsted1991}.
\begin{remark}
  \label{skibstedremark}
  As for Item~(6) in Assumption~\ref{assumskibsted} $\me{-i t H} f(H) B^{-\frac\alpha2}\phi \!\in\! \D((g^{(1)}(A(\tau),\tau))^{1/2})$ for every $\phi\in\ch$ and $\tau\geq t_0$. A sufficient condition is $f_2(H) \mathbb D_H A(\tau) f_2(H)\geq0$. A more refined sufficient condition for Assumption~\ref{assumskibsted}(6) is in \cite[Corollary~2.6]{Skibsted1991}.
\end{remark}

\begin{proof}[Proof of Theorem~\ref{the6}]
  Theorem~\ref{the6} follows from Proposition~\ref{skibstedcor1} once we verify its assumptions.
  Let us begin with two reductions.
  First, it suffices to prove the theorem where the step function in the statement of Theorem~\ref{the6} is replaced with a smoothed step function of slightly larger support. To prove the estimate with this smoothed step function, we apply Proposition~\ref{skibstedcor1} and Remark~\ref{skibstedremark}.
  Second, it suffices to prove Theorem~\ref{the6} for $\nu\in\N$. Once we have achieved this, \eqref{eq_the6_1} holds for every real $1\leq \nu' \leq \nu$ by complex interpolation with the family $\one(\frac{A}{t}\leq\theta)f(H_\kappa)\me{-itH_\kappa}\langle A\rangle^{-\nu z}$ with $0\leq\re(z)\leq1$, noting that $\|\one(\frac{A}{t}\leq\theta)f(H_\kappa)\me{-itH_\kappa}\|\lesssim1$.
  
  \smallskip
  Let $\epsilon>0$ and $\chi\in C^\infty(\R;[0,1])$ be a nonincreasing function which satisfies
  \begin{align*}
    \chi(s)=
    \begin{cases}
      1 & \mbox{if} \quad s<-\frac{3}{2}\epsilon, \\
      0 & \mbox{if} \quad s>-\epsilon,
    \end{cases}
  \end{align*}
  and $\gamma\chi(s)+s\chi'(s)=\psi(s)^2$ where $0\leq\psi\in C^\infty(\R)$ and $\gamma>0$.
  Let $t_0>0$ be fixed, $\tau=t+t_0$, and define $A(\tau)=A-3\epsilon\tau$.
  We will prove the desired estimate with $\chi(A(\tau)/\tau)$.
  All iterated commutators required in Assumption~\ref{assumskibsted} are first computed on $\D(L_\kappa)\cap\D(A)$. Below we show that $\D(H_\kappa)\cap\D(A)$ is dense in $\D(H_\kappa)$. Hence, the iterated commutators occuring in Assumption~\ref{assumskibsted} extend continuously to $\D(H_\kappa)$, uniformly for $\tau\geq t_0$.
  
  To verify Assumption~\ref{assumskibsted}, we note the following:
  \begin{enumerate}
  \item[(i)] ${\rm ad}_{A(\tau)}^n(L_\kappa)=(-i\alpha)^n L_\kappa$.

  \item[(ii)] ${\rm ad}_{A(\tau)}^n(V)$ is the bounded multiplication operator obtained from $(x\cdot\nabla)^n V$ (up to corresponding powers of $i$), which is bounded and vanishes as $|x|\to\infty$ uniformly in $\tau\geq t_0$.

  \item[(iii)] ${\rm ad}^{0}_{A(\tau)}(d_\tau A(\tau))=-3\epsilon$ and ${\rm ad}^{n}_{A(\tau)}(d_\tau A(\tau))=0$ for $n\geq1$.
  \end{enumerate}
  
  We now verify Assumption~\ref{assumskibsted} with Skibsted's parameters $(\alpha,\beta,\theta)$ in Definition~\ref{definitionskibsted}, Assumption~\ref{assumskibsted}, and Proposition~\ref{skibstedcor1} being denoted by $(\alpha,\beta,\theta)=(\gamma,\beta,\omega)$. We will take $\gamma=2\nu$, $\beta=0$, and $\omega=1$ here.
  We also take $k_0=1/4, \beta_0=1, \alpha_0=2\nu+1/4, n_0=2\nu+2$. Then $\alpha_0'=2\nu$ and \eqref{eq:skibstedcor1_1}--\eqref{eq:skibstedcor1_4} are satisfied.
  Moreover, we take Skibsted's $B=\langle A(t_0)\rangle \sim \langle A\rangle$ and $H=H_\kappa$.

  \smallskip
  First, we verify that $\D(H_\kappa)\cap\D(A)$ is dense in $\D(H_\kappa)$. Indeed, since $\me{isA}L_\kappa\me{-isA}=\me{-\alpha s}L_\kappa$ and $V$ is bounded, we have $\me{isA}\D(H_\kappa)=\D(H_\kappa)$ and $s\mapsto H_\kappa\me{isA}u$ is strongly continuous for $u\in\D(H_\kappa)$. Now mollify $u$ by considering $u_\delta=\int\delta^{-1}\eta(s/\delta)\me{isA}u\,ds$, where $\delta>0$ and $\eta\in C_c^\infty$ with $\int_\R\eta(s)\,ds=1$. Then $u_\delta\in \D(H_\kappa)\cap\D(A)$ and $u_\delta\to u$ as $\delta\to0$ in the graph norm of $H_\kappa$. This shows the density assumption. We now proceed with verifying Items~(1)--(6) from Assumption~\ref{assumskibsted}.
  
  \begin{enumerate}
  \item First,
    \begin{equation}
      \label{eq_the6_2}
      i[H_\kappa,A]
      = i[L_\kappa,A] - i[V,A]
      = \alpha L_\kappa + x \cdot \nabla_x V
    \end{equation}
    holds on $\D(A)\cap\D(L_\kappa)$ and is extended to a symmetric operator on $\D(L_\kappa)$ by Lemma~\ref{lem5}. Therefore, $i^n {\rm ad}_{A(\tau)}^n(H_\kappa)$ also extends to a symmetric operator on $\D(L_\kappa)$ and $i^n {\rm ad}_{A(\tau)}^n(H_\kappa)\langle H_\kappa\rangle^{-1}$ is uniformly bounded in $\tau \geq t_0$ for all $n\in\N$.

  \item On $\D(L_\kappa)$, we have $L_\kappa\me{isA}=\me{\alpha s}\me{isA}L_\kappa$. Hence,
    \begin{equation}
      \label{eq_the6_3}
      \|H_\kappa \me{i A(\tau) s}\phi\|
      \leq \me{\alpha s}  \|L_\kappa\phi\| + \|V\|_\infty\|\phi\|
    \end{equation}
    for all $|s|\leq1$ and $\tau\geq t_0$ and $\phi\in\D(L_\kappa)$ since the scalar part of $\me{iA(\tau)s}$ only contributes a phase.
    
  \item This is trivially satisfied.

  \item This follows from the above notes (i)--(iii).

  \item All the items (5)(a,b,c) follow from the above notes (i)--(iii).

  \item We verify Assumption~\ref{assumskibsted}(6) using Remark~\ref{skibstedremark}. Recall
    \begin{align}
      \label{eq:the6-mourre-decomposition}
      i[H_\kappa,A] = \alpha L_\kappa + x\cdot\nabla V = \alpha H_\kappa+K,
      \quad K := \alpha V + x \cdot \nabla V.
    \end{align}
    By our assumptions on $V$, the function $K$ is bounded and tends to zero as $|x|\to\infty$. In particular, $K\langle D\rangle^{-\alpha/2}$ is compact on $L^2(\R^d)$. Moreover, if $J\subset\subset(0,\infty)$ is a bounded interval compactly contained in $(0,\infty)$, then $\langle D\rangle^{\alpha/2}E_J(H_\kappa)$ is bounded, since $\kappa<\kappa_{\rm c}$, $V\in L^\infty$, and the form domain of $H_\kappa$ is $H^{\alpha/2}(\R^d)$. Hence,
    \begin{align}
      \label{eq:the6-compact-K}
      K E_J(H_\kappa)\quad\text{and}\quad
      E_J(H_\kappa)K E_J(H_\kappa)
      \quad\text{ are compact.}
    \end{align}    
    We next prove a strict Mourre estimate locally at every $\lambda\in\supp f$. Fix such a $\lambda$. By assumption, $\lambda\notin\sigma_{\rm p}(H_\kappa)$, so $E_{\{\lambda\}}(H_\kappa)=0$. For $\delta>0$, let $I_\delta = (\lambda-\delta,\lambda+\delta)$. Taking $\delta>0$ small enough that $I_\delta\subset\subset(0,\infty)$, we have
    $$
    \slim_{\delta\searrow0} E_{I_\delta}(H_\kappa) = E_{\{\lambda\}}(H_\kappa) = 0.
    $$
    It follows from \eqref{eq:the6-compact-K} and compactness that
    \begin{align}
      \label{eq:the6-K-small}
      \lim_{\delta\searrow0}\| E_{I_\delta}(H_\kappa) K E_{I_\delta}(H_\kappa) \| = 0.
    \end{align}
    Choose $\delta_\lambda>0$ so small that $\delta_\lambda<\frac{\lambda}{4}$ and
    $$
    \left\| E_{I_{\delta_\lambda}}(H_\kappa) K E_{I_{\delta_\lambda}}(H_\kappa) \right\|
    \leq \frac{\alpha\lambda}{4}.
    $$
    Then, using \eqref{eq:the6-mourre-decomposition}, we obtain the following strict Mourre estimate
    \begin{align}
      \label{eq:the6-local-strict-mourre}
      \begin{split}
        E_{I_{\delta_\lambda}}(H_\kappa) i[H_\kappa,A] E_{I_{\delta_\lambda}}(H_\kappa)
        & = \alpha E_{I_{\delta_\lambda}}(H_\kappa) H_\kappa E_{I_{\delta_\lambda}}(H_\kappa) + E_{I_{\delta_\lambda}}(H_\kappa) K E_{I_{\delta_\lambda}}(H_\kappa) \\
        & \geq \left( \alpha(\lambda-\delta_\lambda) - \frac{\alpha\lambda}{4} \right) E_{I_{\delta_\lambda}}(H_\kappa)
          \geq \frac{\alpha\lambda}{2} E_{I_{\delta_\lambda}}(H_\kappa).
      \end{split}
    \end{align}
    Since $\supp f$ is compact, finitely many of the intervals $I_{\delta_\lambda}$ cover $\supp f$. Taking a smooth partition of unity subordinate to this finite covering, we can write
    $$
    f=\sum_{j=1}^N f_j,
    \qquad
    f_j\in C_c^\infty(I_j),
    $$
    where for each $j$ there is $c_j>0$ such that
    \begin{align}
      \label{eq:the6-local-mourre-j}
      E_{I_j}(H_\kappa)i[H_\kappa,A]E_{I_j}(H_\kappa) \geq c_j E_{I_j}(H_\kappa).
    \end{align}
    Since the desired estimate is linear in $f$, it suffices to prove it for each $f_j$. We therefore fix $j$ and, to simplify notation, write $f=f_j$. Choose a real-valued function
    $$
    f_2\in C_c^\infty(I_j),
    \qquad
    f_2=1 \quad\text{on }\supp f,
    $$
    so that $f_2f=f$. By \eqref{eq:the6-local-mourre-j},
    \begin{align}
      \label{eq:the6-f2-mourre}
      f_2(H_\kappa)i[H_\kappa,A]f_2(H_\kappa) \geq c_j f_2(H_\kappa)^2.
    \end{align}
    By our definition $ A(\tau)=A-3\epsilon\tau$, we have
    $$
    \mathbb D_H A(\tau)
    = i [H_\kappa,A(\tau)] + d_\tau A(\tau)
    = i [H_\kappa,A]-3\epsilon.
    $$
    Choosing $\epsilon>0$ sufficiently small that $3\epsilon<c_j$, we infer from \eqref{eq:the6-f2-mourre} that
    \begin{align}
      \label{eq:the6-positive-heisenberg}
      \begin{split}
        f_2(H_\kappa) \mathbb D_H A(\tau) f_2(H_\kappa)
        & = f_2(H_\kappa)i[H_\kappa,A]f_2(H_\kappa) - 3\epsilon f_2(H_\kappa)^2 \\
        & \geq (c_j-3\epsilon)f_2(H_\kappa)^2 \geq0.
      \end{split}
    \end{align}
    We can now apply Remark~\ref{skibstedremark}. For every pair $(\beta,\gamma)$ occurring in Assumption~\ref{assumskibsted}(6), every $g\in\F_{\beta,\gamma,\epsilon}$, and every $\phi\in\ch$,
    $$
    \zeta(t)
    = \left(g^{(1)}(A(\tau),\tau)\right)^{1/2} \me{-itH_\kappa} f(H_\kappa)B^{-\gamma/2}\phi
    $$
    is well-defined, and \eqref{eq:the6-positive-heisenberg} yields
    $$
    \left( \left\langle \zeta(t), f_2(H_\kappa) \mathbb D_H A(\tau) f_2(H_\kappa)\zeta(t) \right\rangle \right)_- = 0.
    $$
    Thus,
    $$
    \int_0^\infty \left( \left\langle \zeta(t), f_2(H_\kappa) \mathbb D_H A(\tau) f_2(H_\kappa)\zeta(t) \right\rangle \right)_-\, dt = 0.
    $$
    Thus, Assumption~\ref{assumskibsted}(6) holds, in fact with constant $C=0$.
  \end{enumerate}

  Thus, from the conclusion of Proposition~\ref{skibstedcor1} with $(\beta,\gamma)=(0,2\nu)$, we obtain
  \begin{align}
    \label{eq:skibsted-result}
    \begin{split}
    \left\|\sqrt{(-A(\tau))^{\gamma(1-\omega)} \chi\left(\frac{A(\tau)}{\tau}\right)} f(H_\kappa) \me{-i t H_\kappa}\phi \right\|
      & \lesssim \tau^{-\gamma\omega/2} \|\langle A(t_0)\rangle^{\gamma/2}\phi\| \\
      & \lesssim t^{-\gamma\omega/2} \|\langle A\rangle^{\gamma/2}\phi\|
    \end{split}
  \end{align}
  for $\tau \gg 1$ and $0\leq\omega\leq1$.

  We now conclude the proof by passing from the smooth cutoff in Proposition~\ref{skibstedcor1} to the sharp cutoff in \eqref{eq_the6_1}. Recall $\tau=t+t_0$, $A(\tau)=A-3\epsilon\tau$, and $\chi\in C^\infty(\R;[0,1])$ with $\chi(s)=1$ for $s\leq-3\epsilon/2$. Let $P_t:=\one\left(\frac{A}{t}\leq\epsilon\right)$. If $a/t\leq\epsilon$, then, since $\tau>t$, we have $a/\tau-3\epsilon\leq-2\epsilon$. Hence, $\chi(a/\tau-3\epsilon)=1$ on the spectral support of $P_t$. Because both operators are functions of $A$, we have
  \begin{align}
    \label{eq:sharp-smooth-Skibsted}
    P_t = P_t\chi\left(\frac{A(\tau)}{\tau}\right),
    \qquad
    \|P_tu\| \leq \left\| \chi\left(\frac{A(\tau)}{\tau}\right)u \right\|.
  \end{align}
  Applying \eqref{eq:sharp-smooth-Skibsted} to $u=f(H_\kappa)\me{-itH_\kappa}\phi$ and using \eqref{eq:skibsted-result} with $(\gamma,\omega)=(2\nu,1)$ yields
  $$
  \left\| \one\left(\frac{A}{t} \leq\epsilon\right) f(H_\kappa)\me{-itH_\kappa}\phi \right\|
  \lesssim t^{-\nu} \|\langle A\rangle^\nu\phi\|.
  $$
  Thus, \eqref{eq_the6_1} follows with $\theta=\epsilon$.
\end{proof}


\subsection{Outgoing phase space propagation estimate}
\label{s:cookkurodaestimates}

In this subsection, we will prove the following spectrally localized outgoing phase space propagation estimate which is important to prove Theorem~\ref{asympcomplhardyplusperturbation}(\ref{asympcomplhardyplusperturbation:acshortrangesmallalpha})(b).

\begin{lemma}
  \label{the11suff-unified}
  Let $d\in\N$, $\alpha\in(1/2,2]\cap(1/2,d)$, and $\kappa<\kappa_{\rm c}$.
  Assume $V^S$ satisfies \eqref{eq:basicassumpshortrange} and \eqref{ass5} for some $\rho_S>1$. Let $\sigma = \Phi^{-1}(\kappa)<(d-1)/2$ if $\alpha\leq1$ and $H_\kappa=L_\kappa-V^S$.
  Let
  \begin{align}
    \label{eq:mu-unified}
    1 < \mu < \min\{\rho_S,2,2\alpha,d-2\sigma\},
  \end{align}
  $$
  K_0,K_1 \in \{L_\kappa,H_\kappa\},
  $$
  and $f_0,f_1\in C_c^\infty(0,\infty)$ be real-valued.
  Let $\theta>0$, and choose a nondecreasing $\chi\in C^\infty(\R;[0,1])$ such that
  $$
  \chi(s)=0\quad \text{for} \ s\leq\theta/4,
  \qquad
  \chi(s)=1\quad \text{for} \ s\geq\theta/2.
  $$
  Then
  \begin{align}
    \label{eq:unified-localized-Cook}
    \left\| f_1(K_1)\me{itK_1} V^S \chi\left(\frac At\right)^2 \me{-itK_0}f_0(K_0) \right\|
    \lesssim_{f_0,f_1} t^{-\mu}, \quad t\geq1.
  \end{align}
\end{lemma}

\begin{proof}
  We first verify that
  \begin{align}
    \label{eq:unified-spectral-regularity}
    \langle D\rangle^{\mu/2}f_j(K_j) \in\mathcal B(L^2), \qquad j\in\{0,1\}.
  \end{align}
  Since $\mu/\alpha<2$ and $\mu<d-2\sigma$, Proposition~\ref{eqsob} yields
  \begin{align}
    \label{eq:unified-Sobolev}
    \| |D|^{\mu/2}u \| \lesssim \|L_\kappa^{\mu/(2\alpha)} u\|.
  \end{align}
  If $K_j=L_\kappa$, then
  $$
  L_\kappa^{\mu/(2\alpha)} f_j(L_\kappa) \in\mathcal B(L^2)
  $$
  by the spectral theorem. If $K_j=H_\kappa$, then by the boundedness of $V^S$, we have $\D(H_\kappa) = \D(L_\kappa)$ with equivalent graph norms. In particular,
  $$
  L_\kappa f_j(H_\kappa)
  = H_\kappa f_j(H_\kappa) + V^S f_j(H_\kappa) \in \mathcal B(L^2).
  $$
  Interpolation for the nonnegative selfadjoint operator $L_\kappa$ therefore gives
  $$
  L_\kappa^{\mu/(2\alpha)} f_j(H_\kappa) \in \mathcal B(L^2).
  $$
  Together with \eqref{eq:unified-Sobolev}, this proves \eqref{eq:unified-spectral-regularity}. Let
  $$
  w(x) = \langle x\rangle^{-\mu/2}.
  $$
  By Lemma~\ref{lem12},
  $$
  Q_\mu := \langle D\rangle^{-\mu/2} w\langle A\rangle^\mu w \langle D\rangle^{-\mu/2}
  $$
  is bounded. Hence, \eqref{eq:unified-spectral-regularity} implies
  \begin{align}
    \label{eq:unified-weighted-A}
    f_j(K_j)w\langle A\rangle^\mu wf_j(K_j) \in\mathcal B(L^2),
    \qquad j\in\{0,1\}.
  \end{align}
  We next claim that, if $\eta\in S^0(\R)$ with $\supp\eta\subset[c,\infty)$ for some $c>0$, then
  \begin{align}
    \label{eq:unified-weighted-outgoing}
    \left\| w\eta\left(\frac At\right) \me{-itK_j}f_j(K_j) \right\|
    \lesssim_{\eta,f_j} t^{-\mu/2}, \qquad t\geq1.
  \end{align}
  Indeed, functional calculus for $A$ gives
  $$
  \eta\left(\frac At\right)^* \eta\left(\frac At\right)
  \lesssim t^{-\mu}\langle A\rangle^\mu.
  $$
  Thus, for
  $$
  u_{j,t} = \me{-itK_j} f_j(K_j)u, \quad u \in L^2
  $$
  estimate \eqref{eq:unified-weighted-A} yields
  $$
  \left\| \eta\left(\frac At\right)w u_{j,t} \right\|
  \lesssim_{f_j} t^{-\mu/2} \|u\|.
  $$
  Moreover, since $[w,A]=i\, x\cdot\nabla w$ is bounded, by Lemma~\ref{lem:functional-calculus-commutator}(1) with $X=w$ and $f=\eta$ (with $\eta\in C^\infty$ bounded and $\eta'$ compactly supported), we have
  \begin{align}
    \label{eq:unified-commutator-w-eta}
    \left\| \left[w,\eta\left(\frac At\right)\right] \right\| \lesssim_\eta t^{-1}.
  \end{align}
  Since $\mu<2$, we have $t^{-1}\lesssim t^{-\mu/2}$ for $t\geq1$. Therefore,
  \begin{align*}
    \left\| w\eta\left(\frac At\right)u_{j,t} \right\|
    & \leq \left\| \eta\left(\frac At\right)wu_{j,t} \right\| + \left\| \left[w,\eta\left(\frac At\right)\right]u_{j,t} \right\|
      \lesssim_{\eta,f_j} t^{-\mu/2} \|u\|,
  \end{align*}
  which proves \eqref{eq:unified-weighted-outgoing}.

  Let $u,v\in L^2$ and
  $$
  \chi_t=\chi(A/t),
  \qquad
  u_t = \me{-itK_0}f_0(K_0)u,
  \qquad
  v_t = \me{-itK_1}f_1(K_1)v.
  $$
  We have
  $$
  V^S\chi_t^2 = \chi_tV^S\chi_t+[V^S,\chi_t]\chi_t.
  $$
  Since $\mu<\rho_S$,
  $$
  |V^S(x)| \lesssim w(x)^2.
  $$
  Therefore, by \eqref{eq:unified-weighted-outgoing},
  \begin{align}
    \label{eq:unified-main-term}
    \left| \langle v_t, \chi_t V^S \chi_t u_t\rangle \right|
    = \left| \langle \chi_t v_t, V^S\chi_tu_t\rangle \right|
    \lesssim \| w \chi_t u_t \| \, \| w \chi_t v_t\|
    \lesssim_{f_0,f_1} t^{-\mu} \|u\| \|v\|.
  \end{align}
  It remains to estimate the $[V^S,\chi_t]\chi_t$ term. By Lemma~\ref{lem:functional-calculus-commutator}(2) with $X=V^S$ and $f=\chi$ (whose assumptions are satisfied since $x\cdot\nabla V^S, x\cdot\nabla(x\cdot\nabla V^S) \in L^\infty$),
  \begin{align}
    \label{eq:unified-commutator-expansion}
    [V^S,\chi_t] = \frac1t\chi'\left(\frac At\right)[V^S,A] + \ce_t,
    \qquad
    \|\ce_t\| \lesssim t^{-2}
  \end{align}
  with
  \begin{align*}
    \ce_t = \frac{1}{2\pi i t^2} \int_\C \overline{\partial}\tilde\chi(z) \left(z-\frac At\right)^{-2} \, [[V^S,A],A]\, \left(z-\frac{A}{t}\right)^{-1}\,dz \wedge d\overline{z},
  \end{align*}
  which indeed obeys $\|\ce_t\| \lesssim t^{-2} \|(x\cdot \nabla)^2 V^S\|_\infty$. Here, we have used that $[V^S,A] = i \, x\cdot\nabla V^S$ and $[[V^S,A],A] = -(x\cdot\nabla)^2 V^S$ are bounded.
  Since $\chi,\chi'\in S^0(\R)$ and are supported in a positive half-line and we have $|[V^S,A](x)| \lesssim \langle x\rangle^{-\rho_S} \lesssim w(x)^2$, \eqref{eq:unified-weighted-outgoing} gives
  \begin{align*}
    \left| \left\langle v_t, \chi'\left(\frac At\right) [V^S,A]\chi_tu_t \right\rangle \right|
    & \lesssim \|w \chi_t u_t\| \, \left\| w \chi'\left(\frac At\right) v_t \right\| 
      \lesssim_{f_0,f_1} t^{-\mu} \|u\| \|v\|.
  \end{align*}
  It follows from \eqref{eq:unified-commutator-expansion} that
  $$
  \left| \langle[V^S,\chi_t]\chi_tu_t,v_t\rangle \right|
  \lesssim_{f_0,f_1} (t^{-1-\mu}+t^{-2}) \|u\| \|v\|
  \lesssim t^{-\mu}\|u\| \|v\|
  $$
  since $1<\mu<2$. Combining this with \eqref{eq:unified-main-term} proves \eqref{eq:unified-localized-Cook}.
\end{proof}


\subsection{Weighted propagation estimates}
\label{s:weightedmaximalvelocity}

In this subsection, we prove weighted propagation estimates. The first estimate covers $\alpha\geq1$ and uses the Helffer--Sj\"ostrand formula together with the following weighted estimates.

\begin{lemma}
  \label{lem8}
  Let $d\in\N$, $\alpha\in[1,2]\cap(0,d)$, $\kappa<\kappa_{\rm c}$, and $\rho>0$.
  Assume $V \in L^\infty(\R^d)$ satisfies $|V(x)| \lesssim \langle x\rangle^{-\rho}$.
  Let $H_\kappa=L_\kappa-V$, $z\in\C\setminus\R$, and $\mu\in[0,1]$.
  Then there exists a positive function $Z(z,\mu)$, which grows at most polynomially in $\langle z\rangle$ and $|\im(z)|^{-1}$, such that
  \begin{align}
    \label{eq_lem8_1}
    \|\langle x\rangle^\mu(z-H_\kappa)^{-1}\langle x\rangle^{-\mu}\|
    & \lesssim Z(z,\mu)
  \end{align}
  holds for $0\leq\mu\leq1$.
  Moreover, this statement also holds for $1<\mu\leq2$ whenever $\alpha\geq1$ if $d\geq3$ and $\alpha>1$ if $d=2$.   
\end{lemma}

We provide quantitative upper bounds for $Z(z,\mu)$ in the proof below. We note that a slightly longer argument yields analogous estimates for $\alpha\in(0,1)$ and $\mu\in(0,\alpha)$. However, we were unable to use it to prove corresponding propagation bounds for $\alpha<1$; in fact, in Lemma~\ref{lem9-small-alpha}, we will prove weighted propagation bounds for $L_\kappa$ whose proof proceeds rather differently.

\begin{proof}
  We first prove \eqref{eq_lem8_1} for $\mu=1$. We compute
  \begin{align}
    \label{eq_lem8_3}
    \begin{split}
      [x_j,(z-H_\kappa)^{-1}]
      & = (z-H_\kappa)^{-1}[x_j,|D|^\alpha](z-H_\kappa)^{-1} \\
      & = i\, \alpha (z-H_\kappa)^{-1}|D|^{\alpha-2}D_j(z-H_\kappa)^{-1}.
    \end{split}
  \end{align}
  Note that
  \begin{align}
    \label{eq_lem8_4}
    \langle H_\kappa\rangle^{-1}|D|^{\alpha-2}D_j
    = \langle H_\kappa\rangle^{-1}\langle D\rangle^{\alpha/2}\langle D\rangle^{-\alpha/2}|D|^{\alpha-2}D_j \in \mathcal B(L^2)
  \end{align}
  since $\kappa<\kappa_{\rm c}$ and $0\leq\alpha-1\leq\alpha/2$. Thus, together with
  $$
  \left\| (z-H_\kappa)^{-1} \langle H_\kappa\rangle\right\|
  = \sup_{\lambda \in \sigma(H_\kappa)} \frac{\sqrt{1+\lambda^2}}{|\lambda - z|}
  \lesssim \frac{\langle z\rangle}{|\im(z)|},
  $$
  we have
  \begin{align*}
    \|x_j(z-H_\kappa)^{-1}\langle x\rangle^{-1}\|
    \lesssim |\im(z)|^{-1} + |\im(z)|^{-2} \langle z\rangle.
  \end{align*}
  This, and $\langle x\rangle\lesssim1+\sum_{j=1}^d|x_j|$,
  show \eqref{eq_lem8_1} with
  \begin{align*}
    Z(z,1) = |\im(z)|^{-1} + |\im(z)|^{-2} \langle z\rangle.
  \end{align*}
  By complex interpolation, \eqref{eq_lem8_1} therefore holds for all $0\leq\mu\leq1$.

  Next, we prove \eqref{eq_lem8_1} for $\mu=2$ with a similar argument.
  We have
  \begin{align}
    \label{eq_lem8_5}
    [x_j^2,(z-H_\kappa)^{-1}]
    = (z-H_\kappa)^{-1}[x_j^2,|D|^\alpha](z-H_\kappa)^{-1},
  \end{align}
  and
  \begin{align}
    \label{eq_lem8_6}
    [x_j^2,|D|^\alpha]
    = i\alpha \left(2|D|^{\alpha-2}D_j x_j + i |D|^{\alpha-2} + i (\alpha-2) |D|^{\alpha-4} D_j^2 \right).
  \end{align}
  Since $(|D|^{\alpha-2}D_j)^2 \leq |D|^{2\alpha-2}$ and $(|D|^{\alpha-4}D_j^2)^2\leq |D|^{2\alpha-4}$, it suffices to consider $|D|^{\alpha-1}x_j+|D|^{\alpha-2}$ instead of \eqref{eq_lem8_6}. Let us write
  \begin{align}
    \begin{split}
      & (z-H_\kappa)^{-1}(|D|^{\alpha-1}x_j+|D|^{\alpha-2})(z-H_\kappa)^{-1}\langle x\rangle^{-1} \\
      & \quad = (z-H_\kappa)^{-1}(|D|^{\alpha-1}x_j\langle x\rangle^{-1}+|D|^{\alpha-2}\langle x\rangle^{-1})\langle x\rangle(z-H_\kappa)^{-1}\langle x\rangle^{-1}.
    \end{split}
  \end{align}
  Since $|D|^{\alpha-2}\langle x\rangle^{-1}=\left(|D|^{\alpha-2}|x|^{\alpha-2}\right) \left(|x|^{2-\alpha}\langle x\rangle^{-1}\right)$ is bounded for $\alpha\geq1$ if $d\geq3$ and $\alpha>1$ if $d=2$ by Hardy's inequality, we have
  \begin{align}
    \label{eq_lem8_7}
    \|(z-H_\kappa)^{-1}(|D|^{\alpha-1}x_j+|D|^{\alpha-2})(z-H_\kappa)^{-1}\langle x\rangle^{-1}\|
    \lesssim \left(|\im(z)|^{-2}+|\im(z)|^{-3} \langle z\rangle\right) \langle z\rangle
  \end{align}
  using the bound in the case $\mu=1$.
  By \eqref{eq_lem8_5} and \eqref{eq_lem8_7} we have \eqref{eq_lem8_1} with $\mu=2$ and
  \begin{align}
    \begin{split}
      Z(z,2)
      & = |\im(z)|^{-1}+|\im(z)|^{-3}+|\im(z)|^{-4} \langle z\rangle \\
      & \quad + (|\im(z)|^{-2} + |\im(z)|^{-3} \langle z\rangle)\langle z\rangle (|\im(z)|^{-1} + |\im(z)|^{-2} \langle z\rangle).
    \end{split}
  \end{align}
  By complex interpolation between $\mu=0$ and $\mu=2$, we obtain \eqref{eq_lem8_1} for all $\mu\in[0,2]$ with $Z(z,\mu) = |\im(z)|^{\mu/2-1} Z(z,2)^{\mu/2}$.
\end{proof}

We are now ready to state our weighted propagation estimates.
\begin{lemma}
  \label{lem9}
  Let $d\in\N$, $\alpha\in[1,2]\cap(0,d)$, $\kappa<\kappa_{\rm c}$, $\rho>0$, $\mu\in[0,1]$, and $f\in C_c^\infty(0,\infty)$.
  Assume $V \in L^\infty(\R^d)$ satisfies $|V(x)| \lesssim \langle x\rangle^{-\rho}$.
  Then, for $H_\kappa=L_\kappa-V$,
  \begin{equation}
    \label{eq_lem9_1}
    \|\langle x\rangle^\mu f(H_\kappa) \me{-i t H_\kappa}\langle x\rangle^{-\mu}\|
    \lesssim \langle t\rangle^\mu, \quad t\in\R.
  \end{equation}
  Moreover, \eqref{eq_lem9_1} also holds for $1<\mu\leq2$ for all $d\geq3$, or $d=2$ and $\alpha>1$.
\end{lemma}

\begin{proof}
  We first prove \eqref{eq_lem9_1} when $\mu=1$. By the Helffer--Sj\"ostrand formula \eqref{eq:hs},
  \begin{equation}
    \label{eq_lem9_2}
    [x_j,f(H_\kappa)]
    = \frac{\alpha}{2\pi} \int_{\C}\bar{\partial} \tilde{f}(z)(z-H_\kappa)^{-1}|D|^{\alpha-2}D_j(z-H_\kappa)^{-1} \, dz \wedge d\bar{z}.
  \end{equation}
  Note that
  \begin{equation}
    \label{eq_lem9_3}
    x_j \me{-i t H_\kappa}
    = \me{-i t H_\kappa}x_j + \alpha\int_0^t \me{-i(t-s) H_\kappa} |D|^{\alpha-2} D_j \me{-i s H_\kappa} \, ds.
  \end{equation}
  Therefore,
  \begin{equation}
    \label{eq_lem9_5}
    \|x_jf(H_\kappa) \me{-i t H_\kappa}\langle x\rangle^{-1}\|
    \leq \|f(H_\kappa) x_j \me{-i t H_\kappa}\langle x\rangle^{-1}\| + \|[x_j,f(H_\kappa)] \me{-i t H_\kappa}\|
    \lesssim \langle t\rangle
  \end{equation}
  follows from \eqref{eq_lem8_4}, \eqref{eq_lem9_3}, and $\langle x\rangle\lesssim1+\sum_{j=1}^d|x_j|$. By complex interpolation between $\mu=0$ and $\mu=1$, we have \eqref{eq_lem9_1} for all $0\leq\mu\leq1$.

  Next, we prove \eqref{eq_lem9_1} when $\mu=2$. It suffices to prove
  \begin{equation}
    \label{eq_lem9_6}
    \|x_j^2 f(H_\kappa) \me{-i t H_\kappa}g(H_\kappa)\langle x\rangle^{-2}\|
    \lesssim \langle t\rangle^2
  \end{equation}
  for $g\in C_c^\infty(0,\infty)$ such that $f = f g$. We have
  \begin{align}
    \label{eq_lem9_10}
    \begin{split}
      [x_j^2,f(H_\kappa)]
      & = \frac{\alpha}{2\pi}\int_{\C}\bar{\partial} \tilde{f}(z)(z-H_\kappa)^{-1}\left\{2|D|^{\alpha-2}D_jx_j\right. \\
      & \qquad\qquad \left. + i |D|^{\alpha-2} + (\alpha-2)\, i|D|^{\alpha-4}D_j^2\right\}(z-H_\kappa)^{-1} \, dz \wedge d\bar{z}.
    \end{split}
  \end{align}
  Since
  \begin{align*}
    \|[(z-H_\kappa)^{-1},x_j]\|
    & = \alpha\|(z-H_\kappa)^{-1}|D|^{\alpha-2}D_j(z-H_\kappa)^{-1}\| \\
    & \lesssim |\im(z)|^{-2} \langle z\rangle,
    \quad z \in \supp(\tilde f)
  \end{align*}
  by \eqref{eq_lem8_3} and \eqref{eq_lem8_4}, we have
  \begin{align}
    \label{eq_lem9_7}
    \begin{split}
      & \||D|^{\alpha-2}D_j x_j(z-H_\kappa)^{-1}g(H_\kappa) \me{-i t H_\kappa}\langle x\rangle^{-1}\| \\
      & \quad \lesssim \langle t\rangle |\im(z)|^{-1} \langle z\rangle
        + |\im(z)|^{-2} \langle z\rangle, \quad z\in\supp(\tilde f)
    \end{split}
  \end{align}
  using $\|D_j|D|^{-1}\|\leq1$, \eqref{eq_lem8_4}, and \eqref{eq_lem9_5}. We now consider the second and third summand in \eqref{eq_lem9_10}; for the third summand, we use $\|D_j^2|D|^{-2}\|\leq1$. We have
  \begin{align}
    \label{eq_lem9_8}
    \begin{split}
      & \||D|^{\alpha-2}(z-H_\kappa)^{-1}g(H_\kappa) \me{-i t H_\kappa}\langle x\rangle^{\alpha-2}\| \\
      & \quad \lesssim \||D|^{\alpha-2}|x|^{\alpha-2}\|\|\langle x\rangle^{2-\alpha}(z-H_\kappa)^{-1}\langle x\rangle^{\alpha-2}\| \|\langle x\rangle^{2-\alpha}g(H_\kappa) \me{-i t H_\kappa}\langle x\rangle^{\alpha-2}\| \\
      & \quad \lesssim Z(z,2-\alpha) \langle t\rangle^{2-\alpha}
    \end{split}
  \end{align}
  by $|x|^{\alpha-2}\langle D\rangle^{\alpha-2}\lesssim1$, Lemmas~\ref{lem8} and~\eqref{eq_lem9_1} with $\mu=2-\alpha$. By \eqref{eq_lem9_7} and \eqref{eq_lem9_8},
  \begin{equation}
    \label{eq_lem9_9}
    \|[x_j^2,f(H_\kappa)] \me{-i t H_\kappa}g(H_\kappa)\langle x\rangle^{-1}\|
    \lesssim \langle t\rangle,
  \end{equation}
  where we used that $\tilde f$ is compactly supported and can be chosen of sufficiently high order to compensate the inverse powers of $|\im(z)|$. In particular,
  \begin{equation}
    \label{eq_lem9_11}
    \|\langle x\rangle^2f(H_\kappa)\langle x\rangle^{-2}\| \lesssim 1
  \end{equation}
  by applying the Helffer--Sj\"ostrand formula to $f(H_\kappa)$ and using Lemma~\ref{lem8} to control $\langle x\rangle^2(z-H_\kappa)^{-1}\langle x\rangle^{-2}$ in the resulting integrand.
  By \eqref{eq_lem9_9}, to prove \eqref{eq_lem9_6}, it suffices to prove
  \begin{equation}
    \label{eq_lem9_12}
    \|f(H_\kappa) x_j^2 \me{-i t H_\kappa}g(H_\kappa)\langle x\rangle^{-2}\|
    \lesssim \langle t\rangle^2.
  \end{equation}
  Using \eqref{eq_lem9_3}, we compute
  \begin{align}
    \label{eq_lem9_13}
    \begin{split}
      x_j^2 \me{-i t H_\kappa}
      & = \me{-i t H_\kappa}x_j^2 + \alpha \int_0^t \me{-i (t-s) H_\kappa}|D|^{\alpha-2} D_j \me{-i s H_\kappa} x_j \, ds \\
      & \quad + \alpha \int_0^t \me{-i (t-s) H_\kappa} x_j |D|^{\alpha-2} D_j \me{-i s H_\kappa}\, ds \\
      & \quad + \alpha^2 \int_0^t \int_0^{t-s} \me{-i (t-s-r) H_\kappa} |D|^{\alpha-2} D_j \me{-i r H_\kappa} |D|^{\alpha-2} D_j \me{-i s H_\kappa} \, dr \, ds.
    \end{split}
  \end{align}
  The second summand on the right-hand side of \eqref{eq_lem9_13} contributes $\mathcal O(t)$. Indeed, writing $M_j:=|D|^{\alpha-2}D_j$, we see
  \begin{align}
    \label{eq:spectral-Mj-bounds}
    f(H_\kappa)M_j, \quad M_j f(H_\kappa)\in\mathcal B(L^2).
  \end{align}
  by inserting $1=\langle H_\kappa\rangle\langle H_\kappa\rangle^{-1}$ and using \eqref{eq_lem8_4}. Hence, we also have
  $$
  f(H_\kappa)\me{-i(t-s)H_\kappa} M_j
  = \left(f(H_\kappa)\langle H_\kappa\rangle \, \me{-i(t-s)H_\kappa}\right) \left(\langle H_\kappa\rangle^{-1} M_j\right) \in \mathcal B(L^2).
  $$
  Furthermore,
  $$
  x_j g(H_\kappa)\langle x\rangle^{-2}
  = g(H_\kappa) x_j \langle x\rangle^{-2} + [x_j,g(H_\kappa)]\langle x\rangle^{-2} \in \mathcal B(L^2).
  $$
  The boundedness is clear for the first summand; to see this for the second summand, one may write $[x_j,g(H_\kappa)]$ as in \eqref{eq_lem9_2} in terms of the Helffer--Sj\"ostrand formula and then use that the operator norm of
  $$
  (H_\kappa-z)^{-1}M_j(H_\kappa-z)^{-1}
  = \left((H_\kappa-z)^{-1}\langle H_\kappa\rangle\right) \left(\langle H_\kappa\rangle^{-1} M_j\right) (H_\kappa-z)^{-1}
  $$
  is bounded (by \eqref{eq_lem8_4}) by a function which is integrable against a suitable almost-analytic extension of $g$.
  By a similar argument, the contribution from the double-integral in the last line of \eqref{eq_lem9_13} is $\mathcal O(t^2)$. Indeed, by using the spectral localizations and \eqref{eq:spectral-Mj-bounds},
  \begin{align*}
    & \left\| f(H_\kappa)\me{-i(t-s-r)H_\kappa} M_j \me{-irH_\kappa} M_j \me{-isH_\kappa} g(H_\kappa)\langle x\rangle^{-2} \right\| \\
    & \quad = \left\| \me{-i(t-s-r)H_\kappa} \left(f(H_\kappa)M_j\right) \me{-irH_\kappa} \left(M_jg(H_\kappa)\right)\me{-isH_\kappa} \langle x\rangle^{-2} \right\| \lesssim1,
  \end{align*}
  uniformly for $r,s\geq0$. Integration over $\{(s,r):0\leq s\leq t,\ 0\leq r\leq t-s\}$ therefore gives an $\mathcal O(t^2)$ contribution.

  It remains to estimate the third single integral in \eqref{eq_lem9_13}. We use $x_jM_j=M_jx_j+[x_j,M_j]$, where
  \begin{align}
    \label{eq:xj-Mj-comm-repair}
    [x_j,M_j] = i|D|^{\alpha-2} + (\alpha-2)i|D|^{\alpha-4}D_j^2.
  \end{align}
  For the first term we obtain, using \eqref{eq:spectral-Mj-bounds} and the already proved $\mu=1$ case,
  \begin{align*}
    \left\| f(H_\kappa)M_jx_j \me{-isH_\kappa} g(H_\kappa)\langle x\rangle^{-2} \right\|
    \lesssim \left\| x_j g(H_\kappa)\me{-isH_\kappa}\langle x\rangle^{-1}\right\|
    \lesssim \langle s\rangle.
  \end{align*}
  We treat the commutator term next. To this end, let $\beta=2-\alpha$. Then both summands in \eqref{eq:xj-Mj-comm-repair} are of the form $Q|D|^{-\beta}$ with $Q\in\mathcal B(L^2)$ being a bounded Fourier multiplier. By $|D|^{-\beta}\langle x\rangle^{-\beta}\in\mathcal B(L^2)$ and the already proved weighted propagation estimate with $\mu=\beta\leq1$,
  \begin{align*}
    \left\| |D|^{-\beta}g(H_\kappa)\me{-isH_\kappa} \langle x\rangle^{-2} \right\| 
    & \leq \left\| |D|^{-\beta}\langle x\rangle^{-\beta}\right\| \left\| \langle x\rangle^\beta g(H_\kappa)\me{-isH_\kappa} \langle x\rangle^{-\beta} \right\| \left\| \langle x\rangle^{\beta-2}\right\|_\infty \\
    & \lesssim \langle s\rangle^\beta \lesssim \langle s\rangle.
  \end{align*}
  Thus, the integrand of the third term in \eqref{eq_lem9_13} is $\mathcal O(\langle s\rangle)$, and hence the third term contributes $\mathcal O(t^2)$. Together with the previous estimates, this proves \eqref{eq_lem9_12} and thus~\eqref{eq_lem9_6}.

  Finally, by complex interpolation between $\mu=0$ and $\mu=2$, we obtain \eqref{eq_lem9_1} for all $0\leq\mu\leq2$.
\end{proof}

We now consider the more difficult regime $\alpha<1$. Here we prove a weighted propagation estimate only for $L_\kappa$. The reason why we do not treat an external potential here is that our proof uses the equivalence of Sobolev norms in Proposition~\ref{eqsob} and we were unable to use a generalization of Lemma~\ref{lem8} for $\alpha<1$. On the positive side, the following estimate also covers $\kappa=\kappa_{\rm c}$ (unlike Lemma~\ref{lem9} above), and it suffices to prove the nonexistence of $W(L_\kappa-V^L,L_\kappa)$ when $0<\rho_L\leq\alpha\leq1$.

\begin{lemma}
  \label{lem9-small-alpha}
  Let $d\in\N$, $\alpha\in(1/2,1)\cap(0,d)$, $\kappa\leq\kappa_{\rm c}$, and let $f\in C_c^\infty(0,\infty)$. Then,
  \begin{align}
    \label{eq:lem9-small-alpha}
    \left\| \langle x\rangle^\mu f(L_\kappa)e^{-itL_\kappa} \langle x\rangle^{-\mu} \right\|
    \lesssim_f \langle t\rangle^\mu, \quad 0\leq\mu\leq1.
  \end{align}
  The same conclusion holds for $\alpha=1/2$ provided $d\geq2$.
\end{lemma}

\begin{proof}
  We divide the proof into four steps.

  \medskip
  \noindent
  {\bf Step 1: Localization of the velocity.}
  Set $\beta:=1-\alpha\in(0,1/2)$ and
  $$
  s_\alpha:=\frac{2\beta}{\alpha}=\frac{2(1-\alpha)}{\alpha}.
  $$
  Since $\alpha>1/2$, we have $0<s_\alpha<2$. Moreover, $s_\alpha<\frac d\alpha$, which is automatic since $\alpha>1/2$ and $d\geq1$. We now apply Proposition~\ref{eqsob}(2) to $V(x)=\kappa/|x|^\alpha$. Since $L_0^{s_\alpha/2}=|D|^{\alpha s_\alpha/2}=|D|^\beta$, Proposition~\ref{eqsob}(2), implies that $|D|^{-\beta} L_\kappa^{s_\alpha/2}\in\mathcal B(L^2)$.
  
  Let $h\in C_c^\infty(0,\infty)$. Since $L_\kappa^{-s_\alpha/2}h(L_\kappa) \in\mathcal B(L^2)$ by the spectral theorem, we obtain
  \begin{align}
    \label{eq:negative-sobolev-spectral}
    |D|^{-\beta}h(L_\kappa)
    = (|D|^{-\beta}L_\kappa^{s_\alpha/2})\, (L_\kappa^{-s_\alpha/2}h(L_\kappa)) \in\mathcal B(L^2).
  \end{align}
  Introduce $\alpha^{-1}$ times the $j$-th component of the velocity,
  \begin{align}
    v_j := |D|^{\alpha-2}D_j = \frac{D_j}{|D|}\,|D|^{-\beta}.
  \end{align}
  Since $D_j/|D|$ is bounded on $L^2$,
  \begin{align}
    \label{eq:localized-velocity-left}
    h(L_\kappa)v_j\in\mathcal B(L^2).
  \end{align}
  At the endpoint $\alpha=1/2$, one has $s_\alpha=2$. Proposition~\ref{eqsob}(2) is still applicable provided $2<d/\alpha=2d$, i.e., $d\geq2$. Thus, \eqref{eq:negative-sobolev-spectral}--\eqref{eq:localized-velocity-left} remain valid in that case. It is in this step where we need the assumption $\alpha\geq1/2$.

  \medskip
  \noindent
  {\bf Step 2: Local $C^1(x_j)$ regularity of $h(L_\kappa)$.}
  We show that, for every $j\in\{1,2,...,d\}$ and $h\in C_c^\infty(0,\infty)$, one has
  \begin{align}
    \label{eq:hL-C1-x}
    h(L_\kappa)\mathcal D(x_j)\subset\mathcal D(x_j),
    \qquad
    [x_j,h(L_\kappa)]\in\mathcal B(L^2).
  \end{align}
  To this end, let
  $$
  U_j(\tau):=\me{i\tau x_j},
  \qquad
  B_{j,\tau}:=\frac{U_j(\tau)-1}{\tau},
  \qquad
  0<|\tau|\leq1.
  $$
  We first prove
  \begin{align}
    \label{eq:uniform-difference-comm}
    \sup_{0<|\tau|\leq1} \|[B_{j,\tau},h(L_\kappa)]\|<\infty,
  \end{align}
  which eventually allows us to pass to the limit $\tau\to0$. To prove \eqref{eq:uniform-difference-comm}, first note that
  \begin{align*}
    U_j(\tau) \frac{\kappa}{|x|^\alpha} U_j(\tau)^* & = \frac{\kappa}{|x|^\alpha}, \\
    U_j(\tau)|D|^\alpha U_j(\tau)^* & = |D-\tau e_j|^\alpha,
  \end{align*}
  where $e_j$ denotes the $j$-th canonical basis vector in $\R^d$.
  Since $0<\alpha<1$,
  $$
  \left| |\xi-\tau e_j|^\alpha-|\xi|^\alpha \right| \leq |\tau|^\alpha.
  $$
  Hence,
  $$
  U_j(\tau)L_\kappa U_j(\tau)^*-L_\kappa \text{ is bounded for every fixed $\tau$.}
  $$
  In particular, $U_j(\tau)$ preserves $\mathcal D(L_\kappa)$ and the usual resolvent commutator identity may be used with $B_{j,\tau}$. To write it, define
  \begin{align}
    q_\tau^-(\xi) & := \frac{|\xi-\tau e_j|^\alpha-|\xi|^\alpha}{\tau}, \\
    q_\tau^+(\xi) & := \frac{|\xi|^\alpha-|\xi+\tau e_j|^\alpha}{\tau}.
  \end{align}
  Then,
  \begin{align}
    \label{eq:Btau-L-two-forms}
    [B_{j,\tau},L_\kappa] = q_\tau^-(D)U_j(\tau) = U_j(\tau)q_\tau^+(D).
  \end{align}
  We claim that
  \begin{align}
    \label{eq:q-tau-bound}
    |q_\tau^\pm(\xi)| \lesssim |\xi|^{\alpha-1} = |\xi|^{-\beta}, \quad \xi\neq0, \quad 0<|\tau|\leq1,
  \end{align}
  uniformly in $\tau$.
  Indeed, if $|\xi|\geq2|\tau|$, the mean value theorem gives
  $$
  \left| |\xi-\tau e_j|^\alpha-|\xi|^\alpha \right|
  \lesssim |\tau|\,|\xi|^{\alpha-1}.
  $$
  If $|\xi|<2|\tau|$, then, since $\alpha<1$,
  $$
  \left| |\xi-\tau e_j|^\alpha-|\xi|^\alpha \right|
  \lesssim |\tau|^\alpha
  \lesssim |\tau|\,|\xi|^{\alpha-1}.
  $$
  Thus, there are uniformly bounded Fourier multipliers $m_\tau^\pm(D)$ such that
  \begin{align}
    \label{eq:q-factor}
    q_\tau^\pm(D) = m_\tau^\pm(D)|D|^{-\beta},
    \qquad  \sup_{0<|\tau|\leq1}\|m_\tau^\pm(D)\|<\infty.
  \end{align}
  By \eqref{eq:negative-sobolev-spectral}, for every $a\in C_c^\infty(0,\infty)$,
  \begin{align}
    \label{eq:q-spectral-localized}
    \sup_{0<|\tau|\leq1} \left( \|q_\tau^\pm(D)a(L_\kappa)\| + \|a(L_\kappa)q_\tau^\pm(D)\| \right) <\infty.
  \end{align}
  Choose a real-valued $g\in C_c^\infty(0,\infty)$ such that $g=1$ on a neighborhood of $\supp h$. Thus, $h(L_\kappa)=g(L_\kappa)h(L_\kappa)$. Since for every $\tau\neq0$, $B_{j,\tau}$ is bounded,
  \begin{align}
    \label{eq:difference-product-decomp}
    [B_{j,\tau},h(L_\kappa)]
    = g(L_\kappa)[B_{j,\tau},h(L_\kappa)] + [B_{j,\tau},g(L_\kappa)]h(L_\kappa).
  \end{align}
  Write
  $$
  R(z):=(z-L_\kappa)^{-1}.
  $$
  By the Helffer--Sj\"ostrand formula \eqref{eq:hs},
  \begin{align}
    \label{eq:local-C1-first}
    g(L_\kappa)[B_{j,\tau},h(L_\kappa)]
    & = \frac{1}{2\pi i} \int_{\C} \overline\partial\widetilde h(z)\, g(L_\kappa)R(z) [B_{j,\tau},L_\kappa] R(z)\, dz\wedge d\overline z.
  \end{align}
  Using the first representation of $[B_{j,\tau},L_\kappa]$ in \eqref{eq:Btau-L-two-forms} shows that the integrand equals $\overline\partial\widetilde h(z)\, R(z)g(L_\kappa) q_\tau^-(D) U_j(\tau) R(z)$. By \eqref{eq:q-spectral-localized},
  $$
  \left\| g(L_\kappa)R(z) [B_{j,\tau},L_\kappa]R(z) \right\|
  \lesssim |\im z|^{-2},
  $$
  uniformly in $\tau$.
  Similarly, using the second representation in \eqref{eq:Btau-L-two-forms},
  \begin{align}
    \label{eq:local-C1-second}
    [B_{j,\tau},g(L_\kappa)]h(L_\kappa)
    & = \frac{1}{2\pi i} \int_{\C} \overline\partial\widetilde g(z)\, R(z)[B_{j,\tau},L_\kappa] R(z)h(L_\kappa)\, dz\wedge d\overline z.
  \end{align}
  Similarly as above, we get, uniformly in $\tau$,
  $$
  \left\| R(z)[B_{j,\tau},L_\kappa] R(z)h(L_\kappa) \right\| \lesssim |\im z|^{-2}.
  $$
  Choosing the almost analytic extensions of sufficiently high order, the two Helffer--Sj\"ostrand integrals are absolutely convergent in operator norm. Thus, \eqref{eq:difference-product-decomp} gives \eqref{eq:uniform-difference-comm}.

  \smallskip
  We now pass from bounded difference quotients to the commutator with $x_j$. Let $u,v\in\mathcal D(x_j)$. Since
  $$
  B_{j,\tau}v \to i x_jv
  \quad \text{and} \quad
  B_{j,\tau}^*u \to -i x_ju
  \quad \text{as }\tau\to0,
  $$
  the forms associated with $[B_{j,\tau},h(L_\kappa)]$ converge to those associated with $i[x_j,h(L_\kappa)]$. Since the uniform estimate \eqref{eq:uniform-difference-comm} gives
  $$
  |\langle x_ju,h(L_\kappa)v\rangle - \langle h(L_\kappa)^*u,x_jv\rangle|
  \lesssim \|u\| \|v\|,
  $$
  the limiting quadratic commutator form
  $$
  (u,v)\mapsto \langle x_ju,h(L_\kappa)v\rangle - \langle h(L_\kappa)^*u,x_jv\rangle
  $$
  extends continuously from $\mathcal D(x_j)\times\mathcal D(x_j)$ to $L^2\times L^2$.
  Denote the corresponding bounded operator by $C_h$. If $v\in\D(x_j)$, then for every $u\in\D(x_j)$,
  $$
  \langle x_ju,h(L_\kappa)v\rangle = \langle h(L_\kappa)^*u,x_jv\rangle + \langle u,C_hv\rangle
  $$
  The right-hand side is bounded by $C_v\|u\|$. Since $x_j$ is selfadjoint, this implies $h(L_\kappa)v\in\mathcal D(x_j)$. Moreover, $[x_j,h(L_\kappa)]=C_h$ as a bounded operator. This proves \eqref{eq:hL-C1-x}.

  \medskip
  \noindent
  {\bf Step 3: Conclusion of \eqref{eq:lem9-small-alpha} with $\mu=1$.}
  Choose $g\in C_c^\infty((0,\infty);\R)$ such that $g=1$ on a neighborhood of $\supp f$, and define the bounded selfadjoint operator
  $$
  K := k(L_\kappa),
  \qquad
  k(\lambda) := \lambda g(\lambda).
  $$
  By Step~2,
  $$
  K\in C^1(x_j),
  \qquad
  [x_j,K]\in\mathcal B(L^2).
  $$
  Hence,
  \begin{align}
    \label{eq:bounded-K-Duhamel}
    [x_j,\me{-itK}] = -i\int_0^t \me{-i(t-s)K} [x_j,K] \me{-isK}\,ds,
  \end{align}
  and therefore
  \begin{align}
    \label{eq:bounded-K-comm-bound}
    \|[x_j,\me{-itK}]\| \lesssim |t|.
  \end{align}
  Since $k(\lambda)=\lambda$ on $\supp f$, the spectral theorem gives $\me{-itK}f(L_\kappa) = \me{-itL_\kappa}f(L_\kappa)$. Thus,
  \begin{align}
    \label{eq:x-weight-small-alpha}
    \begin{split}
      \|x_jf(L_\kappa)\me{-itL_\kappa} \langle x\rangle^{-1}\|
      & = \|x_j \me{-itK}f(L_\kappa) \langle x\rangle^{-1}\| \\
      & \leq \|\me{-itK}x_jf(L_\kappa) \langle x\rangle^{-1}\| + \|[x_j,\me{-itK}] f(L_\kappa)\langle x\rangle^{-1}\|.
    \end{split}
  \end{align}
  By Step~2, $[x_j,f(L_\kappa)]\in\mathcal B(L^2)$, and hence
  \begin{align}
    x_jf(L_\kappa)\langle x\rangle^{-1}
    = f(L_\kappa)x_j\langle x\rangle^{-1} + [x_j,f(L_\kappa)] \langle x\rangle^{-1}
  \end{align}
  is bounded. Together with \eqref{eq:bounded-K-comm-bound}, this yields
  \begin{align}
    \label{eq:x-weight-small-alpha-final}
    \|x_jf(L_\kappa) \me{-itL_\kappa} \langle x\rangle^{-1}\|
    \lesssim_f \langle t\rangle.
  \end{align}
  Since $\langle x\rangle \lesssim 1+\sum_{j=1}^d|x_j|$, we conclude \eqref{eq:lem9-small-alpha} with $\mu=1$.

  \medskip
  \noindent
  {\bf Step 4: Interpolation.}
  For fixed $t$, set $T_t := f(L_\kappa) \me{-itL_\kappa}$. We have $\|T_t\|\leq\|f\|_\infty$. Thus, complex interpolation with \eqref{eq:lem9-small-alpha} with $\mu=1$ yields \eqref{eq:lem9-small-alpha} for all $0\leq\mu\leq1$.
\end{proof}


\section{Asymptotic completeness of $W(L_\kappa - V^S,L_\kappa)$ for $1/2<\alpha\leq1$ using propagation estimates}

\label{s:asympcomplperturbedwosmallalpha_a}

\begin{proof}[Proof of Theorem~\ref{asympcomplhardyplusperturbation}(\ref{asympcomplhardyplusperturbation:acshortrangesmallalpha})(b)]
  We first show the existence of $W(H,H_0)$ using the Cook--Kuroda method with the notation
  $$
  H_0 := L_\kappa,
  \qquad
  H := H_\kappa = L_\kappa-V^S.
  $$
  To this end, we aim to apply the minimal escape velocity estimate in Theorem~\ref{the6} and the outgoing phase space propagation estimate in Lemma~\ref{the11suff-unified}.

  \medskip
  \noindent
  {\bf Step 1: Preparation.}
   We check that the conditions of Theorem~\ref{the6} and Lemma~\ref{the11suff-unified} are satisfied. To this end, let $\chi\in C^\infty(\R)$ be such that
  $$
  0\leq\chi\leq1,
  \qquad
  \chi(s) = 0 \quad \text{for } s \leq \frac{\theta}{4},
  \qquad
  \chi(s) = 1 \quad \text{for } s \geq \frac{\theta}{2}.
  $$
  Let also
  $$
  \chi_t := \chi(A/t),
  \qquad
  F_t := \one(A/t\leq\theta).
  $$
  Then
  \begin{align}
    \label{eq:proof-acshortrangesmallalpha-b-cutoff}
    1-\chi_t^2=(1-\chi_t^2)F_t.
  \end{align}
  Recall $H_0=L_\kappa$ has purely absolutely continuous spectrum $[0,\infty)$ (Remark~\ref{rem:Lkappa-ac-spectrum}). Moreover, since $V^S$ is bounded and decays, $V^S$ is relatively compact with respect to $L_0$ and, since $H_0$ has the same form domain as $L_0$, also with respect to $H_0=L_\kappa$. Hence, Weyl's theorem for the essential spectrum $\sigma_{\rm ess}(H)$ of $H$ implies $\sigma_{\rm ess}(H)=[0,\infty)$, and the spectrum of $H$ in $(-\infty,0)$ is discrete. Moreover, the Mourre estimate \eqref{eq:the6-local-strict-mourre} used in the proof of Theorem~\ref{the6}, together with $H\in C^2(A)$, implies absence of singular continuous spectrum on $(0,\infty)$ and local finiteness of the positive eigenvalues away from zero. Hence,
  \begin{align}
    \label{eq:proof-acshortrangesmallalpha-b-no-sc}
    \mathcal H_{\rm sc}(H) = \{0\}.
  \end{align}
  Let $f_0\in C_c^\infty((0,\infty);\R)$ be such that $\supp f_0 \cap \sigma_{\rm p}(K)=\emptyset$ with $K\in\{H_0,H\}$. Choose $f_1\in C_c^\infty((0,\infty);\R)$ such that $f_1=1$ on a neighborhood of $\supp f_0$.  
  Thus, Theorem~\ref{the6} with $f\equiv f_0$ there is applicable. Let $\theta>0$ and $T>0$ be the resulting constants.
  By the assumptions $1/2<\alpha\leq1$, $\sigma<(d-1)/2$, and $\rho_S>1$, we may choose
  \begin{align}
    \label{eq:proof-acshortrangesmallalpha-b-mu}
    1 < \mu < \min\{\rho_S,2,2\alpha,d-2\sigma\}.
  \end{align}  
  With this choice, Lemma~\ref{the11suff-unified} is applicable.

  \medskip
  \noindent
  {\bf Step 2: Existence of $W(H,H_0)$.}
  Let $f_0\in C_c^\infty((0,\infty);\R)$ and $f_1\in C_c^\infty((0,\infty);\R)$ be such that $f_1=1$ on a neighborhood of $\supp f_0$. Consider
  $$
  U_{10}(t) := f_1(H) \me{itH} \me{-itH_0} f_0(H_0)
  $$
  and fix $\nu>1$. Since $H-H_0=-V^S$,
  $$
  \frac{d}{dt}U_{10}(t) = -i f_1(H) \me{itH} V^S \me{-itH_0}f_0(H_0).
  $$
  Then, Lemma~\ref{the11suff-unified}, with $K_1=H$ and $K_0=H_0$, yields
  $$
  \left\| f_1(H) \me{itH} V^S\chi_t^2 \me{-itH_0}f_0(H_0) \right\|
  \lesssim_{f_0,f_1,T} t^{-\mu}, \quad t\geq T.
  $$
  By \eqref{eq:proof-acshortrangesmallalpha-b-cutoff} and Theorem~\ref{the6},
  \begin{align*}
    \left\| f_1(H) \me{itH} V^S (1-\chi_t^2) \me{-itH_0}f_0(H_0)\phi \right\| 
    & \lesssim \left\| F_t f_0(H_0) \me{-itH_0}\phi \right\| \\
    & \lesssim_{f_0,\nu,T} t^{-\nu} \|\langle A\rangle^\nu \phi \|, \quad t\geq T, \ \phi \in \D(\langle A\rangle^\nu).
  \end{align*}
  Since $\mu,\nu>1$, the derivative $\frac{d}{dt}U_{10}(t)\phi$ is integrable on $[T,\infty)$ for every $\phi\in\mathcal D(\langle A\rangle^\nu)$. Hence, $U_{10}(t)\phi$ has a limit on this dense subspace. As $\sup_{t\geq T}\|U_{10}(t)\| \leq \|f_1\|_\infty\|f_0\|_\infty$, this strong limit extends to all of $L^2$.

  Since $V^S$ is relatively compact with respect to $H_0$,
  \begin{align}
    \label{eq:proof-acshortrangesmallalpha-b-compact-fc}
    h(H)-h(H_0) \ \text{is compact for all} \ h\in C_0(\R).
  \end{align}
  Since $(1-f_1)f_0=0$,
  \begin{align*}
    (1-f_1(H)) \me{itH} \me{-itH_0}f_0(H_0)
    = \me{itH} (1-f_1(H)) (f_0(H_0)-f_0(H)) \me{-itH_0}.
  \end{align*}
  By \eqref{eq:proof-acshortrangesmallalpha-b-compact-fc}, $f_0(H_0)-f_0(H)$ is compact.
  Moreover, $\me{-itH_0}\phi\rightharpoonup0$ as $t\to\infty$ since $H_0$ has purely absolutely continuous spectrum; here $\rightharpoonup$ denotes weak convergence. Thus,
  $
  \slim_{t\to\infty}(1-f_1(H)) \me{itH} \me{-itH_0}f_0(H_0)=0.
  $
  Therefore, $\slim_{t\to\infty} \me{itH} \me{-itH_0}f_0(H_0)$ exists.

  Since $H_0$ has purely absolutely continuous spectrum $[0,\infty)$, the union of the ranges of $f_0(H_0)$ is dense in $L^2$. Since $\me{i t H} \me{-i t H_0}$ are unitary (hence uniformly bounded in $t$), we obtain the existence of $W(H,H_0)$.

  \medskip
  \noindent
  {\bf Step 3: Existence of $W(H_0,H)$.}
  Let $f_0\in C_c^\infty((0,\infty);\R)$ and satisfy $\supp f_0\cap\sigma_{\rm p}(H)=\emptyset$. Choose $f_1\in C_c^\infty((0,\infty);\R)$ such that $f_1=1$ on a neighborhood of $\supp f_0$.
  Fix $\nu>1$ and let
  $$
  \widetilde U_{10}(t) := f_1(H_0) \me{itH_0} \me{-itH} f_0(H).
  $$
  By exactly the same arguments that showed the existence of $\slim_{t\to\infty}\me{itH}\me{-itH_0}f_0(H_0)$, we obtain the existence of $\slim_{t\to\infty} \me{itH_0}\me{-itH}f_0(H)$. It remains to replace $f_0(H)$ with $P_{\rm ac}(H)$. We proceed similarly as before. First note that the union of the ranges of the operators $f_0(H)$ is dense in $\mathcal H_{\rm ac}(H)$; indeed, the negative spectrum is discrete, hence carries no absolutely continuous part. Moreover, $E_H(\{0\})P_{\rm ac}(H)=0$, i.e., the absolutely continuous subspace is exhausted by compact positive-energy intervals away from zero. On any interval $[a,b]\subset(0,\infty)$, there are only finitely many eigenvalues of $H$. Since an absolutely continuous spectral measure gives zero mass to those points, one can choose smooth compactly supported functions $f_0$ whose supports avoid those eigenvalues and which approximate $1$ on $[a,b]$ except in arbitrarily small neighborhoods of them. Thus,
  $$
  \overline{\bigcup_{f_0\in C_c^\infty(0,\infty),\,\supp(f_0)\cap\sigma_{\rm p}(H)=\emptyset}\ran(f_0(H))} = \mathcal H_{\rm ac}(H).
  $$
  Hence, $W(H_0,H)=\slim_{t\to\infty}\me{itH_0}\me{-itH}P_{\rm ac}(H)$ exists.

  \medskip
  \noindent
  {\bf Step 4: Conclusion.}
  Passing to the strong limit in $\left( \me{itH} \me{-itH_0}\right) \left(\me{itH_0} \me{-itH}P_{\rm ac}(H)\right) = P_{\rm ac}(H)$ yields $W(H,H_0)W(H_0,H)=P_{\rm ac}(H)$. Since $\ran W(H,H_0) \subset\mathcal H_{\rm ac}(H)$, the preceding identity gives the reverse inclusion. Therefore, $\ran W(H,H_0) = \mathcal H_{\rm ac}(H)$. Thus, $W(L_\kappa-V^S,L_\kappa)$ is asymptotically complete.
\end{proof}


\section{Proof of Corollary~\ref{cor:massiverelativisticexistence}}
\label{s:proofcormassiverelativistic}

\begin{proof}[{Proof of Corollary~\ref{cor:massiverelativisticexistence}}]
  Let $\alpha>1$ and let $T:=C_0(-\Delta)$ with $C_0\in\ck_\alpha$. We first verify that the free operator $T$ satisfies the hypotheses on the Fourier multiplier in Theorem~\ref{giere}. We then verify Assumption~\ref{assumheatdifference} for $H_1:=T-V_\kappa$ and $H_2:=T-V_\kappa-V^S$.

  \medskip
  \noindent
  {\bf Step 1: Admissibility of the Fourier multiplier.}
  By the definition of $\ck_\alpha$, for every $t>0$ there is a probability density $\sigma_t^{(C_0)}$ on $\R_+$ such that
  $$
  \me{-tC_0(\lambda)} = \int_0^\infty \me{-s\lambda}\sigma_t^{(C_0)}(s)\,ds
  \quad\text{with}\quad
  \sigma_t^{(C_0)}(s) \leq\me{ct}\sigma_t^{(\alpha/2)}(s)
  $$
  for some $c\geq0$.
  In particular, $C_0(0)=0$. Moreover, $\lambda\mapsto\me{-tC_0(\lambda)}$ is completely monotone for every $t>0$, and therefore $C_0$ is a Bernstein function. In particular,
  \begin{align}
    \label{eq:bernstein-consequence}
    C_0\in C^\infty(0,\infty),
    \quad\text{and}\quad
    C_0'(\lambda)>0
    \quad\text{for all } \lambda>0
  \end{align}
  since $C_0$ is nonconstant.
  Recall also that, by operator monotonicity of the logarithm,
  \begin{align}
    \label{eq:C0-lower-fractional}
    C_0(-\Delta)\geq L_0-c.
  \end{align}
  Set
  $$
  P(\xi):=C_0(|\xi|^2).
  $$
  Then, for $\xi\neq0$, we have $\nabla P(\xi) = 2C_0'(|\xi|^2)\xi$, and hence $\nabla P(\xi)\neq0$.
  Now, for $h(|\xi|):=P(\xi)$, we verify that the set $\{k>0:\, h''(k)=(h'(k))^2\}$ is discrete. To this end, note that $h:(0,\infty)\to(0,\infty)$ is analytic since $C_0$ is a Bernstein function.
  Hence, also $g(k):=h''(k)-(h'(k))^2$ is analytic on $(0,\infty)$. Thus, by the identity theorem for real-analytic functions, the zero set $\{k>0:\, g(k)=0\}$ is either discrete or $g$ vanishes identically. Let us show that $g\not\equiv0$ via a contradiction, i.e., suppose $g\equiv0$. Then, define $f(k):=\me{-h(k)}$, which obeys the differential equation
  $$
  f''(k) = - \left(h''(k)-(h'(k))^2\right)\me{-h(k)}.
  $$
  By our assumption, $f''(k)=0$ for all $k>0$, i.e., $f(k)=Ak+B$ for some integration constants $A,B\in\R$. Thus, $h(k)=-\ln\left(Ak+B\right)$ and therefore $C_0(\lambda)=-\ln\left(A\sqrt\lambda +B\right)$.
  But this forces $C_0$ into an impossible form; indeed, for $A<0$ the argument of $C_0$ becomes eventually negative; for $A\geq0$, the function $C_0(\lambda)$ is nonincreasing, which is in contrast to \eqref{eq:bernstein-consequence}.
  We have thus shown $M_1=\emptyset$ and, since $0$ is not an accumulation point of $M_2$, that $M_2=\{k>0:\,g(k)=0\}$ is discrete.
  Thus, $T$ satisfies the assumptions of Theorem~\ref{giere}.

  \medskip
  \noindent
  {\bf Step 2: Heat kernel comparison.}
  Let
  $$
  S_T(t):=\me{-tT}.
  $$
  By subordination and the defining estimate for $\sigma_t^{(C_0)}$,
  \begin{align}
    \label{eq:massive-free-domination}
    S_T(t)(x,y) \leq \me{ct}\me{-tL_0}(x,y).
  \end{align}
  By assumption,
  \begin{align}
    \label{eq:Vkappa-absolute-bound}
    |V_\kappa(x)| \leq C_{\kappa,\tilde\kappa}|x|^{-\alpha}.
  \end{align}
  Thus, by Trotter's formula and \eqref{eq:massive-free-domination},
  \begin{align}
    \label{eq:massive-H1-domination}
    0\leq \me{-tH_1}(x,y) \lesssim \me{ct}\me{-tL_\kappa}(x,y).
  \end{align}
  Notice also that, by \eqref{eq:C0-lower-fractional},
  $$
  H_1 = T-V_\kappa \geq L_0-\frac{\kappa}{|x|^\alpha}-c = L_\kappa-c,
  $$
  so $H_1$ is bounded from below.

  We next consider $H_2=H_1-V^S$, which is also bounded from below since $V^S$ is bounded. Moreover, by \eqref{eq:massive-H1-domination} and Trotter's formula,
  \begin{align}
    \label{eq:massive-H2-domination}
    0 \leq \me{-tH_2}(x,y) \lesssim \me{c't}\me{-tL_\kappa}(x,y).
  \end{align}

  \medskip
  \noindent
  {\bf Step 3: On $W(H_1,T)$.}
  We verify the assumptions of Theorem~\ref{giere} for $(H,H_0)=(H_1,T)$.
  By Duhamel's formula,
  $$
  \me{-tH_1}-\me{-tT} = \int_0^t \me{-(t-s)H_1} V_\kappa \me{-sT}\,ds.
  $$
  By \eqref{eq:massive-free-domination}, \eqref{eq:massive-H1-domination}, and \eqref{eq:Vkappa-absolute-bound},
  \begin{align}
    \label{eq:massive-Hardy-Duhamel-majorant}
    \left| (\me{-tH_1}-\me{-tT})(x,y) \right|
    \lesssim_t \int_0^t ds \int_{\R^d}dz\, \me{-(t-s)L_\kappa}(x,z) |z|^{-\alpha} \me{-sL_0}(z,y).
  \end{align}
  The kernel on the right-hand side is exactly the kernel majorant estimated in Subsections~\ref{s:compactness}--\ref{s:weightedbound}.
  Thus, Theorem~\ref{giere} applies and shows that
  \begin{align}
    \label{eq:massive-WH1T}
    W(H_1,T) = W(C_0(-\Delta)-V_\kappa,C_0(-\Delta))
  \end{align}
  exist and are asymptotically complete.

  \medskip
  \noindent
  {\bf Step 4: On $W(H_2,T)$.}
  We now verify the same two assumptions for $H_2$ and $T$. Since
  $$
  H_2=T-V_\kappa-V^S,
  $$
  Duhamel's formula yields
  \begin{align}
    \label{eq:massive-H2-Duhamel}
    \me{-tH_2}-\me{-tT}
    = \int_0^t \me{-(t-s)H_2} (V_\kappa + V^S) \me{-sT}\,ds.
  \end{align}
  By \eqref{eq:massive-H2-domination}, \eqref{eq:massive-free-domination}, and \eqref{eq:Vkappa-absolute-bound},
  \begin{align*}
    \left|\me{-tH_2}-\me{-tT}(x,y)\right|
    \lesssim_t \int_0^t ds \int_{\R^d}dz\, \me{-(t-s)L_\kappa}(x,z)\left(|z|^{-\alpha} + |V^S(z)|\right)\me{-sL_0}(z,y).
  \end{align*}
  The kernel on the right-hand side is exactly the kernel majorant estimated in Section~\ref{s:wlkappavlkappacomplete}.
  Thus, Theorem~\ref{giere} yields the asymptotic completeness of
  \begin{align}
    \label{eq:massive-WH2T}
    W(H_2,T) = W(C_0(-\Delta)-V_\kappa-V^S,C_0(-\Delta)).
  \end{align}

  \medskip
  \noindent
  {\bf Step 5: On $W(H_2,H_1)$.}
  By \eqref{eq:massive-WH1T} and \eqref{eq:massive-WH2T}, both $W(H_1,T)$ and $W(H_2,T)$ are asymptotically complete. Since $T$ has purely absolutely continuous spectrum, completeness gives $W(T,H_1) = W(H_1,T)^*$ on $\ch_{\rm ac}(H_1)$. Hence, the chain rule yields also asymptotic completeness of $W(H_2,H_1)=W(C_0(-\Delta)-V_\kappa-V^S,C_0(-\Delta)-V_\kappa)$.
\end{proof}


\section{Nonexistence of wave operators}
\label{s:nonexistence}

In this section, we prove the nonexistence results in Theorems~\ref{asympcomplhardy}(2) and~\ref{asympcomplhardyplusperturbationnonexistence}, and Corollary~\ref{cor:massiverelativisticnonexistence}.

\subsection{Abstract nonexistence criteria}

The proofs of nonexistence of $W(L_\kappa-V^L,L_0)$ and, for $\alpha\leq1$, also of $W(L_\kappa,L_0)$, are inspired by \cite[Section~3]{IshidaWada2020}. In the following, we consider a splitting of an external potential $U$ into a bounded, eventually sign-definite, long-range tail $U^L$, and a singular part $U^R$ which is dynamically of lower order. We denote the Schwartz space by $\mathcal S(\R^d)$ and $\hat \phi(\xi)=(2\pi)^{-d/2}\int_{\R^d}\me{-ix\cdot\xi}\phi(x)\,dx$ denotes the Fourier transform of $\phi$.

\begin{lemma}
  \label{lem:generalized-ishida-wada}
  Let $H_0=p(D)$ be a selfadjoint Fourier multiplier on $L^2(\R^d)$, where $p\in C^0(\R^d;[0,\infty))$.
  Let $U=U^{L}+U^{R}$ be real-valued potentials such that $H:=H_0-U$ has a selfadjoint realization.
  Assume that there exist
  $$
  0<\gamma\leq1,\qquad
  \epsilon_{\rm s}\in\{-1,1\},\qquad
  c_0,C_0,R_0>0
  $$
  such that $U^{\rm L}\in L^\infty(\R^d)$ and
  \begin{align}
    \label{eq:IW-leading-lower}
    \epsilon_{\rm s} U^{L}(x) & \geq c_0\langle x\rangle^{-\gamma}, \quad |x|\geq R_0, \\
    \label{eq:IW-leading-upper}
    |U^{L}(x)| & \leq C_0\langle x\rangle^{-\gamma}, \quad |x|\geq R_0.
  \end{align}
  Suppose moreover that there exists $0\not\equiv\phi\in\mathcal S(\R^d)$ with $\hat\phi\in C_c^\infty(\R^d)$ such that, with $K:=\supp\hat\phi$, the function $p$ is $C^\infty$ on a neighborhood of $K$ and
  \begin{align}
    \label{eq:IW-nonzero-velocity}
    0< v_-:= \inf_{\xi\in K}|\nabla p(\xi)|
    \leq \sup_{\xi\in K}|\nabla p(\xi)| =: v_+
    <\infty.
  \end{align}
  Writing $\phi_t:=\me{-itH_0}\phi$, assume that
  \begin{align}
    \label{eq:IW-remainder-small}
    \|U^{R}\phi_t\|_2 = o(t^{-\gamma}), \quad t\to\infty.
  \end{align}
  Finally, assume that the Duhamel formula
  \begin{align}
    \label{eq:IW-Duhamel}
    \me{it_2H}\me{-it_2H_0}\phi - \me{it_1H}\me{-it_1H_0}\phi
    = -i \int_{t_1}^{t_2} \me{itH}U\me{-itH_0}\phi\,dt
  \end{align}
  holds whenever $t_2>t_1$ are sufficiently large.
  Then $W(H,H_0):=\slim_{t\to\infty}\me{itH}\me{-itH_0}$ does not exist.

  \smallskip
  In particular, if $p(\xi)=\Psi(|\xi|^2)$ is radial, $\Psi\in C^\infty(0,\infty)$, and $\Psi'(\lambda)>0$ for all $\lambda>0$, and if $U^{R}$ is supported in a fixed ball and, for some $0\leq a<d$,
  \begin{align}
    \label{eq:IW-local-singularity}
    |U^{R}(x)| \lesssim 1+|x|^{-a},
  \end{align}
  then one can choose a nonzero $\phi$ with $\hat\phi\in C_c^\infty(\R^d\setminus\{0\})$ such that \eqref{eq:IW-nonzero-velocity} and \eqref{eq:IW-remainder-small} hold. In fact, for every $N\in\N$, one has $\|U^{R}\phi_t\|_2 \lesssim_N t^{-N}$.
\end{lemma}


\begin{proof}
  We first prove the abstract statement. To this end, we argue by contradiction and assume that $W:=W(H,H_0)$ exists. Let $W_t:=\me{itH}\me{-itH_0}$.

  \medskip
  \noindent
  {\bf Step 1: Free propagation estimate.}
  Fix $0<a<v_-$. Since
  $$
  \phi_t(x) = (2\pi)^{-d/2} \int_K \me{i(x\cdot\xi-tp(\xi))} \hat\phi(\xi)\,d\xi,
  $$
  we have, for $|x|\leq at$ and $\xi\in K$,
  $$
  \left| \nabla_\xi (x\cdot\xi-tp(\xi)) \right| \geq t(v_--a).
  $$
  Repeated integration by parts in $\xi$ therefore gives, for every $M\in\N$,
  $$
  \sup_{|x|\leq at} |\phi_t(x)| \lesssim_{M,\phi} t^{-M}.
  $$
  Thus, for every $N\in\N$,
  \begin{align}
    \label{eq:IW-inner-propagation}
    \left\| \one_{|x|\leq at}\phi_t \right\|_2 \lesssim_{N,\phi} t^{-N}.
  \end{align}
  Next, we show a crude propagation estimate. Since $x_j\me{-itH_0} = \me{-itH_0} (x_j+t\,\partial_jp(D))$,
  \begin{align}
    \label{eq:IW-first-moment}
    \|x\phi_t\|_2 \leq \|x\phi\|_2 + t\|\nabla p(D)\phi\|_2 \lesssim_\phi 1 + v_+ t.
  \end{align}
  We can choose $b=b_\phi>a$ sufficiently large and then $T=T_\phi>0$ sufficiently large, both depending on $\phi$, so that
  \begin{align}
    \label{eq:IW-outer-mass}
    \left\| \one_{|x|\geq bt}\phi_t \right\|_2^2
    \leq \frac{1}{b^2t^2} \|x\phi_t\|_2^2
    \leq \frac14\|\phi\|_2^2, \quad t\geq T
  \end{align}
  and 
  $$
  \left\| \one_{|x|\leq at}\phi_t \right\|_2^2 \leq \frac14\|\phi\|_2^2, \quad t\geq T.
  $$
  Thus, since $\|\phi_t\|_2=\|\phi\|_2$,
  \begin{align}
    \label{eq:IW-annulus-mass}
    \left\| \one_{at<|x|<bt}\phi_t \right\|_2^2 \geq
    \frac12\|\phi\|_2^2, \quad t \geq T.
  \end{align}

  \medskip
  \noindent
  {\bf Step 2: Lower bound for $\epsilon_{\rm s}\langle\phi_t,U\phi_t\rangle$.}
  We take $t$ so large that $at\geq R_0$. By \eqref{eq:IW-leading-lower}, we may discard all contributions from $|x|\geq R_0$ outside the annulus in \eqref{eq:IW-annulus-mass} for a lower bound for $\epsilon_{\rm s}\langle\phi_t,U\phi_t\rangle$. Hence,
  \begin{align}
    \label{eq:IW-leading-expectation-pre}
    \epsilon_{\rm s} \langle \phi_t,U^{L}\phi_t\rangle
    \geq c_0 \int_{at<|x|<bt} \langle x\rangle^{-\gamma} |\phi_t(x)|^2\,dx
    - \|U^{L}\|_\infty \left\| \one_{|x|<R_0}\phi_t \right\|_2^2.
  \end{align}
  For the first term, we use $\langle x\rangle^{-\gamma} \gtrsim_b t^{-\gamma}$ and \eqref{eq:IW-annulus-mass}, and the last term in \eqref{eq:IW-leading-expectation-pre} is $\mathcal O_N(t^{-N})$ by \eqref{eq:IW-inner-propagation}. Therefore,
  \begin{align}
    \label{eq:IW-leading-expectation}
    \epsilon_{\rm s} \langle \phi_t, U^{L}\phi_t\rangle
    \geq c_\phi t^{-\gamma} - C_{N,\phi}t^{-N}
  \end{align}
  for every $N$ and all sufficiently large $t$, where $c_\phi>0$.

  By \eqref{eq:IW-remainder-small} and conservation of the $L^2$ norm,
  $$
  \left| \langle \phi_t, U^{R}\phi_t\rangle \right|
  \leq \|U^{R}\phi_t\|_2\|\phi\|_2
  = o(t^{-\gamma}).
  $$
  Hence, choosing $N>\gamma$,
  \begin{align}
    \label{eq:IW-total-expectation}
    \epsilon_{\rm s} \langle \phi_t, U\phi_t\rangle \geq \frac{c_\phi}{2}t^{-\gamma}
  \end{align}
  for all sufficiently large $t$.

  \medskip
  \noindent
  {\bf Step 3: Upper bound for $\|U\phi_t\|_2$.}
  Using \eqref{eq:IW-leading-upper} and \eqref{eq:IW-inner-propagation}, we have, for $at\geq R_0$,
  \begin{align*}
    \|U^{L}\phi_t\|_2
    & \leq \left\| U^{L} \one_{|x|<at}\phi_t \right\|_2 + \left\| U^{L} \one_{|x|\geq at}\phi_t \right\|_2 \\
    & \lesssim \|U^{L}\|_\infty \left\| \one_{|x|<at}\phi_t \right\|_2 + t^{-\gamma}\|\phi\|_2 \lesssim_\phi t^{-\gamma}.
  \end{align*}
  Together with the assumption \eqref{eq:IW-remainder-small}, this yields
  \begin{align}
    \label{eq:IW-total-norm}
    \|U\phi_t\|_2 \leq C_\phi t^{-\gamma}
  \end{align}
  for all sufficiently large $t$.

  \medskip
  \noindent
  {\bf Step 4: Contradiction.}
  Since, by assumption, $W_t\phi\to W\phi$ in $L^2$, for every $\delta>0$ there exists $T_\delta$ such that
  \begin{align}
    \label{eq:IW-wave-convergence}
    \|(W-W_t)\phi\|_2 \leq\delta, \quad t\geq T_\delta.
  \end{align}
  Let $t_2>t_1\geq T_\delta$. Then, by \eqref{eq:IW-Duhamel},
  \begin{align}
    \label{eq:IW-Duhamel-consequence}
    W_{t_2}\phi-W_{t_1}\phi = -i \int_{t_1}^{t_2} \me{itH}U\phi_t\,dt.
  \end{align}
  Since $\me{-itH}W_t\phi=\phi_t$, we have $\me{-itH}W\phi = \phi_t + \me{-itH}(W-W_t)\phi$. Therefore,
  \begin{align}
    \label{eq:IW-integrand-lower}
    \epsilon_{\rm s} \re \left\langle \me{-itH}W\phi, U\phi_t \right\rangle
    \geq \epsilon_{\rm s} \langle \phi_t, U\phi_t\rangle - \delta\|U\phi_t\|_2, \quad t\geq T_\delta.
  \end{align}
  Choose $\delta>0$ so small that $\delta C_\phi<\frac{c_\phi}{4}$, where $C_\phi$ is the constant in \eqref{eq:IW-total-norm}. Then \eqref{eq:IW-total-expectation}, \eqref{eq:IW-total-norm}, and \eqref{eq:IW-integrand-lower} imply
  $$
  \epsilon_{\rm s} \re \left\langle \me{-itH}W\phi,U\phi_t \right\rangle
  \geq \frac{c_\phi}{4}t^{-\gamma}, \quad t \text{ sufficiently large.}
  $$
  Taking the scalar product of the Duhamel formula \eqref{eq:IW-Duhamel-consequence} with $W\phi$, multiplying by $i\epsilon_{\rm s}$ and taking its real part, we obtain
  \begin{align}
    \label{eq:IW-final-lower}
    \left| \left\langle W\phi, (W_{t_2}-W_{t_1})\phi \right\rangle \right|
    \geq \frac{c_\phi}{4} \int_{t_1}^{t_2}t^{-\gamma}\,dt.
  \end{align}
  Since $0<\gamma\leq1$, the right-hand side diverges as $t_2\to\infty$.
  On the other hand, $\left|\left\langle W\phi,(W_{t_2}-W_{t_1})\phi \right\rangle \right|
  \leq 2\|\phi\|_2^2$, a contradiction. Thus, $W(H,H_0)$ does not exist.

  \medskip
  \noindent
  {\bf Step 5: Application.}
  Assume now
  $$
  p(\xi)=\Psi(|\xi|^2),
  \qquad
  \Psi\in C^\infty(0,\infty),
  \qquad
  \Psi'(\lambda)>0,
  \quad \lambda>0,
  $$
  and \eqref{eq:IW-local-singularity}, i.e., there is $R_1>0$ such that
  $$
  \supp U^{R}\subset B(0,R_1),
  \qquad
  |U^{R}(x)| \lesssim 1+|x|^{-a}, \qquad
  0\leq a<d.
  $$
  We now show the existence of a nonzero $\phi$ with $\hat\phi\in C_c^\infty(\R^d\setminus\{0\})$ such that \eqref{eq:IW-nonzero-velocity} and $\|U^{R}\phi_t\|_2 = O_N(t^{-N})$ hold. To this end, choose $0<r_0<r_1<\infty$. Let $\ell\in\N_0$ satisfy
  \begin{align}
    \label{eq:IW-angular-condition}
    \ell>a-\frac d2,
  \end{align}
  choose $0\not\equiv g\in C_c^\infty((r_0,r_1))$, and let $Y_\ell$ be a spherical harmonic of degree $\ell$; in $d=1$, there are two spherical harmonics, $Y_0(x)=1$ and $Y_1(x)=x/|x|$ for $x\in\R$, and we take $\ell=1$ in this case. Take
  \begin{align}
    \label{eq:IW-angular-test-state}
    \hat\phi(\xi) := g(|\xi|) Y_\ell\left(\frac{\xi}{|\xi|}\right).
  \end{align}
  Then, indeed, $\supp\hat\phi \subset\subset\R^d\setminus\{0\}$ and $0\neq\phi\in\mathcal S(\R^d)$.
  Since $\Psi'>0$, \eqref{eq:IW-nonzero-velocity} holds. Since $H_0$ is spherically symmetric, $\phi_t$ remains in the same angular momentum channel $\ell$ for all $t$. Moreover, on the fixed ball $B(0,R_1)$, non-stationary phase gives, for every multi-index $\beta$ and every $N\in\N$,
  \begin{align}
    \label{eq:IW-local-derivative}
    \sup_{|x|\leq R_1} |\partial_x^\beta\phi_t(x)| \lesssim_{\beta,N,\phi} t^{-N}, \quad t\geq 2R_1/v_-.
  \end{align}
  For $|\beta|<\ell$, we have $\partial_x^\beta \phi_t(0)=0$ as a consequence of
  \begin{align}
    \label{eq:IW-spherical-orthogonality}
    \int_{\bs^{d-1}}\omega^\beta Y_\ell(\omega)\,d\omega=0.
  \end{align}
  To see this, recall that every homogeneous polynomial $p$ on $\R^d$ of degree $m$ can be uniquely written in the form
  $$
  p = p_m + |x|^2p_{m-2} + ... + |x|^{2k}p_{m-2k},
  $$
  where $k=\lfloor m/2\rfloor$ and each $p_j$ is a harmonic, homogeneous polynomial of degree $j$; see \cite[Theorem~5.7]{Axleretal2001}. Restricting this equality to $\bs^{d-1}$ allows us to write
  $$
  \omega^\beta = Y_{|\beta|}(\omega) + Y_{|\beta|-2}(\omega)+ ...
  $$
  This expansion, $|\beta|<\ell$, and the orthonormality of spherical harmonics yield \eqref{eq:IW-spherical-orthogonality}.
  
  Taylor's theorem and \eqref{eq:IW-local-derivative} thus give
  \begin{align}
    \label{eq:IW-origin-vanishing}
    |\phi_t(x)| \lesssim_{N,\phi} t^{-N}|x|^\ell, \quad |x|\leq R_1, \ t\geq2R_1/v_-.
  \end{align}
  Hence,
  \begin{align}
    \label{eq:IW-local-remainder}
    \|U^{R}\phi_t\|_2^2
    \lesssim_{N,\phi} t^{-2N} \int_{|x|\leq R_1} (1+|x|^{-2a})|x|^{2\ell}\,dx
    \lesssim_{N,\phi} t^{-2N}, \quad t\geq2R_1/v_-,
  \end{align}
  where the integral is finite since $2\ell-2a+d>0$. This concludes the proof.
\end{proof}

\subsection{Nonexistence of $W(L_\kappa,L_0)$ and $W(C_\kappa,C_0)$ for $\alpha\leq1$}

\begin{proof}[{Proofs of Theorem~\ref{asympcomplhardy}(2) and
  Corollary~\ref{cor:massiverelativisticnonexistence}(1)}]
  We treat both statements simultaneously.
  Let $0<\alpha\leq1$, $\kappa\neq0$, and
  $$
  T_0=
  \begin{cases}
    L_0=(-\Delta)^{\alpha/2},\\
    C_0=C_0(-\Delta),
  \end{cases}
  \qquad
  T_\kappa := T_0-\frac{\kappa}{|x|^\alpha}.
  $$
  We aim to apply Lemma~\ref{lem:generalized-ishida-wada} with
  $$
  H_0 = T_0,
  \qquad
  U(x) = \frac{\kappa}{|x|^\alpha},
  \qquad
  H = T_\kappa.
  $$
  For $T_0=L_0$, the Fourier symbol is
  $$
  p(\xi)=|\xi|^\alpha = \Psi(|\xi|^2),
  \qquad
  \Psi(\lambda)=\lambda^{\alpha/2}
  $$
  with $\Psi'(\lambda)>0$ for all $\lambda>0$.
  For $T_0=C_0$, the defining properties of $\ck_\alpha$ imply that $C_0$ is a nonconstant Bernstein function. Hence, $C_0\in C^\infty(0,\infty)$ and $C_0'(\lambda)>0$ for all $\lambda>0$.
  Thus, in both cases the symbol of $T_0$ has the form $p(\xi)=\Psi(|\xi|^2)$ with $\Psi'(\lambda)>0$ for $\lambda>0$, so the final assertion of Lemma~\ref{lem:generalized-ishida-wada} is applicable.

  Let $R\geq1$ and choose a radial function $\chi\in C^\infty(\R^d;[0,1])$ such that
  $$
  \chi(x)=0\quad\text{for }|x|\leq R,
  \qquad
  \chi(x)=1\quad\text{for }|x|\geq2R,
  $$
  Set
  $$
  U^{L}(x) := \chi(x)\frac{\kappa}{|x|^\alpha},
  \qquad
  U^{R}(x) := (1-\chi(x))\frac{\kappa}{|x|^\alpha}.
  $$
  Then $U^{L}\in L^\infty(\R^d)$ and, with $\epsilon_{\rm s}:=\sgn(\kappa)$, there are $c,C>0$ such that
  $$
  c\langle x\rangle^{-\alpha} \leq \epsilon_{\rm s}U^{L}(x) \leq C\langle x\rangle^{-\alpha}, \quad |x|\geq2R.
  $$
  Thus, the long-range exponent in Lemma~\ref{lem:generalized-ishida-wada} is $\gamma=\alpha\leq1$. Moreover, $U^{R}$ is supported in $B(0,2R)$ and $|U^{R}(x)| \lesssim 1+|x|^{-\alpha}$. Since $\alpha<d$, Lemma~\ref{lem:generalized-ishida-wada} yields the existence of a nonzero $\phi\in\mathcal S(\R^d)$ with $\hat\phi\in C_c^\infty(\R^d\setminus\{0\})$, such that, with $\phi_t:=\me{-itT_0}\phi$, we have $\|U^{R}\phi_t\|_2 \lesssim_N t^{-N}$ for every $N\in\N$.

  Let us also verify the Duhamel formula \eqref{eq:IW-Duhamel}. The function $\hat\phi_t$ has the same compact support as $\hat\phi$, and hence $\phi_t\in\mathcal S(\R^d)$ and $T_0\phi_t\in L^2(\R^d)$. Furthermore, $U\phi_t = U^{L}\phi_t + U^{R}\phi_t \in L^2(\R^d)$.
  Since $T_\kappa$ is defined by the corresponding semibounded quadratic form,
  we have $\phi_t\in\D(T_\kappa)$ and $T_\kappa\phi_t = T_0\phi_t-U\phi_t$. As $t\mapsto U\phi_t$ is locally continuous in $L^2$,
  $$
  \frac{d}{dt} \left(\me{itT_\kappa}\me{-itT_0}\phi\right)
  = -i\me{itT_\kappa}U\me{-itT_0}\phi
  $$
  for sufficiently large $t$, and integration yields the Duhamel formula \eqref{eq:IW-Duhamel}.

  All assumptions of that lemma are therefore satisfied. Hence, $W(T_\kappa,T_0)$ does not exist. Thus, taking $T_0=L_0$ proves Theorem~\ref{asympcomplhardy}(2), while taking $T_0=C_0$ proves Corollary~\ref{cor:massiverelativisticnonexistence}(1).
\end{proof}

\subsection{Nonexistence of $W(L_\kappa-V^L,L_0)$ and $W(C_\kappa-V^L,C_0)$}

\begin{proof}[{Proof of Theorem~\ref{asympcomplhardyplusperturbationnonexistence}(\ref{asympcomplhardyplusperturbation:nonexistencefree}) and
    Corollary~\ref{cor:massiverelativisticnonexistence}(2)}]
  Again, we treat the homogeneous and massive kinetic energies simultaneously.
  Let
  $$
  T_0 =
  \begin{cases}
    L_0 = (-\Delta)^{\alpha/2},\\
    C_0 = C_0(-\Delta).
  \end{cases}
  $$
  Set
  $$
  U(x) := \frac{\kappa}{|x|^\alpha}+V^L(x),
  \qquad
  H := T_0-U.
  $$
  Thus,
  $$
  H =
  \begin{cases}
    L_\kappa-V^L,\\
    C_\kappa-V^L.
  \end{cases}
  $$
  As in the preceding proof, the Fourier symbol of $T_0$ is of the form
  $$
  p(\xi) = \Psi(|\xi|^2),
  \qquad
  \Psi'(\lambda)>0,
  \quad \lambda>0.
  $$
  Hence, Lemma~\ref{lem:generalized-ishida-wada} is applicable.
  We first determine the leading long-range behavior of $U$.

  \smallskip
  \noindent
  {\bf Case 1: $\rho_L<\alpha$.}
  In this case, $|x|^{-\alpha} = o(\langle x\rangle^{-\rho_L})$ as $|x|\to\infty$.
  Let $\epsilon_L\in\{-1,1\}$ denote the sign of $V^L$, so that by \eqref{ass1} or \eqref{ass2},
  $$
  K_1\langle x\rangle^{-\rho_L} \leq \epsilon_LV^L(x) \leq K_2\langle x\rangle^{-\rho_L}.
  $$
  Hence, after increasing $R>0$ if necessary,
  \begin{align}
    \label{eq:nonexistence-total-leading-rho}
    c\langle x\rangle^{-\rho_L} \leq \epsilon_L U(x) \leq C\langle x\rangle^{-\rho_L}, \quad |x|\geq R.
  \end{align}
  Thus, the leading exponent is $\gamma=\rho_L$.

  \smallskip
  \noindent
  {\bf Case 2: $\rho_L>\alpha$.}
  If $\kappa\neq0$, then $\langle x\rangle^{-\rho_L} = o(|x|^{-\alpha})$, and hence, after increasing $R$,
  \begin{align}
    \label{eq:nonexistence-total-leading-alpha}
    c\langle x\rangle^{-\alpha} \leq \sgn(\kappa)U(x) \leq C\langle x\rangle^{-\alpha}, \quad |x|\geq R.
  \end{align}
  Thus, $\gamma=\alpha$. Notice that $\alpha<\rho_L\leq1$, and therefore $\gamma=\alpha<1$.

  If $\kappa=0$, then $U=V^L$, and the long-range assumptions directly give
  $$
  c\langle x\rangle^{-\rho_L} \leq \epsilon_LU(x) \leq C\langle x\rangle^{-\rho_L}
  $$
  for all $x$, so that $\gamma=\rho_L\leq1$.

  \smallskip
  \noindent
  {\bf Case 3: $\rho_L=\alpha$.}
  Here the two contributions have the same order and may cancel. However, by assumption there exist $\epsilon_{\rm s}\in\{-1,1\}$, $c>0$, and $R>0$ such that
  \begin{align*}
    \epsilon_{\rm s} \left(\frac{\kappa}{|x|^\alpha}+V^L(x)\right)
    \geq c\langle x\rangle^{-\alpha}, \quad |x|\geq R.
  \end{align*}
  The assumptions on $V^L$ also give the corresponding upper bound
  $$
  |U(x)| \lesssim \langle x\rangle^{-\alpha}, \quad |x|\geq R.
  $$
  Thus, $\gamma=\alpha=\rho_L\leq1$.

  \smallskip
  In all three cases, we have found $0<\gamma\leq1$, $\epsilon_{\rm s}\in\{-1,1\}$, and $c,C,R>0$ such that
  \begin{align}
    \label{eq:nonexistence-total-tail}
    c\langle x\rangle^{-\gamma} \leq \epsilon_{\rm s}U(x),
    \qquad
    |U(x)| \leq C\langle x\rangle^{-\gamma}, \quad |x|\geq R.
  \end{align}
  Choose a radial cutoff $\chi\in C^\infty(\R^d;[0,1])$ with
  $$
  \chi=0
  \quad\text{on }B(0,R),
  \qquad
  \chi=1
  \quad\text{on }\R^d\setminus B(0,2R),
  $$
  and define
  \begin{align}
    \label{eq:nonexistence-UL-UR-cutoff}
    U^{L} := \chi U,
    \qquad
    U^{R} := (1-\chi)U.
  \end{align}
  Then, $U^{L}\in L^\infty(\R^d)$ and \eqref{eq:nonexistence-total-tail} implies the assumptions \eqref{eq:IW-leading-lower}--\eqref{eq:IW-leading-upper} in Lemma~\ref{lem:generalized-ishida-wada}.
  On the other hand, $\supp U^{R} \subset B(0,2R)$. Since $V^L\in L^\infty(\R^d)$,
  $$
  |U^{R}(x)| \lesssim 1+|x|^{-\alpha}.
  $$
  Since $\alpha<d$, Lemma~\ref{lem:generalized-ishida-wada} provides a nonzero $\phi\in\mathcal S(\R^d)$ with $\hat\phi \in C_c^\infty(\R^d\setminus\{0\})$ such that, with $\phi_t:=\me{-itT_0}\phi$, we have $\|U^{R}\phi_t\|_2 \lesssim_N t^{-N}$ for every $N\in\N$.

  The Duhamel formula \eqref{eq:IW-Duhamel} is verified exactly as in the previous subsection.
  Thus, by Lemma~\ref{lem:generalized-ishida-wada}, $W(H,T_0)$ does not exist.
  For $T_0=L_0$, this proves $W(L_\kappa-V^L,L_0)$ does not exist.
  For $T_0=C_0$, it proves $W(C_\kappa-V^L,C_0)$ does not exist.
\end{proof}

\subsection{Nonexistence of $W(L_\kappa-V^L,L_\kappa)$ and $W(C_\kappa-V^L,C_\kappa)$ for $\alpha>1$}

\begin{proof}[{Proof of the nonexistence assertions in
    Theorem~\ref{asympcomplhardyplusperturbationnonexistence}(\ref{asympcomplhardyplusperturbation:nonexistencerelativelargealpha}) and
    Corollary~\ref{cor:massiverelativisticnonexistence}(3)}]
  We assume $\alpha>1$. We again treat both kinetic energies simultaneously. Set
  $$
  H_0 =
  \begin{cases}
    L_0, \\
    C_0,
  \end{cases}
  \qquad
  H_1 =
  \begin{cases}
    L_\kappa, \\
    C_\kappa,
  \end{cases}
  \qquad
  H_2:=H_1-V^L.
  $$
  By Theorem~\ref{asympcomplhardy}(1), respectively Corollary~\ref{cor:massiverelativisticexistence}(1), the wave operator
  $$
  W_{10}:=W(H_1,H_0)
  $$
  exists and is asymptotically complete. Thus, $\ran W_{10} = \ch_{\rm ac}(H_1)$.
  Since $0<\rho_L\leq1<\alpha$, we have, by Theorem~\ref{asympcomplhardyplusperturbationnonexistence}(\ref{asympcomplhardyplusperturbation:nonexistencefree}) and Corollary~\ref{cor:massiverelativisticnonexistence}(2) (proved in the previous subsection), that $W(H_2,H_0)$ does not exist.

  Suppose, contrary to the claim, that the relative wave operator
  \begin{align}
    \label{eq:relative-wave-assume}
    W_{21} := \slim_{t\to\infty} \me{itH_2}\me{-itH_1}P_{\rm ac}(H_1)
  \end{align}
  exists. Then, by the chain rule for wave operators, $W_{20}$ exists; this is a contradiction.
  Thus, taking $(H_0,H_1,H_2) = (L_0,L_\kappa,L_\kappa-V^L)$ proves $W(L_\kappa-V^L,L_\kappa)$ does not exist.
  Taking $(H_0,H_1,H_2) = (C_0,C_\kappa,C_\kappa-V^L)$ proves $W(C_\kappa-V^L,C_\kappa)$ does not exist.
\end{proof}

\subsection{Nonexistence of $W(L_\kappa-V^L,L_\kappa)$ for $0<\rho_L\leq\alpha\leq1$ and $\alpha\geq1/2$}

The proof crucially relies on the propagation estimates in Lemmas~\ref{lem9} and~\ref{lem9-small-alpha}, which we use to infer kinds of maximal velocity estimates in Step~3 below, and the estimates in Subsection~\ref{ss:preliminaries-generator-of-dilations}.

\begin{proof}[{Proof of Theorem~\ref{asympcomplhardyplusperturbationnonexistence}
  (\ref{asympcomplhardyplusperturbation:nonexistencelongrangesmallalpha})}]
  We write
  $$
  L:=L_\kappa,
  \qquad
  H:=L_\kappa-V^L,
  \qquad
  \rho:=\rho_L,
  $$
  and argue by contradiction. Thus, suppose that $W:=W(H,L) = \slim_{t\to\infty}\me{i t H}\me{-i t L}$
  exists.

  \medskip
  \noindent
  {\bf Step 1: Constructing a test function.}
  Since $\kappa<\kappa_{\rm c}$, the spectrum of $L_\kappa$ is $[0,\infty)$ and is purely absolutely continuous. Choose $0\not\equiv h\in C_c^\infty(0,\infty)$ such that $h(L)\neq0$. Since $\D(\langle x\rangle)\cap\D(\langle A\rangle^2)$ is dense in $L^2(\R^d)$, we can choose
  $$
  \psi\in \D(\langle x\rangle)\cap\D(\langle A\rangle^2)
  $$
  such that
  $$
  \phi:=h(L)\psi\neq0.
  $$
  By Lemma~\ref{lem9-small-alpha}, respectively by Lemma~\ref{lem9} when $\alpha=1$, evaluated at $t=0$ and $\mu=1$, we have $h(L)\D(\langle x\rangle) \subset \D(\langle x\rangle)$. Moreover, by $\me{i s A}L\me{-i s A} = \me{-\alpha s}L$,
  $$
  \me{i s A}h(L)\me{-i s A} = h(\me{-\alpha s}L).
  $$
  Since $h\in C_c^\infty(0,\infty)$, the right-hand side is $C^\infty$ in $s$ in operator norm. Hence, $h(L)\in C^\infty(A)$ and, in particular, $h(L)\D(\langle A\rangle^2) \subset \D(\langle A\rangle^2)$. Thus, $\phi\in \D(\langle x\rangle)\cap\D(\langle A\rangle^2)$. Choose $f\in C_c^\infty(0,\infty)$ such that $f=1$ on a neighborhood of $\supp h$. Then
  \begin{align}
    \label{eq:nonexistence-unified-spectral-localization}
    \phi=f(L)\phi.
  \end{align}
  In particular, $\phi\in\D(L^m)$ for every $m\geq0$. We now set $u_t:=\me{-i t L}\phi$.

  \medskip
  \noindent
  {\bf Step 2: Duhamel decomposition.}
  Since $V^L$ is bounded,
  $$
  \frac{d}{dt} \left( \me{i t H}\me{-i t L} \right) = -i\me{i t H}V^L\me{-i t L}
  $$
  in the strong sense. Therefore, for $0<t_1<t_2$,
  \begin{align}
    \label{eq:nonexistence-unified-Duhamel}
    \left\langle W\phi, \left(\me{i t_2 H}\me{-i t_2 L} - \me{i t_1 H}\me{-i t_1 L} \right)\phi \right\rangle 
    = -i\int_{t_1}^{t_2} \left\langle W\phi, \me{i t H}V^Lu_t \right\rangle\,dt.
  \end{align}
  For some $\eta>0$ to be specified below (in Step~3), we choose $\lambda>0$ so large that
  \begin{align}
    \label{eq:nonexistence-unified-positive-resolvent}
    \lambda+H\geq\eta>0.
  \end{align}
  By the intertwining property, $(\lambda+H)W\phi = W(\lambda+L)\phi$. Writing $R_\lambda:=(\lambda+H)^{-1}$, we get
  \begin{align}
    \label{eq:nonexistence-unified-resolvent-split}
    \begin{split}
      \left\langle W\phi,\me{i t H}V^Lu_t \right\rangle
      & = \left\langle \me{-i t H}W(\lambda+L)\phi,R_\lambda V^L u_t \right\rangle \\
      & = \left\langle \me{-i t L}(\lambda+L)\phi, R_\lambda V^Lu_t \right\rangle \\
      & \quad + \left\langle \me{-i t H} \left(W-\me{i t H}\me{-i t L} \right)(\lambda+L)\phi, R_\lambda V^Lu_t \right\rangle.
    \end{split}
  \end{align}
  Define
  \begin{align}
    \label{eq:nonexistence-unified-J1}
    J_1 & := \int_{t_1}^{t_2} \left\langle \me{-i t L}(\lambda+L)\phi, R_\lambda V^Lu_t \right\rangle\,dt, \\
    \label{eq:nonexistence-unified-J2}
    J_2 & := \int_{t_1}^{t_2}\left\langle \me{-i t H} \left( W-\me{i t H}\me{-i t L} \right)(\lambda+L)\phi, R_\lambda V^Lu_t \right\rangle\,dt.
  \end{align}
  Then,
  \begin{align}
    \label{eq:nonexistence-unified-J1J2}
    \left\langle W\phi, \left(\me{i t_2 H}\me{-i t_2 L} - \me{i t_1 H}\me{-i t_1 L} \right)\phi \right\rangle
    = -i (J_1 + J_2).
  \end{align}

  \medskip
  \noindent
  {\bf Step 3: Lower bound for the main term $J_1$.}
  Since $\me{-i t L}(\lambda+L)\phi = (\lambda+L)u_t$ and $R_\lambda V^Lu_t\in\D(H)=\D(L)$,
  $$
  q(t)
  := \left\langle (\lambda+L)u_t, R_\lambda V^Lu_t \right\rangle
  = \left\langle u_t, (\lambda+L)R_\lambda V^Lu_t \right\rangle
  $$
  is well-defined.
  Since $L=H+V^L$, we have $(\lambda+L)R_\lambda = 1+V^LR_\lambda$. Hence,
  \begin{align}
    \label{eq:nonexistence-unified-q}
    q(t) = \langle u_t, V^Lu_t \rangle + \langle V^Lu_t, R_\lambda V^Lu_t \rangle.
  \end{align}
  Suppose first that $V^L$ satisfies \eqref{ass1}, i.e., is positive. Since $R_\lambda\geq0$,
  \begin{align}
    \label{eq:nonexistence-unified-q-positive}
    q(t) \geq K_1 \left\langle u_t, \langle x\rangle^{-\rho}u_t \right\rangle.
  \end{align}
  Suppose instead that $V^L$ satisfies \eqref{ass2}, i.e., is negative. In this case, choose $\eta$ (and hence $\lambda$) in~\eqref{eq:nonexistence-unified-positive-resolvent} so large that $\eta>K_2^2/K_1$. Then, $0<R_\lambda\leq\eta^{-1}$. Since $|V^L|^2 \leq K_2^2\langle x\rangle^{-\rho}$,
  \begin{align}
    \label{eq:nonexistence-unified-q-negative}
    \begin{split}
      q(t)
      & \leq -K_1 \left\langle u_t, \langle x\rangle^{-\rho} u_t \right\rangle + \eta^{-1} \langle u_t, |V^L|^2 u_t\rangle \\
      & \leq -\left(K_1-\frac{K_2^2}{\eta} \right)\left\langle u_t, \langle x\rangle^{-\rho} u_t \right\rangle.
    \end{split}
  \end{align}
  Thus, in either case, $q(t)$ has a fixed sign, independent of $t$, and there exists $c_0>0$ such that
  \begin{align}
    \label{eq:nonexistence-unified-q-lower}
    |q(t)| \geq c_0 \left\langle u_t, \langle x\rangle^{-\rho} u_t\right\rangle.
  \end{align}
  We now show that a fixed positive fraction of the mass of $u_t$ stays inside a ballistic cone. By \eqref{eq:nonexistence-unified-spectral-localization} and Lemma~\ref{lem9-small-alpha}, respectively Lemma~\ref{lem9} when $\alpha=1$, we have, for $t\geq1$, that $\|\langle x\rangle u_t\| \lesssim_f t\|\langle x\rangle\phi\|$. Therefore, for every $\Theta>0$,
  \begin{align}
    \label{eq:nonexistence-unified-outer-cone}
    \left\| \one_{|x|\geq\Theta t}u_t \right\|
    \leq \frac{1}{\Theta t} \|\langle x\rangle u_t\|
    \lesssim_f \Theta^{-1}\|\langle x\rangle\phi\|.
  \end{align}
  Choose $\Theta=\Theta_\phi\geq1$ so large that
  $$
  \left\| \one_{|x|\geq\Theta t}u_t \right\|^2 \leq \frac12\|\phi\|^2
  $$
  for every $t\geq1$. Since $\|u_t\|=\|\phi\|$, we then have
  \begin{align}
    \label{eq:nonexistence-unified-inner-mass}
    \left\| \one_{|x|<\Theta t}u_t\right\|^2 \geq \frac12\|\phi\|^2.
  \end{align}
  On $\{|x|<\Theta t\}$, with $t\geq1$ and $\Theta\geq1$, we have $\langle x\rangle \leq \sqrt{1+\Theta^2t^2} \leq \sqrt2\,\Theta t$. Hence,
  $$
  \langle x\rangle^{-\rho} \geq (\sqrt2\,\Theta)^{-\rho}t^{-\rho} \quad \text{on } \{|x|<\Theta t\}.
  $$
  By \eqref{eq:nonexistence-unified-q-lower} and \eqref{eq:nonexistence-unified-inner-mass}, there exists $c_\phi>0$ such that
  \begin{align}
    \label{eq:nonexistence-unified-q-final}
    |q(t)| \geq c_\phi t^{-\rho}, \quad t\geq1.
  \end{align}
  Since $q(t)$ has a fixed sign,
  \begin{align}
    \label{eq:nonexistence-unified-J1-lower}
    |J_1| = \int_{t_1}^{t_2}|q(t)|\,dt \geq c_\phi \int_{t_1}^{t_2}t^{-\rho}\,dt
  \end{align}
  whenever $t_1\geq1$.

  \medskip
  \noindent
  {\bf Step 4: Estimate of the error term $J_2$.}
  Let $\epsilon>0$. Since, by assumption, $\me{i t H}\me{-i t L} \to W$ strongly, there exists $T_\epsilon\geq1$ such that
  \begin{align}
    \label{eq:nonexistence-unified-wave-convergence}
    \left\| \left(W-\me{i t H}\me{-i t L}\right)(\lambda+L)\phi \right\| \leq\epsilon
  \end{align}
  for all $t\geq T_\epsilon$. Thus, if $t_1\geq T_\epsilon$,
  \begin{align}
    \label{eq:nonexistence-unified-J2-first}
    |J_2| \leq \epsilon \int_{t_1}^{t_2}\|R_\lambda V^Lu_t\|\,dt.
  \end{align}
  We now estimate $R_\lambda V^Lu_t$. To this end, let $\theta>0$ be the minimal escape velocity constant from Theorem~\ref{the6} with $V=0$, the function $f$ from \eqref{eq:nonexistence-unified-spectral-localization}, and $\nu=2$ there. Choose $\chi\in C^\infty(\R;[0,1])$ such that
  $$
  \chi(s)=0
  \quad\text{for }s\leq\frac{\theta}{2},
  \qquad
  \chi(s)=1
  \quad\text{for }s\geq\theta.
  $$
  We decompose
  \begin{align}
    \label{eq:nonexistence-unified-low-high-A}
    \|R_\lambda V^Lu_t\|
    \leq \|R_\lambda V^L(1-\chi(A/t))u_t\| + \|R_\lambda V^L\chi(A/t)u_t\|.
  \end{align}
  Since $|1-\chi(s)| \leq \one_{(-\infty,\theta]}(s)$, Theorem~\ref{the6} and \eqref{eq:nonexistence-unified-spectral-localization} yield, for all sufficiently large $t$,
  \begin{align}
    \label{eq:nonexistence-unified-low-A}
    \|R_\lambda V^L(1-\chi(A/t))u_t\|
    \lesssim \left\| \one\left(\frac At\leq\theta\right) f(L)\me{-i t L}\phi \right\|
    \lesssim t^{-2}\|\langle A\rangle^2\phi\|.
  \end{align}
  We next estimate the high-$A/t$ contribution. By Lemma~\ref{lem:commutators-spectral-cutoff},
  \begin{align}
    \label{eq:nonexistence-unified-cutoff-commutators}
    \|[\chi(A/t),V^L]\| + \|[\chi(A/t),R_\lambda]\| \lesssim_{V,\eta} t^{-1}.
  \end{align}
  Therefore,
  \begin{align}
    \label{eq:nonexistence-unified-move-chi}
    R_\lambda V^L\chi(A/t) = \chi(A/t)R_\lambda V^L + E_t,
    \qquad
    \|E_t\|\lesssim_{V,\eta} t^{-1}.
  \end{align}
  We claim that
  \begin{align}
    \label{eq:nonexistence-unified-A-regularity}
    \langle A\rangle^\rho R_\lambda V^Lf(L) \in\mathcal B(L^2).
  \end{align}
  If $0<\rho<1$, this is Corollary~\ref{cor:fractional-A-resolvent-V} since $\rho\leq\alpha$.
  It remains to prove \eqref{eq:nonexistence-unified-A-regularity} for $\rho=1$, i.e., $\alpha=1$ in the present parameter range. To this end, we use $A = -i\left(x\cdot\nabla+\frac d2\right)$ to get
  $$
  AV^L =-i\sum_{j=1}^d x_jV^L\partial_j - i\left(x\cdot\nabla V^L+\frac d2V^L \right).
  $$
  By \eqref{ass5} with $\rho=1$, we have $x_jV^L\in W^{1,\infty}(\R^d)$ and $x\cdot\nabla V^L,\, V^L\in L^\infty(\R^d)$. Hence,
  \begin{align}
    \label{eq:nonexistence-unified-AV-sandwich}
    \langle D\rangle^{-1/2} AV^L\langle D\rangle^{-1/2} \in\mathcal B(L^2),
  \end{align}
  because $\xi_j\langle\xi\rangle^{-1} \in L^\infty(\R^d)$ and multiplication by $W^{1,\infty}$ functions is bounded on $H^{1/2}(\R^d)$. Since $\kappa<\kappa_{\rm c}$, we have that $R_\lambda\langle D\rangle^{1/2}$ and $\langle D\rangle^{1/2}f(L)$ are bounded for $\alpha=1$. Thus,
  \begin{align*}
    R_\lambda AV^Lf(L)
    & = R_\lambda\langle D\rangle^{1/2} \left( \langle D\rangle^{-1/2} AV^L \langle D\rangle^{-1/2}\right) \langle D\rangle^{1/2}f(L) \in \mathcal B(L^2).
  \end{align*}
  Since
  $$
  AR_\lambda V^Lf(L) = R_\lambda AV^Lf(L) + [A,R_\lambda]V^Lf(L),
  $$
  and $[A,R_\lambda]\in\mathcal B(L^2)$ by Lemma~\ref{lem:commutators-spectral-cutoff}, we obtain $AR_\lambda V^Lf(L)\in\mathcal B(L^2)$. Hence, we have shown $\langle A\rangle R_\lambda V^Lf(L)\in\mathcal B(L^2)$, i.e., \eqref{eq:nonexistence-unified-A-regularity} also holds for $\rho=\alpha=1$.

  Since $\chi(A/t)$ is supported where $A/t\geq\theta/2$,
  \begin{align}
    \label{eq:nonexistence-unified-chi-A-weight}
    \|\chi(A/t)\langle A\rangle^{-\rho}\|
    \lesssim_{\theta,\rho} t^{-\rho}, \quad t\geq1.
  \end{align}
  Using $u_t=f(L)u_t$, we obtain from \eqref{eq:nonexistence-unified-A-regularity} and \eqref{eq:nonexistence-unified-chi-A-weight} that
  \begin{align}
    \label{eq:nonexistence-unified-high-A-main}
    \|\chi(A/t)R_\lambda V^Lu_t\| \lesssim t^{-\rho}\|\phi\|.
  \end{align}
  Combining this with \eqref{eq:nonexistence-unified-move-chi}, and using $0<\rho\leq1$, gives
  \begin{align}
    \label{eq:nonexistence-unified-high-A}
    \|R_\lambda V^L\chi(A/t)u_t\| \lesssim_\phi t^{-\rho}.
  \end{align}

  Together with \eqref{eq:nonexistence-unified-low-A}, we conclude
  \begin{align}
    \label{eq:nonexistence-unified-RV-bound}
    \|R_\lambda V^Lu_t\| \lesssim_\phi t^{-2}+t^{-\rho} \quad \text{for $t$ large enough.}
  \end{align}
  Thus,
  \begin{align}
    \label{eq:nonexistence-unified-J2-final}
    |J_2| \leq \epsilon C_\phi \left(1+\int_{t_1}^{t_2}t^{-\rho}\,dt\right),
  \end{align}
  where $C_\phi$ is independent of $\epsilon,t_1,t_2$.

  \medskip
  \noindent
  {\bf Step 5: Contradiction.}
  Combining \eqref{eq:nonexistence-unified-J1-lower} and \eqref{eq:nonexistence-unified-J2-final}, we obtain
  \begin{align}
    \label{eq:nonexistence-unified-final-lower}
    |J_1|-|J_2| \geq \left(c_\phi-\epsilon C_\phi\right) \int_{t_1}^{t_2}t^{-\rho}\,dt - \epsilon C_\phi.
  \end{align}
  Choose $\epsilon>0$ sufficiently small that $c_\phi-\epsilon C_\phi>0$, and then choose $t_1\geq T_\epsilon$ sufficiently large that all preceding propagation estimates apply. Since
$0<\rho\leq1$, we have
  $$
  \int_{t_1}^{t_2}t^{-\rho}\,dt \to \infty \quad \text{as }t_2\to\infty.
  $$
  Hence,
  \begin{align}
    \label{eq:nonexistence-unified-divergence}
    |J_1|-|J_2|\to \infty.
  \end{align}
  On the other hand, by \eqref{eq:nonexistence-unified-J1J2} and the triangle inequality,
  \begin{align}
    \label{eq:nonexistence-unified-upper}
    |J_1|-|J_2|
    \leq |J_1+J_2|
    = \left| \left\langle W\phi, \left(\me{i t_2 H}\me{-i t_2 L} - \me{i t_1 H}\me{-i t_1 L} \right)\phi \right\rangle \right|
    \leq 2\|\phi\|^2
  \end{align}
  since $\me{i t H}\me{-i t L}$ is unitary and $W$ is an isometry. This contradicts \eqref{eq:nonexistence-unified-divergence}. Thus, $W(L_\kappa-V^L,L_\kappa)$ does not exist.
\end{proof}

\appendix


\section{Proof of Lemma~\ref{lem:MY-sufficient}}
\label{s:proof-MY-sufficient}

Set $h(x):=|x|^{-\alpha}$ for $x\in\R^d\setminus\{0\}$ and
$$
W_{\kappa,j}:=(-x\cdot\nabla)^j W_\kappa.
$$
We first record
\begin{align}
  \label{eq:dilation-hardy}
  (-x\cdot\nabla)^j h=\alpha^j h, \quad j=0,1,2,
\end{align}
and
\begin{align}
  \label{eq:Wj}
  W_{\kappa,j} = \alpha^j\kappa h+V^S_j, \quad j=0,1,2.
\end{align}
Hence,
$$
W_\kappa \leq (\kappa+M)h, \qquad
W_{\kappa,1}\leq (\alpha\kappa_++M)h.
$$

\begin{enumerate}
\item By the definition of $M$,
  \begin{align}
    \label{eq:Vj-pointwise}
    |V_j^S(x)| \leq M|x|^{-\alpha} = Mh(x), \quad j=0,1,2,
  \end{align}
  for almost every $x\neq0$. In particular,
  \begin{equation}
    \label{eq:Wj-pointwise}
    |W_{\kappa,j}(x)| \leq (\alpha^j|\kappa|+M)h(x).
  \end{equation}
  Since $h\in L^1_{\rm loc}(\R^d)$, this also shows that $W_{\kappa,j}\in L_{\rm loc}^1(\R^d)$ for $j=0,1,2$. Taking $j=0$ in \eqref{eq:Vj-pointwise} gives $-Mh(x)\leq V^S(x)\leq Mh(x)$, and therefore
  $$
  (\kappa-M)h(x) \leq W_\kappa(x) \leq(\kappa+M)h(x),
  $$
  which proves \eqref{eq:natural-sandwich}.
  
\item We next verify that $W_\kappa$ satisfies Assumption~\ref{mizutaniyaoconditions} when $M<\min\{1,\alpha\}(\kappa_{\rm c}-\kappa_+)$.
  
  \smallskip
  \noindent
  {\em Verification of \eqref{eq:mizutaniyaocond1}.}
  Let $f\in L^2(\R^d)$ and set $g=(L_0+1)^{-1/2}f$. By \eqref{eq:Wj-pointwise}, Hardy's inequality, and the spectral theorem,
  \begin{align*}
    \| |W_{\kappa,j}|^{1/2} (L_0+1)^{-1/2}f\|_2^2
    & = \int_{\R^d}|W_{\kappa,j}(x)|\,|g(x)|^2\,dx
      \leq (\alpha^j|\kappa|+M) \int_{\R^d}\frac{|g(x)|^2}{|x|^\alpha}\,dx \\
    & \leq \frac{\alpha^j|\kappa|+M}{\kappa_{\rm c}} \langle g,L_0g\rangle
      \leq \frac{\alpha^j|\kappa|+M}{\kappa_{\rm c}} \|f\|_2^2.
  \end{align*}
  Thus,
  \begin{align}
    \label{eq:MY1-explicit}
    \| |W_{\kappa,j}|^{1/2} (L_0+1)^{-1/2}\|
    \leq \left( \frac{\alpha^j|\kappa|+M}{\kappa_{\rm c}} \right)^{1/2},
    \quad j=0,1,2,
  \end{align}
  which proves \eqref{eq:mizutaniyaocond1} without using any upper bound on $M$.
  
  \smallskip
  \noindent
  {\em Verification of \eqref{eq:mizutaniyaocond2}.}
  From $W_\kappa(x)\leq(\kappa+M)h(x)$ and Hardy's inequality,
  \begin{align}
    \label{eq:MY2-explicit}
    \begin{split}
      \langle u, (L_0-W_\kappa)u\rangle
      & \geq \langle u, L_0u\rangle - (\kappa+M) \int_{\R^d}h(x)|u(x)|^2\,dx 
      \geq \left(1-\frac{\kappa_++M}{\kappa_{\rm c}}\right) \langle u,L_0u\rangle.
    \end{split}
  \end{align}
  Hence, \eqref{eq:mizutaniyaocond2} holds whenever
  \begin{align}
    \label{eq:smallness-zero}
    M<\kappa_{\rm c} - \kappa_+.
  \end{align}
  
  \smallskip
  \noindent
  {\em Verification of \eqref{eq:mizutaniyaocond3}.}
  By \eqref{eq:Wj}, we have $W_{\kappa,1} = \alpha\kappa h+V^S_1$. Since $V^S_1\leq Mh$,
  \begin{align}
    \label{eq:MY3-explicit}
    \begin{split}
      \langle u,(\alpha L_0-W_{\kappa,1})u\rangle
      & \geq \alpha \langle u,L_0u\rangle - (\alpha\kappa+M) \int_{\R^d}h(x)|u(x)|^2\,dx \\
      & \geq \left( \alpha - \frac{\alpha \kappa_++M}{\kappa_{\rm c}} \right) \langle u,L_0u\rangle \\
      & = \frac{\alpha(\kappa_{\rm c}-\kappa_+)-M}{\kappa_{\rm c}} \langle u, L_0u\rangle.
    \end{split}
  \end{align}
  Thus, \eqref{eq:mizutaniyaocond3} holds whenever
  \begin{align}
    \label{eq:smallness-first}
    M<\alpha(\kappa_{\rm c} - \kappa_+).
  \end{align}
  
  \smallskip
  \noindent
  {\em Verification of \eqref{eq:mizutaniyaocond4}.}
  By \eqref{eq:Wj-pointwise} with $j=2$ and Hardy's inequality,
  \begin{align}
    \label{eq:second-comm-upper}
    \begin{split}
      \left| \langle u, (\alpha^2L_0-W_{\kappa,2})u\rangle \right|
      & \leq \alpha^2 \langle u,L_0u\rangle + \int_{\R^d}|W_{\kappa,2}(x)|\,|u(x)|^2\,dx \\
      & \leq \left[ \alpha^2 + \frac{\alpha^2|\kappa|+M}{\kappa_{\rm c}} \right] \langle u,L_0u\rangle.
    \end{split}
  \end{align}
  On the other hand, provided \eqref{eq:smallness-first} holds, \eqref{eq:MY3-explicit} yields
  \begin{align}
    \label{eq:L0-from-first}
    \langle u,L_0u\rangle
    \leq \frac{\kappa_{\rm c}}{\alpha(\kappa_{\rm c}-\kappa_+)-M} \langle u,(\alpha L_0-W_{\kappa,1})u\rangle.
  \end{align}
  Combining \eqref{eq:second-comm-upper} and \eqref{eq:L0-from-first} yields
  \begin{align}
    \label{eq:MY4-explicit}
    \left| \langle u,(\alpha^2L_0-W_{\kappa,2})u\rangle \right|
    & \leq \frac{\alpha^2\kappa_{\rm c} + \alpha^2|\kappa|+M}{\alpha(\kappa_{\rm c}-\kappa_+)-M} \langle u,(\alpha L_0-W_{\kappa,1})u\rangle.
  \end{align}
  This proves \eqref{eq:mizutaniyaocond4} when $M<\alpha(\kappa_{\rm c}-\kappa_+)$.
  
  We have therefore shown that $W_\kappa$ satisfies Assumption~\ref{mizutaniyaoconditions} whenever
  $$
  M < \min\{ \kappa_{\rm c}-\kappa_+,\, \alpha(\kappa_{\rm c}-\kappa_+) \}
  = \min\{1,\alpha\}(\kappa_{\rm c}-\kappa_+).
  $$
  Thus, Item~(2) is proved.
  
\item Suppose that $\tilde\kappa<\kappa<\kappa'<\kappa_{\rm c}$ are prescribed and that \eqref{eq:MY-smallness-fixed} holds. Since $M<\kappa-\tilde\kappa$ and $M<\kappa'-\kappa$, the bounds \eqref{eq:natural-sandwich} imply $\tilde\kappa/|x|^\alpha \leq W_\kappa(x) \leq \kappa'|x|^{-\alpha}$.
  Moreover, since $\kappa'<\kappa_{\rm c}$,
  $$
  \langle u,(L_0-W_\kappa)u\rangle \geq \left(1-\frac{(\kappa')_+}{\kappa_{\rm c}}\right) \langle u,L_0u\rangle.
  $$
  Thus, \eqref{eq:mizutaniyaocond2} holds. The third inequality in \eqref{eq:MY-smallness-fixed} yields
  $$
  M<\alpha(\kappa_{\rm c}-\kappa_+),
  $$
  and hence \eqref{eq:mizutaniyaocond3} and \eqref{eq:mizutaniyaocond4} follow from the preceding argument.
  This proves Item~(3).

\item We first show the bounds for $A_j$, $j=0,1,2$. Since $\rho_S\geq\alpha$, we have $|x|^\alpha\langle x\rangle^{-\rho_S} \leq1$. Thus, by \eqref{eq:A0},
  \begin{align}
    \label{eq:VS-zero-critical}
    |V^S(x)| \leq A_0|x|^{-\alpha}.
  \end{align}
  Next,
  \begin{align}
    |x|^\alpha |x\cdot\nabla V^S(x)|
    \leq A_1 |x|^{\alpha+1} \langle x\rangle^{-\rho_S-1}
    \leq A_1.
  \end{align}
  Thus,
  \begin{align}
    \label{eq:VS-one-critical}
    |x\cdot\nabla V^S(x)| \leq A_1|x|^{-\alpha}.
  \end{align}
  Finally,
  \begin{align}
    \label{eq:D2-expansion}
    (x\cdot\nabla)^2V^S(x)
    = x\cdot\nabla V^S + \langle x,{\rm Hess}(V^S(x))x\rangle_{\R^d}.
  \end{align}
  Hence, $|(-x\cdot\nabla)^2 V^S(x)| \leq A_1|x|\langle x\rangle^{-\rho_S-1} + A_2|x|^2\langle x\rangle^{-\rho_S-2}$. Since $\rho_S\geq\alpha$, this implies
  \begin{equation}
    \label{eq:VS-two-critical}
    |(-x\cdot\nabla)^2 V^S(x)| \leq (A_1+A_2)|x|^{-\alpha}.
  \end{equation}
  We now prove that $W_\kappa$ is bounded from above and below by Hardy potentials. By \eqref{eq:VS-zero-critical} and \eqref{eq:symbol-smallness-zero},
  $$
  -(\kappa-\tilde\kappa)|x|^{-\alpha} \leq V^S(x) \leq (\kappa'-\kappa)|x|^{-\alpha}.
  $$
  Adding $\kappa|x|^{-\alpha}$ yields $\tilde\kappa/|x|^\alpha \leq W_\kappa(x) \leq \kappa'/|x|^\alpha$, as claimed.

  It remains to verify Assumption~\ref{mizutaniyaoconditions}.

  \smallskip
  \noindent
  \emph{Verification of \eqref{eq:mizutaniyaocond1}.}
  From \eqref{eq:Wj}, \eqref{eq:VS-zero-critical}, \eqref{eq:VS-one-critical}, and \eqref{eq:VS-two-critical}, we have
  \begin{align}
    \label{eq:W0-critical}
    |W_{\kappa,0}(x)| & \leq (|\kappa|+A_0)|x|^{-\alpha}, \\
    \label{eq:W1-critical}
    |W_{\kappa,1}(x)| & \leq (\alpha|\kappa|+A_1)|x|^{-\alpha}, \\
    \label{eq:W2-critical}
    |W_{\kappa,2}(x)| & \leq (\alpha^2|\kappa|+A_1+A_2)|x|^{-\alpha}.
  \end{align}
  In particular, $W_{\kappa,j}\in L^1_{\mathrm{loc}}(\mathbb R^d)$ for $j=0,1,2$.

  Let $f\in L^2(\R^d)$ and set $g=(L_0+1)^{-1/2}f$. Then, $g$ belongs to the quadratic form domain of $L_0$. Thus, if
  $$
  B_0:=|\kappa|+A_0,
  \qquad
  B_1:=\alpha|\kappa|+A_1,
  \qquad
  B_2:=\alpha^2|\kappa|+A_1+A_2,
  $$
  then
  \begin{align*}
    \| |W_{\kappa,j}|^{1/2} (L_0+1)^{-1/2}f \|_2^2
    & \leq B_j \int_{\R^d} \frac{|g(x)|^2}{|x|^\alpha}\,dx 
      \leq \frac{B_j}{\kappa_{\rm c}} \langle g,L_0g\rangle 
      \leq \frac{B_j}{\kappa_{\rm c}}\|f\|_2^2.
  \end{align*}
  This shows \eqref{eq:mizutaniyaocond1}.

  \smallskip
  \noindent
  \emph{Verification of \eqref{eq:mizutaniyaocond2}.}
  By $W_\kappa(x)\leq\kappa'/|x|^\alpha$ and Hardy's inequality,
  \begin{align*}
    \langle u,(L_0-W_\kappa)u\rangle
    & \geq \left(1-\frac{(\kappa')_+}{\kappa_{\rm c}}\right) \langle u,L_0u\rangle
  \end{align*}
  for all $u\in C_c^\infty(\R^d)$. Since $\kappa'<\kappa_{\rm c}$, this proves \eqref{eq:mizutaniyaocond2}.

  \smallskip
  \noindent
  \emph{Verification of \eqref{eq:mizutaniyaocond3}.}
  Recall $W_{\kappa,1} = \alpha\kappa|x|^{-\alpha} -x\cdot\nabla V^S$. Thus, by \eqref{eq:VS-one-critical}, we have $W_{\kappa,1}(x) \leq (\alpha\kappa+A_1)|x|^{-\alpha}$. Hence, by Hardy's inequality,
  \begin{align}
    \label{eq:MY-coercivity-symbol}
    \begin{split}
      \langle u,(\alpha L_0-W_{\kappa,1})u \rangle
      & \geq \frac{\alpha(\kappa_{\rm c}-\kappa_+)-A_1}{\kappa_{\rm c}} \langle u,L_0u\rangle.
    \end{split}
  \end{align}
  By \eqref{eq:symbol-smallness-one}, the coefficient on the right-hand side is strictly positive. This proves~\eqref{eq:mizutaniyaocond3}.

  \smallskip
  \noindent
  \emph{Verification of \eqref{eq:mizutaniyaocond4}.}
  By \eqref{eq:W2-critical} and Hardy's inequality,
  \begin{align}
    \label{eq:MY-second-upper-symbol}
    \begin{split}
      \left| \langle u,(\alpha^2L_0-W_{\kappa,2})u \rangle \right|
      & \leq \left(\alpha^2 + \frac{\alpha^2|\kappa|+A_1+A_2}{\kappa_{\rm c}}\right) \langle u,L_0u\rangle.
    \end{split}
  \end{align}
  On the other hand, \eqref{eq:MY-coercivity-symbol} implies
  \begin{align}
    \label{eq:L0-controlled-symbol}
    \langle u,L_0u\rangle
    \leq \frac{\kappa_{\rm c}}{\alpha(\kappa_{\rm c}-\kappa_+)-A_1} \langle u,(\alpha L_0-W_{\kappa,1})u \rangle.
  \end{align}
  Combining \eqref{eq:MY-second-upper-symbol} and \eqref{eq:L0-controlled-symbol}, we obtain
  \begin{align}
    \left| \langle u,(\alpha^2L_0-W_{\kappa,2})u \rangle \right|
    \leq \frac{\alpha^2\kappa_{\rm c} +\alpha^2|\kappa| + A_1 + A_2}{\alpha(\kappa_{\rm c}-\kappa_+) - A_1} \cdot \langle u,(\alpha L_0-W_{\kappa,1})u\rangle.
  \end{align}
  This is \eqref{eq:mizutaniyaocond4}. Notice that $A_2$ occurs only in the finite constant on the right-hand side, and therefore no smallness assumption on $A_2$ is needed.
  This proves the first part of Item~(4).

  Finally, assume \eqref{eq:symbol-common-constant}. Then $A_0\leq C_S$. Moreover, since every first-order partial derivative is bounded by $C_S\langle x\rangle^{-\rho_S-1}$,
  $$
  |\nabla V^S(x)| \leq \sqrt d\,C_S \langle x\rangle^{-\rho_S-1},
  $$
  and hence $A_1\leq\sqrt d\,C_S$. Similarly, every entry of ${\rm Hess}(V^S(x))$ has absolute value at most $C_S\langle x\rangle^{-\rho_S-2}$. By bounding the $\|\cdot\|_{\R^d,\R^d}$-norm, e.g., by the Frobenius norm, we get
  $$
  \|{\rm Hess}(V^S(x))\|_{\R^d,\R^d} \leq d\,C_S\langle x\rangle^{-\rho_S-2}.
  $$
  Hence, $A_2\leq dC_S<\infty$.
  Thus, \eqref{eq:common-constant-smallness} implies \eqref{eq:symbol-smallness-zero}, \eqref{eq:symbol-smallness-one}, and \eqref{eq:symbol-second-finite}, and the preceding argument applies.
\end{enumerate}
    
\qed


\section{Proof of Lemma~\ref{calderoncommutator}}
\label{s:calderoncommutator}

We distinguish the cases $0<s<1$ and $s=1$.

\medskip
\noindent
\textbf{Case 1: $0<s<1$.}
Let $\beta:=1-s\in(0,1)$, i.e., $0<\beta<d/2$. Then,
\begin{align}
  \label{eq:comm-fourier}
  ([|D|,|x|^s]g)^\wedge(\xi)
  = c_{d,s} \int_{\R^d} \frac{|\xi|-|\eta|}{|\xi-\eta|^{d+s}} \hat g(\eta)\,d\eta
\end{align}
for some $c_{d,s}\neq0$ and with the singularity at $\eta=\xi$ being locally integrable in view of the triangle inequality. Therefore, the integrand is bounded in absolute value by a constant times $|\xi-\eta|^{-d-s+1} = |\xi-\eta|^{-d+\beta}$. Let $f(\eta) := |\eta|^\beta \hat g(\eta)$. Then the right-hand side of \eqref{eq:comm-fourier} is $Tf(\xi)$, where
$$
(Tf)(\xi) = \int_{\R^d}K(\xi,\eta)f(\eta)\,d\eta
$$
with kernel
\begin{align}
  \label{eq:K}
  K(\xi,\eta) = c_{d,s} \frac{|\xi|-|\eta|}{|\xi-\eta|^{d+1-\beta}} |\eta|^{-\beta}.
\end{align}
We prove that $T$ is bounded on $L^2$ by a weighted Schur test with weight $(|\xi|/|\eta|)^{d/2}$. The kernel $K$ is homogeneous of degree $-d$, i.e., $K(\lambda\xi,\lambda\eta)=\lambda^{-d}K(\xi,\eta)$ for all $\lambda>0$, and invariant under simultaneous rotations of $\xi$ and $\eta$. Thus, to verify the first Schur bound it suffices to prove, for a fixed unit vector $e\in\bs^{d-1}$,
\begin{align}
  \label{eq:schur1}
  \int_{\R^d}|K(e,\eta)|\,|\eta|^{-d/2}\,d\eta<\infty.
\end{align}
We split the integral into neighborhoods of $\eta=0$, $\eta=e$, and infinity.

\smallskip
{\em Case (a): Near $\eta=e$.} Here, $|1-|\eta| | \leq |e-\eta|$, and hence
$$
|K(e,\eta)|\,|\eta|^{-d/2} \lesssim |e-\eta|^{-d+\beta},
$$
which is locally integrable because $\beta>0$.

\smallskip
{\em Case (b): Near $\eta=0$.} Here, $|e-\eta|\simeq1$ and $|1-|\eta||\lesssim1$, so
$$
|K(e,\eta)|\,|\eta|^{-d/2} \lesssim |\eta|^{-\beta-d/2}.
$$
This is locally integrable  because $\beta<d/2$.

\smallskip
{\em Case (c): $|\eta|\geq2$.} Here, $|1-|\eta| | \lesssim |\eta|$ and $|e-\eta|\simeq|\eta|$. Therefore,
$$
|K(e,\eta)|\,|\eta|^{-d/2} \lesssim |\eta|^{-3d/2},
$$
which is integrable at infinity. Thus, \eqref{eq:schur1} holds. Similarly, by homogeneity, the second Schur bound reduces to
\begin{align}
  \label{eq:schur2}
  \int_{\R^d}|K(\xi,e)|\,|\xi|^{-d/2}\,d\xi<\infty.
\end{align}

{\em Case (a): Near $\xi=e$.} We again use $| |\xi|-1 | \leq |\xi-e|$ to obtain the locally integrable bound
$$
|K(\xi,e)|\,|\xi|^{-d/2} \lesssim |\xi-e|^{-d+\beta}.
$$

{\em Case (b): Near $\xi=0$.} Here,
$$
|K(\xi,e)|\,|\xi|^{-d/2} \lesssim |\xi|^{-d/2},
$$
which is locally integrable.

\smallskip
{\em Case (c): $|\xi|\geq2$.} Here,
$$
|K(\xi,e)|\,|\xi|^{-d/2} \lesssim |\xi|^{-3d/2+\beta}.
$$
After radial integration, the corresponding integral is bounded by
$$
\int_2^\infty r^{-d/2+\beta-1}\,dr \lesssim1
$$
since $\beta<d/2$. This shows \eqref{eq:schur2}.

\smallskip
The weighted Schur test therefore yields
$$
\|Tf\|_2\lesssim_{d,s}\|f\|_2.
$$
By Plancherel's theorem and the definition of $f$,
$$
\| [|D|,|x|^s]g \|_2
\lesssim_{d,s} \| |\eta|^\beta \hat g(\eta)\|_2
= \||D|^\beta g\|_2
= \||D|^{1-s}g\|_2.
$$
This proves \eqref{eq:commutator-weight} for $0<s<1$.

\medskip
\noindent
\textbf{Case 2: $s=1$.}
The function $a(x)=|x|$ is Lipschitz continuous and obeys $\|\nabla a\|_{L^\infty}\leq1$. Thus, by Calder\'on's commutator theorem \cite{Calderon1965} for the order-one Fourier multiplier $|D|$,
$$
\|[|D|,|x|]g\|_2
\lesssim_d \|\nabla |x|\|_\infty \, \|g\|_2
\lesssim_d \|g\|_2.
$$
This is \eqref{eq:commutator-weight} when $s=1$.
\qed


\section{Proof of (\ref{eq:derivativeheatkernel})}

\label{s:derivativeheatkernel}

\begin{proposition}
  \label{heatkernelestimates-halfplane}
  Let $d\in\N$, $\alpha\in(0,2]\cap(0,d)$, $\tilde\kappa\leq\kappa\leq\kappa_{\rm c}$, and $\sigma=\Phi^{-1}(\kappa)\in(-M,(d-\alpha)/2]$.
  Assume $V:\R^d\to\R$ obeys $\tilde\kappa/|x|^\alpha \leq V(x)\leq\kappa/|x|^\alpha$ for all $x\in\R^d$. Let $H=(-\Delta)^{\alpha/2}-V$.  
  Let $p_z(x,y)$ be the kernel associated to the semigroups $\me{-zH}$ with $z\in \C_{\pi/2}=\{z\in \C: |\arg z|<\pi/2\}$.
  \begin{enumerate}
  \item Let $\alpha\in (0,2)$. Then for any $\epsilon\in (0,1)$, there exists a constant $C=C_\epsilon$ such that
    \begin{align}
      \label{boundedp_t(x,y)-complex-alpha < 2}
      |p_z(x,y)| \leq C \left(1\wedge\frac{|x|}{|z|^{1/\alpha}}\right)^{-\sigma} \left(1\wedge\frac{|y|}{|z|^{1/\alpha}}\right)^{-\sigma} |z|^{-d/\alpha}\left(\frac{|z|^{1/\alpha}}{|z|^{1/\alpha}+|x-y|}\right)^{(d+\alpha)(1-\epsilon)}
    \end{align}
    for all $x,y \in \R^d$ and $z\in \C_{\epsilon\pi/4}$.
        
    \item Let $\alpha=2$. Then there exist constants $C, c>0$ such that
    \begin{align}
      \label{boundedp_t(x,y)-complex-alpha = 2}
      |p_z(x,y)|\leq C  \left(1\wedge\frac{|x|}{\sqrt{|z|}}\right)^{-\sigma} \left(1\wedge\frac{|y|}{\sqrt{|z|}}\right)^{-\sigma}|z|^{-d/2} \exp\left(-\frac{|x-y|^2}{c|z|}\right)
    \end{align}
    for all $x,y \in \Rd$ and $z\in \C_{\pi/4}$.
  \end{enumerate}
\end{proposition}

The proof is strongly inspired by \cite[Proposition~2.4]{BuiMerz2026}.

\begin{proof} 
  We only give the proof for \eqref{boundedp_t(x,y)-complex-alpha < 2} since the proof of  \eqref{boundedp_t(x,y)-complex-alpha = 2} is similar. To that end, we use a Phragm\'en--Lindel\"of argument as in \cite[Section~3.4]{Davies1990}. We set
  \[
    w_z(x)= \left(1+\frac{ {|z|^{1/\alpha}}}{|x|}\right)^{-\sigma}, \quad x\in  \R^d,\ z\in\C
  \]
  and first claim that there exists $C>0$ such that
  \begin{align}\label{boundedp_t(x,y)-complex1'}
    \left|w_z(x) p_z(x,y)w_z(y)\right|\le \frac{C}{|z|^{d/\alpha}}
  \end{align}
  holds for all $x,y\in \R^d$ and $z\in\C_{\pi/4}$.
  To prove this, we define, for $f: \R^d \rightarrow \C$, the norms
  $$
  |f|_{L_{w_{z}}^\infty(\R^d)} = \sup_{x\in \R^d} \left|f(x)w_z(x)\right|
  $$
  and
  \[
    |f|_{L_{w_{z}}^1(\R^d)} = \int_{\Rd}|f(x)|w_z(x)dx.
  \]
  Hence, the inequality  (\ref{boundedp_t(x,y)-complex1'}) is equivalent to
  $$
  \|\me{-zH}\|_{L^1_{w_z^{-1}}(\R^d)\rightarrow L^\infty_{w_z}(\R^d)} \leq \frac{C}{|z|^{d/\alpha}}.
  $$
  For $z\in\C_{\pi/4}$ we can write $z=2t+is$ for some $t\geq 0$ and $s\in \R$ such that $t\sim |z|$. Then, 
  \begin{align*}
    & \|\me{-z H}\|_{L^1_{w_z^{-1}}(\R^d)\rightarrow L^\infty_{w_z}(\R^d)} \\
    & \quad \le \|\me{-tH}\|_{L^2(\R^d)\rightarrow L^\infty_{w_z}(\R^d)} \|\me{-isH}\|_{L^2(\R^d)\rightarrow L^2(\R^d)} \|\me{-tH}\|_{L^1_{{w_z}^{-1}}(\Rd)\rightarrow L^2(\Rd)}.
  \end{align*}
  We first estimate $\|\me{-tH}\|_{L^1_{{w_z}^{-1}}(\Rd)\rightarrow L^2(\Rd)}$. For $f\in L^1_{{w_z}^{-1}}(\Rd)$, we have, by the real-time heat kernel estimates for $\La$, Trotter's formula, and
  $$
  \left(1\wedge\frac{|y|}{t^{1/\alpha}}\right)^{-\sigma}
  \sim \left(1\wedge\f{|y|}{|z|^{1/\alpha}}\right)^{-\sigma}
  \sim w_z(y)^{-1},
  $$
  that
  \begin{align*}
    \footnotesize
    \begin{split}
      \|\me{-tH}f\|_{L^2(\Rd)}
      & \leq \|\me{-t\La}|f|\|_{L^2(\Rd)} \\
      & \lesssim \left[\int_{\R^d} \left|\int_{\R^d} {t^{-\frac d\alpha}  \left( \f{t^{1/\alpha}+|x-y|}{t^{1/\alpha}}\right)^{-d-\alpha} w_z(y)^{-1} |f(y)|}\,dy \right|^2 \left(1\wedge\frac{|x|}{t^{1/\alpha}}\right)^{-2\sigma} \,dx\right]^{\frac12}.
    \end{split}
  \end{align*}
  By Minkowski's inequality,
  \begin{align*}
    \footnotesize
    \begin{split}
    \|\me{-tH}f\|_{L^2(\Rd)}
    & \lesssim \int_{\R^d}\left[\int_{\R^d}\left| t^{-\frac d\alpha} \left(\f{t^{1/\alpha}+|x-y|}{t^{1/\alpha}}\right)^{-d-\alpha}\right|^2 \left(1\wedge\frac{|x|}{t^{1/\alpha}}\right)^{-2\sigma}\,dx\right]^{\frac12} w_z(y)^{-1}|f(y)|\,dy \\
    & \lesssim  {\sup_{y\in \Rd}}\left[\int_{\R^d}\left| t^{-\frac d\alpha}\left(\f{t^{1/\alpha}+|x-y|}{t^{1/\alpha}}\right)^{-d-\alpha}\right|^2 \left(1\wedge\frac{|x|}{t^{1/\alpha}}\right)^{-2\sigma}\,dx\right]^{\frac12} \|f\|_{L^1_{w_z^{-1}}}.
    \end{split}
  \end{align*}
  This, along with the fact that
  \begin{align}
    \label{eq1-norm}
    {\sup_{y\in \Rd}} \left[\int_{\R^d}\left| t^{-d/\alpha} \left(\f{t^{1/\alpha}+|x-y|}{t^{1/\alpha}}\right)^{-d-\alpha}\right|^2 \left(1\wedge\frac{|x|}{t^{1/\alpha}}\right)^{-2\sigma}\, dx\right]^{1/2}
    \lesssim t^{-\f{d}{2\alpha}},
  \end{align}
  which can be seen by distinguishing between $|x|\lessgtr t^{1/\alpha}$ and using $\sigma<d/2$, implies
  \[
    \|\me{-tH}f\|_{L^2(\Rd)} \lesssim t^{-\f{d}{2\alpha}}\|f\|_{L^1_{w_z^{-1}}(\R^d)}.
  \]
  It follows that 
  \[
    \|\me{-tH}\|_{L^1_{{w_z}^{-1}}(\Rd)\rightarrow L^2(\Rd)} \lesssim t^{-\f{d}{2\alpha}}.
  \]
  Similarly,
  \[
    \|\me{-tH}\|_{L^2(\Rd)\to L^\infty_{{w_z}}(\Rd)} \lesssim t^{-\f{d}{2\alpha}}.
  \]
  Moreover, since $H$ is a (nonnegative) selfadjoint operator, $\|\me{-isH}\|_{L^2(\Rd)\rightarrow L^2(\Rd)}\leq 1$. Taking these three estimates into account, we get
  $$
  \|\me{-zH}\|_{L^1_{w_z^{-1}}(\Rd)\rightarrow L^\infty_{w_z}(\Rd)}
  \le t^{-\f{d}{\alpha}}
  \sim \f{1}{|z|^{d/\alpha}},
  $$
  which implies \eqref{boundedp_t(x,y)-complex1'}, i.e.,
  \[
    |p_z(x,y)| \lesssim \f{1}{|z|^{d/\alpha}} \left(1\wedge\f{|x|}{|z|^{1/\alpha}}\right)^{-\sigma}\left(1\wedge\f{|y|}{|z|^{1/\alpha}}\right)^{-\sigma}
  \]
  for all $x,y\in \Rd$ and $z\in\C_{\pi/4}$. By the Phragm\'en--Lindel\"of theorem in \cite[Proposition 3.3]{DuongRobinson1996} or \cite[Theorem~2.1]{Merz2022}, we get \eqref{boundedp_t(x,y)-complex-alpha < 2}.
\end{proof}

\begin{proposition}
  \label{thm-ptk}
  Let $d\in\N$, $\alpha\in(0,2]\cap(0,d)$, $\tilde\kappa\leq\kappa\leq\kappa_{\rm c}$, and $\sigma=\Phi^{-1}(\kappa)\in(-M,(d-\alpha)/2]$.
  Assume $V:\R^d\to\R$ obeys $\tilde\kappa/|x|^\alpha \leq V(x)\leq\kappa/|x|^\alpha$ for all $x\in\R^d$. Let $H=(-\Delta)^{\alpha/2}-V$.
  For each $k\in \N$, we denote by $p_{t,k}(x,y)$ the kernel of $H^k \me{-tH}$.
  \begin{enumerate}
  \item Let $\alpha\in (0,2)$. Then for any $\epsilon>0$ and $k\in \N$, there exists a constant $C=C(k,\epsilon)$ such that
    \begin{align}
      \label{boundedp_t k(x,y)-complex-alpha < 2}
      |p_{t,k}(x,y)|\le C  \left(1\wedge\f{|x|}{t^{1/\alpha}}\right)^{-\sigma} \left(1\wedge\f{|y|}{t^{1/\alpha}}\right)^{-\sigma} t^{-(k+d/\alpha)} \left(\f{t^{1/\alpha}+|x-y|}{t^{1/\alpha}}\right)^{-d-\alpha+\epsilon}
    \end{align}
    for all $x,y \in \Rd$ and $t>0$.
    
  \item Let $\alpha=2$. Then for each $k\in \N$ there exist constants $C, c>0$ such that
    \begin{align}
      \label{boundedp_t k(x,y)-complex-alpha = 2}
      |p_{t,k}(x,y)| \le C \left(1\wedge\f{|x|}{\sqrt t}\right)^{-\sigma} \left(1\wedge\f{|y|}{\sqrt t}\right)^{-\sigma} t^{-(k+d/2)}\exp\left(-\f{|x-y|^2}{ct}\right)
    \end{align}
    for all $x,y \in \Rd$ and $t>0$.
  \end{enumerate}		
\end{proposition}

\begin{proof}
  Applying Cauchy's formula, for every $t>0$ and $k\in \N$,
  $$
  H^k \me{-tH} = \f{(-1)^kk!}{2\pi i}\int_{|\xi-t|=\eta t} \me{-\xi H}\f{d\xi}{(\xi-t)^{k+1}},
  $$
  where $\eta>0$ is small enough so that $\{\xi: |\xi-t|=\eta t\}\subset \C_{\tilde\epsilon\pi/4}$ with $\tilde \epsilon = \f{\epsilon}{d+\alpha}$, and the integral does not depend on the choice of $\eta$. Thus, by Proposition \ref{heatkernelestimates-halfplane} and $|\xi|\simeq |\xi-t|\simeq t$,
  \begin{align*}
    |p_{t,k}(x,y)|
    & \leq C_{k,\tilde\epsilon} \left(1\wedge\f{|x|}{t^{1/\alpha}}\right)^{-\sigma} \left(1\wedge\f{|y|}{t^{1/\alpha}}\right)^{-\sigma}t^{-(k+d/\alpha)} \left(\f{t^{1/\alpha}+|x-y|}{t^{1/\alpha}}\right)^{-(d+\alpha)(1-\tilde\epsilon)} \\
    & \leq C_{k,\epsilon} \left(1\wedge\f{|x|}{t^{1/\alpha}}\right)^{-\sigma} \left(1\wedge\f{|y|}{t^{1/\alpha}}\right)^{-\sigma} t^{-(k+d/\alpha)} \left(\f{t^{1/\alpha}+|x-y|}{t^{1/\alpha}}\right)^{-(d+\alpha-\epsilon)}
  \end{align*}
  holds for all $x,y\in \Rd$ and $t>0$.
\end{proof}


\section*{Acknowledgments}

A.~I.~was supported by JSPS KAKENHI Grant Numbers JP21K03279, JP26K06834.
K.~M.~thanks Katsushika Division, Institute of Arts and Sciences of Tokyo University of Science for their hospitality during a stay in April and May 2026, when part of this work was carried out.
We are grateful to Gian Michele Graf for fruitful discussions.
The authors acknowledge the use of AI tools for the final proofreading of the manuscript. All mathematical arguments and proofs in the final manuscript were checked and written by the authors.

\newcommand{\etalchar}[1]{$^{#1}$}
\def\cprime{$'$}

\end{document}